\documentclass[11pt, reqno]{amsart}
\usepackage{amsmath,amsthm,amssymb,amsfonts,color}

\usepackage[latin1]{inputenc}

\newcommand\cF{\mathcal{F}}

\newcommand\bR{\mathbb{R}}

\newcommand\bQ{\mathbb{Q}}

\newcommand\cP{\mathcal{P}}

\newcommand\cR{\mathcal{R}}

\newcommand\cM{\mathcal{M}}

\newcommand\cI{\mathcal{I}}

\newcommand\cJ{\mathcal{J}}

\newcommand\bN{\mathbb{N}}

\newcommand\bZ{\mathbb{Z}}

\newcommand\bH{\mathbb{H}}

\newcommand\cG{\mathcal{G}}

\newcommand\cH{\mathcal{H}}

\newcommand\fH{\mathfrak{H}}

\newcommand\fR{\mathfrak{R}}

 \theoremstyle{definition}

\newtheorem{theorem}{Theorem}[section]
\newtheorem{lemma}[theorem]{Lemma}
\newtheorem{corollary}[theorem]{Corollary}

\newtheorem{definition}{Definition}[section]

\theoremstyle{remark}
\newtheorem{remark}[theorem]{Remark}

\newtheorem{assumption}[theorem]{Assumption}

\newcommand{\mysection}[1]{\section{#1}
 \setcounter{equation}{0}}

\begin{document}

	\title[On the $l_p$ stability estimates]{On the $l_p$ stability estimates for stochastic and deterministic difference equations and their application to SPDEs and PDEs} 
  
	\author{Timur Yastrzhembskiy}
\email{yastr002@umn.edu}
\address{127 Vincent Hall, University of Minnesota,
 Minneapolis, MN, 55455}


  \begin{abstract}
	In this paper we develop the $l_p$-theory of space-time stochastic difference equations which can be considered as a discrete counterpart of N.V. Krylov's $L_p$-theory of stochastic partial differential equations (SPDEs).
	We also prove a Calderon-Zygmund type  estimate for deterministic parabolic finite difference schemes with variable coefficients under relaxed assumptions on the coefficients, the initial data and the forcing term.
   \end{abstract}

 \maketitle

	\mysection{Introduction}
						\label{section 1}
		
		Let 
		$(\Omega, \mathcal{F}, P)$
		 be a complete probability space, 
		and  let 
		$(\mathcal{F}_t, t \geq 0)$
		 be an increasing filtration of $\sigma$-fields 
		$\mathcal{F}_t \subset \mathcal{F}$
		 containing all $P$-null sets of $\Omega$.
		By $\mathcal{P}$ 
		we denote the predictable $\sigma$-field
		 generated by 
		 $(\mathcal{F}_t, t \geq 0)$.
		Let $d_0$ be a positive integer,
		 and 
		$w^i (t), t \in \bR, i = 1, \ldots, d_0$
		 be a sequence of independent standard $\cF_t$-adapted Wiener processes.

		Let $\bR^d$ be a $d$-dimensional Euclidean space of points
	$
		x = (x^1, \ldots, x^d)
	$
	with the standard basis 
	 $
		e_i, i = 1, \ldots, d.
	$
	  By $\bZ^d$ we denote the subset of $\bR^d$ of all points with integer coordinates, 
		and we  set $\bZ_{+} = \{0, 1, 2, \ldots\}$, $\bN = \{1, 2, \ldots\}$.
	For 
	$\gamma, h \in (0, 1]$
	we denote
	$\gamma \bZ = \{\gamma n, n \in \bZ\}$, 
	$\gamma \bZ_{+} = \{\gamma n, n \in \bZ_{+}\}$, 
	$ h \bZ^d = \{ h y, y \in \bZ^d\}$.
		For a function on $h \bZ^d$, we denote
	  $$ 
			\delta_{ \xi, h}  f (x) = (f (x + h \xi) - f (x))/h,  \quad
		\Delta_{\xi, h} f (x) = (f (x + h \xi) - 2 f(x) + f (x -h \xi))/h^2.
	   $$
	 	In this paper we consider the following stochastic  difference equation:
		\begin{equation}
					\label{3.1}
		\begin{aligned}
			 & v_{\gamma, h}  (t+\gamma, x) -  v_{\gamma, h}  (t, x)
												  = \gamma \sum_{k = 1}^{d_1} \lambda_{k} (t, x) \Delta_{l_k, h} v_{\gamma, h} (t, x)\\
						&	 + \gamma \sum_{i = 1}^d  b^i (t, x) \delta_{e_i, h} v_{\gamma, h} (t, x)   +  \gamma f (t, x) \\
			&  + \sum_{ l = 1}^{d_0}    (g^l (t, x) + \nu^l (t, x) v_{\gamma, h} (t, x)) (w^l (t+\gamma) - w^l (t)), \\
		&  v_{\gamma, h} (0, x) = u_0 (x), \, t \in \gamma \bZ_{+}, x \in h \bZ^d.
		\end{aligned}
		\end{equation}
		Here $d_1 \geq d$, and $\lambda_k (t, x) \in [\delta_{1}^{*}, \delta_2^{*}]$, where $\delta_1^{*}, \delta_2^{*}$ are positive numbers.
			We also consider a  deterministic equation given by 	
		\begin{equation}
					\label{4.1}
		\begin{aligned}
			 & u_{\gamma, h}  (t+\gamma, x) -  u_{\gamma, h}  (t, x)
												  = \gamma \sum_{k = 1}^{d_1} \lambda_{k} (t, x) \Delta_{l_k, h} u_{\gamma, h} (t, x)\\
																												&	 + \gamma \sum_{i = 1}^d b^i (t, x) \delta_{e_i, h} u_{\gamma, h} (t, x)
			+  \gamma f (t, x),\\
		 &  u_{\gamma, h} (0, x) = u_0 (x), \,  t \in \gamma \bZ_{+}, x \in h \bZ^d.
		\end{aligned}
		\end{equation}

		Our goal is to prove stability estimates 
		for $v_{\gamma, h}$ and $u_{\gamma, h}$ 
		 in  discrete $l_p$ Sobolev norms.
		We do it  under relatively mild assumptions 
		 on the coefficients, the free terms and the initial data (see Theorem \ref{theorem 3.2} and Theorem \ref{theorem 4.1}).
		Theorem \ref{theorem 3.2}  can be viewed as a discrete counterpart of  N.V. Krylov's $L_p$-theory of SPDEs (see Theorem 5.1 of  \cite{Kr_08}).
		In the deterministic case one can consider the result of Theorem \ref{theorem 4.1}
		 as a discrete version of the $W^2_p (\bR^d)$
		 estimate for nondivergence parabolic equations (see, for example, Theorem 5.2.10 of \cite{Kr_08}).
		One of the virtues of  such $l_p$ estimates  is that they can be used
		 to estimate the rate of convergence of $v_{\gamma, h}$ in discrete Sobolev spaces
		 to the  solution of the corresponding  SPDE.
		In Theorem \ref{theorem 3.1}
		we will show that this rate of convergence 
		is of order
		$
			h^{1 - \varepsilon}
		$
		with any $\varepsilon > 0$.
		We also estimate the rate of convergence in  
		  the uniform norm on the lattice
		 $
			\gamma \bZ_{+}  \times h\bZ^d.
		$
		In the deterministic case we prove that the rate of convergence is of order $h^2$
		which is optimal.

		The reason why we consider  finite difference schemes of type \eqref{3.1} is twofold.
		First, it is well-known that
		 if $a (t, x)$ 
			is a bounded symmetric matrix-valued function 
		with eigenvalues bounded away from $0$ for all $(t, x)$, 
		then there exist vectors
		 $
			l_1, \ldots, l_{d_1} \in \bZ^d
		$
		 and functions 
		$
			\lambda_1, \ldots, \lambda_{d_1}
		$ 
		such that
		$$
			a (t, x) = \sum_{k = 1}^{d_1} \lambda_k (t, x) l_k l_k^{*}.
		$$
		In addition, all the functions 
		$\lambda_k (t, x)$
		 are as regular as $a (t, x)$.
		Interestingly, one can choose such vectors 
		$
			l_k, k =  1, \ldots, d_1
		$
		uniformly for all functions that take values in the set of symmetric matrices
		 with eigenvalues in a fixed interval, 
		say 
		$
			[\delta_1, \delta_2]
		$
		 (see Remark \ref{remark 3.1}). 
		Another virtue  of a second order difference operator 
		$
			\gamma \sum_{k  =1}^{d_1} \lambda_k (t, x) \Delta_{l_k, h}
		$
		is that, for $\gamma/h^2$ small enough, it can be viewed as a generator
		of a random walk with non identically distributed increments. 
		This fact implies the maximal principle and some useful pointwise bounds of a discrete fundamental solution.

		The $l_p$ approach to  the space-time finite difference schemes for deterministic parabolic and elliptic equations is well developed.
		The interior $l_2$ estimates
		 for elliptic type  finite difference schemes  
		  are established in \cite{TB_68},  and for  parabolic schemes in   \cite{B_71} (see also \cite{St_91}).
		 For $p > 1$ the interior $l_p$ estimate in the parabolic case is proved in \cite{B_73}
		and in the elliptic case in \cite{S_73, D_12}.
		In \cite{B_73, S_73} the results are proved for  equations with variable infinitely smooth coefficients,
		whereas in \cite{D_12} the stability estimate is proposed only for finite difference operators
		of type
		 $\sum_{i  = 1}^d a_i (x) \delta_{e_i, h}$.
		It is also worth mentioning that the results of \cite{B_73} cover  finite difference equations of order $2m, m \in \bN$ and also  both explicit and implicit schemes.
		Our main theorem in the deterministic case (see Theorem \ref{theorem 4.2})
		 differs from the corresponding result of \cite{B_73} in the fact that we 
		do not require the coefficients  of \eqref{4.1} to be infinitely smooth.
		In fact, they are just bounded measurable in temporal variable 
		and H\"older continuous in the spatial ones.
		In addition, 
		in \cite{B_73} the $l_p$ estimate of the second order differences of the solution of the difference scheme  (see Lemma 6.1 of \cite{B_73})
		 is proved by using a version of Mikhlin multiplier theorem (see Theorem 2.1 of \cite{B_73}).
		In contrast, we use a different argument developed by N.V. Krylov in \cite{Kr_06}.

		There is a number of articles dedicated to stochastic finite difference schemes.
		The $l_2$-theory of stochastic explicit  finite difference schemes  with uniformly nondegenerate leading coefficients
		 has been developed in  \cite{Y_98}. 
		In the same work the author has obtained an estimate of the rate of convergence  of the difference scheme in Hilbert-Sobolev spaces.
		In \cite{GM_09} I. Gy\"ongy and A. Millet have proved a  general result
		   that implies 
		 an estimate of the rate of convergence for both explicit and implicit difference schemes in Hilbert-Sobolev space (see also \cite{H_12}).
		In addition, their result covers the case of quasilinear SPDEs. 
		Further, there has been significant progress 
		recently in the studies of  semi-discrete finite difference schemes in the $l_p$ framework (see \cite{Y_98, GK_11, G_14, GG_15} and references therein).
		The existing results now cover the case when 
		the leading coefficients are possibly degenerate \cite{GG_15}.
		In addition, in  \cite{GK_11, G_14, GG_15}
		 it has been shown how to accelerate semi-discrete numerical schemes.
		 However, the existing articles on the $l_p$-theory
		of semi-discrete equations
		do not address the question of maximal regularity of $l_p$ type.
		Having this result available, one could  obtain the results of the present paper.
		In addition, using this result, under certain conditions, 
		 one can improve the rate of convergence in Theorem \ref{theorem 3.1} (see also Remark \ref{remark 3.3}).
		To the best of this author's knowledge, the present work
		is the first paper where  space-time discrete finite difference schemes of type \eqref{3.1} with nondegenerate leading coefficients
		  are considered in the framework of $l_p$ discrete Sobolev spaces.

		The advantage of the $l_p$-theory over  the $l_2$-theory in the deterministic setting
		 can be demonstrated on the following example.
		Let $n \in \bN$ be a number 
		such that 
		$
			n \geq \lfloor d/2 \rfloor .
		$
		 Consider the following PDE:
		$$
			D_t u (t, x)  =   \sum_{i = 1}^d (a_i (x) D_{i i} u (t, x) + b^i (x) D_i u (t, x)) +  f (t, x), \, u (0, x) = 0, x \in \bR^d.
		$$
		 Assume that $f$ and $g$ are infinitely smooth and compactly supported
		and let $u$ be the unique weak solution of this equation.
		Let $u_{\gamma, h}$ be the solution  of Eq. \eqref{4.1}
		with $d_1 = d, \lambda_k \equiv a_k$, 
		$
		k = 1, \ldots, d.
		$
		In addition, assume that all  the functions 
		$
			\lambda_k (t, x), b^i (t, x), (t, x) \in \bR^{d+1}
		$
		 are continuous with bounded derivatives up to order $n$.
		For a multi-index $\alpha$ by $|\alpha|$ we denote the order of $\alpha$
		and 
		$$
			\delta^{\alpha}_h =  \delta_{e_1, h}^{\alpha_1} \ldots \delta_{e_d, h}^{\alpha_d}. 
		$$	
		Then, by the $l_2$-theory  (see, for example, \cite{Y_98}), for any $T > 0$,
		$$
			\sum_{ |\alpha| \leq n  - \lfloor d/2 \rfloor } \sup_{t = 1, \ldots,  \lfloor T/\gamma \rfloor} \sup_{x \in h \bZ^d} |\delta^{\alpha}_h (u_{\gamma, h} (t, x) - u (t, x))|  = O (h^2).
		$$ 
		In contrast, by taking $p > d$ and applying   Theorem \ref{theorem 4.2}, we obtain
		$$
			\sum_{ |\alpha| \leq n } \sup_{t = 1, \ldots,  \lfloor T/\gamma \rfloor} \sup_{x \in h \bZ^d} |\delta^{\alpha}_h (u_{\gamma, h} (t, x) - u (t, x))|   = O (h^{2}).
		$$
		The major drawback of the stochastic $l_p$-theory lies in the fact that		
		it does not cover the equations containing the  term 
		$
			\sigma^{i l} (t, x) \delta_{e_i, h} v_{\gamma, h} \, (w^l (t + \gamma) - w^l (t)).
		$
		In the case of SPDEs the term $\sigma^{i l} (t, x) D_{i} v \, dw^l (t)$
		is handled by the It\^o-Wentzell formula
		which is not available in the discrete case.

		The key ingredient of the proof of the main result  in the deterministic case
		 is a  discrete counterpart of the parabolic Calderon-Zygmund type estimate (see Theorem \ref{theorem 4.3})
		which is similar to the one given in Theorem 4.3.7 of \cite{Kr_08}.
		 The stochastic $l_p$-theory is based on the discrete version of the  parabolic  Littlewood-Paley inequality (see Theorem \ref{theorem 3.3}).
		The original parabolic Littlewood-Paley has been proved by N.V. Krylov in \cite{Kr_06}.
		To prove Theorem \ref{theorem 3.3} and Theorem \ref{theorem 4.3} we use the method of \cite{Kr_06}
		involving an estimation of the sharp function
		of the second order derivatives of the solution $v_{\gamma, h}$.
		The main problem in the discrete case  is that the discrete heat kernel 
		corresponding to 
		$
			\sum_{k = 1}^{d_1} \lambda_k (t) \Delta_{l_k, h}
		$
		does not have a simple explicit representation.
		Moreover, it lacks certain symmetries
		of a `continuous' heat kernel such as 
		the scaling property and the spherical symmetry. 
		To overcome this issue we establish a local central limit theorem for a discrete heat kernel
		 (see Lemma \ref{lemma 11.2})
		and  use it to derive some pointwise difference estimates of a discrete heat kernel (see Corollary \ref{corollary 11.3})
		 which are similar to the ones of Theorem 2.3.6 of \cite{LL}.
		Our method in the deterministic case is somewhat similar to the one used in \cite{D_12} and in Chapter 4 of \cite{Kr_08}.
		However, instead of exploiting qualitative properties of finite difference equations, 
		we  prefer to use pointwise estimates of a discrete heat kernel.
		It would be interesting to prove Theorem \ref{theorem 4.1} in the spirit of \cite{D_12, Kr_08}
		because this method  enables one to derive the mixed norm estimates at no cost.
		Finally, the techniques of the present paper can be applied to stochastic semi-discrete equations.
		This will be done somewhere else.
		
		This author would like to thank his advisor N.V. Krylov for stimulating discussions, valuable suggestions and attention to this work.

	\mysection{Framework and some notations}
					\label{section 2}

	For two vectors $\xi, \zeta \in \bR^d$
	we denote their scalar product
	by $\xi \cdot \zeta$.
	For $x \in \bR^d$ denote
	$$
		|x| = (\sum_{i = 1}^d |x^i|^2)^{1/2}, 
		\quad ||x|| =  \max_{i =  1, \ldots, d} |x^i|,
	$$
	$$
		D_t = \frac{\partial}{\partial t}, \quad
				 D_i = \frac{\partial}{\partial x^i}, 
							\quad D_{i j}  = \frac{\partial}{\partial x^i x^j}. 
	$$
	 We say that a vector
	$
		\alpha  = (\alpha^1, \ldots, \alpha^d) \in \bZ_{+}^d
	$
	 is a multi-index.
	By $|\alpha| = :\sum_{i = 1}^d \alpha_i$
	we denote the order of this multi-index.
	For two multi-indexes 
	$
		\alpha_1, \alpha_2,
	$
	we write 
	$
	  \alpha_1 \leq \alpha_2
	$
	if 
	$
	  \alpha_1^{i} \leq  \alpha_2^i
	$, 
	  $i = 1, \ldots, d$.
	For any multi-index $\alpha$, 
	denote
	$$
		D^{\alpha}_x = D^{\alpha}_1  \ldots D^{\alpha_d}_d.
	$$
	For any vector $\xi \in \bR^d$,
	 denote
	$$
		D_{(\xi)}  = \sum_{i = 1}^d \xi^i D_i, 
			\quad D_{(\xi) (\xi)}  = \sum_{i, j = 1}^d \xi^i \xi^j D_{i j}. 
	$$

	Fix some number $T > 0$. 
	Throughout this article
	$h \in (0, 1]$ is a number, 
	and
 	$\gamma \in (0, 1]$ is a number such that
	$
		T/\gamma \in \bN.
	$
	Denote
	  $
		t_n = n\gamma, 
	 $
	and recall that 
	$$
		 \gamma \bZ = \{t_n, n \in \bZ\}, \quad
	 \quad
												 h\mathbb{Z}^d = \{x: x =  h y, y  \in   \bZ^d\}.
	$$
	Also recall that, for any function  
	$
		f 
	$
	on 
	$
		h\mathbb{Z}^d
	$
	and any vector $\xi \in \bZ^d$,
	$$
		\delta_{ \xi, h}  f (x) = (f (x + h \xi) - f (x)/h,
		 \quad \Delta_{\xi, h} f (x) = (f (x + h \xi) - 2 f(x) + f (x -h \xi))/h^2.
	$$
	  For any function  $g$ on 
	$
		\gamma \bZ \times h \bZ^d,
	$
	  any vector 
	$\xi \in h \bZ^d$,
	and any multi-index 
	$
		\alpha,
	$
	 denote
	$$
		T_{\xi, h} f (x) = f (x+h \xi), \quad T^{\alpha}_h  = T^{\alpha_1}_{e_1, h} \, \ldots T^{\alpha_d}_{e_d, h},
	$$
	  $$
			\delta^{\alpha}_h = \delta^{\alpha_1}_{e_1, h}  \ldots  \delta^{\alpha_d}_{e_d, h}.
	 $$
		For any two functions $f$ and $g$ defined on $h \bZ^d$
	we denote
	$$
		f \ast_h g (x) = \sum_{y \in  h\bZ^d} f (x-y) g(y).
	$$
	For $p \geq 1$,
	we denote by $l_p ( h \bZ^d)$
	the Banach space of all functions $f$ on $h \bZ^d$
	such that the norm 
	$$
		|| f ||_{ l_p (h \bZ^d) }  = ||f ||_{p; \, h} = ( h^d \sum_{x \in h \bZ^d} |f (x)|^p)^{1/p}
	$$
	is finite.
	For 
	$s \in \bZ_{+}$,
	 and a function $f$ on 
	$ h \bZ^d$ 
	 the Sobolev norm
	 $
		|| \cdot ||_{s, p; \, h}
	$
	 is defined as follows:
	$$
			  [ f ]_{s, p; \,  h} = \sum_{  |\alpha| = s} || \delta^{\alpha}_h f ||_{p; \, h},    
	$$
	$$
						|| f ||_{s, p; \, h} = \sum_{  n = 0}^s [f]_{n, p; \,  h}.
	$$
	For $s, m \in \bZ_{+}$,
	 and a function $f$  on 
	$
		\gamma \bZ  \times h \bZ^d
	$,
	we write  
	$
		f \in \bH^s_{p; \, \gamma, h} (m)
	$
	 if, for any 
	$x \in h \bZ^d$, 
	$
		k \in \{0, 1, \ldots, m\},
	$
	the random variable
	$
		f (t_k, x)$ is $\cF_{t_k}
	$-measurable,
	and
	$$
		|| f ||^p_{\bH^s_{p; \gamma, \,  h} (m)} :=   \gamma E  \sum_{k = 0}^{m} || f (t_k, \cdot)||^p_{s, p; \,  h} < \infty.
	$$

	Next, let
		$
			B = B (\bR^d)
		$
		be the space of bounded Borel functions, 
		$
		  C^k  = C^k (\bR^d), k \in \bN
		$ 
		be the space of  bounded $k$ times differentiable functions with bounded derivatives up to order $k$,
		$
		  C^{\infty}_0 =  C^{\infty}_0 (\bR^d)
		$
		 be the space of infinitely differentiable functions with compact support.
		For $s \in (0, \infty) \setminus \bN$,
		 we denote by 
		$
			 C^{s} (\bR^d)
		$
		    the usual H\"older space,
		and, for 
		$k \in \bZ$,
		by 
		$C^{k, 1}$ we denote the Banach space of 
		 $k$-times continuously differentiable functions
		such that the derivatives of order $k$ are Lipschitz continuous.
		For 
		$
		   s \in \bN
		$,
		 $
			p \in (1, \infty),
		$
			 by
		 $W^s_p$ we denote
		the Sobolev space, 
		i.e. the space of $L_p = L_p (\bR^d)$ functions
		such that all the generalized derivatives up to order $s$ belong to $L_p$.
		 For $s \in \bR$, we define the space of Bessel potentials as follows:
		 $$
			H^{s}_p  : = (1 - \Delta)^{- s/2} L_p.
		$$
		Here $\Delta$ is the Laplacian, and $(1 - \Delta)^{-s/2}$ is understood as a Fourier multiplier. 
		The norm in these spaces is defined as follows:
		$$
			|| f ||_{ H^s_p }  = || (1 -\Delta)^{s/2} f ||_{L_p}. 
		$$
		A detailed discussion of the spaces of Bessel potentials can be found, for example, in \cite{T} and \cite{Kr_08}.
		
		Next, for $T_2 > T_1 \geq 0$, denote
		  $$
			\mathbb{H}^{s}_p (T_1, T_2) := L_p ( \Omega \times [T_1, T_2], \mathcal{P}, H^{s}_p ).
		$$
		We define stochastic Banach spaces
		 $
			\cH^{s}_p (T_1, T_2).
		$
		\begin{definition}
					\label{definition 2.1}
	   	For any 
		$
		p \geq 2, s \in \bR,
		$
		  and numbers $T_2 > T_1 \geq 0$,
		 we write 
		 $
		   u \in \mathcal{H}^{s}_p (T_1, T_2)
		$ 
		if
		\begin{enumerate} 
		  \item $u$ is a distribution-valued process,
			 and 	
		$
			u \in  \mathbb{H}^{s}_p (T_1, T_2)
		$,  
		 \item 	$
					u(0, \cdot) \in L_p (\Omega, \mathcal{F}_0, H^{ s - 2/p }_p)
				$,
		\item there exist
		 $
			f \in \mathbb{H}^{s - 2}_p (T_1, T_2)
		$
		 and 
		 $
		   h^k \in 
	 	    						\mathbb{H}^{s - 1}_{p} (T_1, T_2),
		$ 
		$k=  1, \ldots, d_0$,
		such that, for any 
		$
		   \phi \in C^{\infty}_0,
		$
	 	and  any 
		$
		    t \in [T_1, T_2], \omega \in \Omega,
		$
		\begin{equation}
				\label{2.3}
			\begin{aligned}
			(u (t, \cdot), \phi) = &( u(0, \cdot), \phi) + 
												\int_0^{t}  (f(s, \cdot), \phi) \, ds\\
																			 & + 	\sum_{k = 1}^{d_0} \int_0^{t} (h^k (s, \cdot), \phi) \, dw^k (s).
			\end{aligned}
		\end{equation}
		\end{enumerate}
		The $\cH^s_p (T_1, T_2)$ norm is defined in the following way:
		$$
		    || u ||_{ \cH^{s}_p (T_1, T_2) } :=  || u ||_{\mathbb{H}^{s }_p (T_1, T_2)} 
		$$
		 $$
		       + || f ||_{\mathbb{H}^{s - 2}_p (T_1, T_2) }   +  \sum_{k  = 1}^{d_0} || h^k ||_{\mathbb{H}^{s - 1}_p (T_1, T_2) } +
																				  (E   || u (0, \cdot) ||^p_{  s - 2/p, p} )^{1/p}.
		 $$

		For 
		$
			u \in \cH^{s}_p (T_1, T_2)
		$,
		 we denote 
		 $
			\mathbb{D} u := f
		$,
		 $
			\mathbb{S} u := (h^k, k = 1, \ldots, d_0)
		 $.

		For $p > 1$ we write  
		$
		   u \in \fH^s_p (T_1, T_2)
		$ 
		if $u_0, f$ are independent of $\omega$, and
		$
		  h^k \equiv 0, k = 1, \ldots, d_0.
		$
		\end{definition}
		
			\begin{remark}
						\label{remark 2.1}
			Stochastic Banach spaces have been introduced in \cite{Kr_99}. 
			By Theorem 3.7 of \cite{Kr_99}
			 Definition \ref{definition 2.1}
			is equivalent to  Definition 3.1 of \cite{Kr_99}.
			By the same theorem 
			$
			\cH^s_p (T_1, T_2)
			$ is, indeed, a
			Banach space.
			In addition, by Theorem 7.2 of \cite{Kr_99},
			 for any numbers 
			$
				\mu, \theta
			$
			such that
			$
				1/2 > \mu > \theta > 1/p,
			$
			we have 
			$$
				E || u ||^p_{  C^{\theta - 1/p} ([T_1, T_2], H^{s - 2 \mu}_p) } \leq N (d, p, T_1, T_2, \mu, \theta) || u ||^p_{ \cH^s_p (T_1, T_2) }.
			$$
			Further, by the embedding theorem for $H^s_p$ spaces (see for example, Theorem 13.8.1 of  \cite{Kr_08}),
			if 
			$
				s - 2 \mu  -  d/p > 0,
			$ 
			then, 
			we may assume that,
			 for any $\omega$,
			$$
				u \in C^{\theta - 1/p} ([T_1, T_2], C^{  s - 2 \mu - d/p  }).
			$$
			\end{remark}
		
			\begin{remark}
						\label{remark 2.2}
			Let $p > 1$
			and
			$
				u \in \fH^s_p (T_1, T_2).
			$
			 Then,
			by Theorem 7.3 of \cite{Kr_01},
			for any number $\beta \in  (1/p, 1]$,
			we have
			$$
				 || u ||_{  C^{\beta - 1/p} ([T_1, T_2], H^{s - 2 \beta}_p) } \leq N (d, p, T_1, T_2, \beta) || u ||_{ \fH^s_p (T_1, T_2) }.
			$$
			If  $s - 2 \beta  -  d/p > 0$, then, by the embedding theorem for $H^s_p$ spaces
			$$
				u \in C^{\beta - 1/p} ([T_1, T_2], C^{  s - d/p - 2 \beta}).
			$$
			\end{remark}

		\mysection{Main results in the stochastic case}
									\label{section 3}

	For $n \in \bZ_{+}$ we denote
	\begin{equation}
				\label{3.4.1.0}
		B^{n} = \begin{cases} 
				B, \,  \textit{if} \,\,   n = 0, \\
		C^{  n  -1, 1}, \, n \in \bN.
		\end{cases}
	\end{equation}

		\textbf{\textit{Stability estimates.}}
		Fix some numbers 
		$n \in \bZ_{+}$,
		 $p \geq 2$,  
		$h, \gamma \in (0, 1]$,
		and $T > 0$.

		\begin{assumption}
						\label{assumption 3.1}
		   $
			a (t, x) = (a^{i j}(t, x), i, j = 1, \ldots d)
		    $ 
		are 
		 $
		    \mathcal{P} \times B(\bR^d) -
		$ 
		 measurable functions.
		In addition, there exists 
		$\delta_1, \delta_2 > 0$ such that,
		   for all 
		$
		   t \geq 0, x, \xi \in \bR^d, \omega,
		$
		 	\begin{equation}
						\label{3.2}
				\delta_1 |\xi|^2  \leq a^{i j } (t, x) \xi^i \xi^j \leq \delta_2 |\xi|^2.
		 	\end{equation}
		\end{assumption}

	\begin{assumption}
					\label{assumption 3.2}
	 There exist numbers 
	$d_1 \geq d$,
		$\delta^{*}_1, \delta^{*}_2 > 0$,
	   and a set of  vectors
	$
		\Lambda = \{l_k \in \bZ^d, k = 1, \ldots, d_1\},
	$
	and 
	  $\cP \times B (\bR^d)$-measurable functions 
	$
		\lambda_k, k = 1, \ldots, d_1
	$
	such that the following properties hold:
	
	$(i)$ $l_k = e_k, k = 1, \ldots, d$;

	$(ii)$	 for any $k,\omega, t, x$, 
		 $
			 \lambda_k (t, x) \in  [\delta^{*}_1, \delta^{*}_2];
		$
 
	$(iii)$ for any $i, j, \omega, t, x$,
		$$
			a^{i j} (t, x) =  \sum_{ k = 1}^{d_1}  \lambda_k (t, x) l_k^i   l_k^j;
		$$

	$(iv)$ if  the function $a$ is independent of $x$, then so are $\lambda_k, k = 1, \ldots, d_1$.

	\end{assumption}

	\begin{assumption}
					\label{assumption 3.3}
	For any $\varepsilon > 0$,
	 there exists $\kappa_{\varepsilon} > 0$ such that,
	for any $k, \omega,  t$  and any $x, y$ such that 
	$
		|x - y| < \kappa_{\varepsilon}
	$,
	 one has 
	\begin{equation}
		|\lambda_k (t, x) - \lambda_k (t, y)| < \varepsilon.
	\end{equation}
	\end{assumption}

	\begin{assumption}
					\label{assumption 3.4}
	$
		  b (t, x) =  (b^i (t, x), i = 1, \ldots, d),
	$
	 $
		\nu (t, x) = (\nu^l (t, x), l = 1, \ldots, d_0)
	$
		are $\cP \times \bR^d$-measurable functions.
	\end{assumption}
	
	\begin{assumption}
					\label{assumption 3.4.1}
	 There exists a constant $K_1 > 0$ such that,
	for any 
	$
		k, i, l,
	$
	and any
	 $
		\omega, t, x,
	$
	$$
		 ||\lambda_k (t, \cdot)||_{  B^n }
								 + ||b^i (t, \cdot)||_{  B^n  } + ||\nu^l (t, \cdot) ||_{  B^{n + 1}  } < K_1.
	 $$
	\end{assumption}

	\begin{assumption}
					\label{assumption 3.8}
	Fix some number 
	$
		\rho \in (0, (2d_1 \delta_2^{*})^{-1})
	$
	and assume that
	$
		\gamma \leq \rho h^2
	$,
	and
	$
		T/\gamma \geq 2
	$ 
	is an integer.
	\end{assumption}

	\begin{remark}
				\label{remark 3.1}
	Let $S_{\delta_1, \delta_2}$ be the set of symmetric $d\times d$ matrices such that, 
	for any 
	$a \in S_{\delta_1, \delta_2}$,
	 and 
	$\xi \in \bR^d$,
	\begin{equation}
				\label{1.0}
		\delta_1 |\xi|^2 \leq a^{i j} \xi^i \xi^j \leq \delta_2 |\xi|^2.
	\end{equation}
	Let 
	$B (S_{\delta_1, \delta_2})$ 
	be the sigma-algebra of Borel sets of $S_{\delta_1, \delta_2}$.

	We set 
	$
		l_i = e_i, i = 1, \ldots, d.
	$
	It turns out that there exists  a number $d_1 > d$ and
	\begin{itemize}
	\item	 vectors 
	$
		l_k \in \bZ^d, k = d+1, \ldots, d_1
	$,\\
	
	\item real-analytic functions 
	$
		\lambda_k (a), k = 1, \ldots, d_1
	$
	 on
	 $
		S_{\delta_1, \delta_2}
	$,\\
 
	\item numbers 
	$
		\delta^{*}_1, \delta^{*}_2 > 0,
	$
	 depending only on 
	$d, \delta_1, \delta_2$,
	\end{itemize}
 		 such that, for any 
	$
		a \in S_{\delta_1, \delta_2}
	$,   
	$$
		a  =  \sum_{k = 1}^{d_1} \lambda_k (a) l_k l^{*}_k,
															\quad \lambda_k \in [\delta^{*}_1, \delta^{*}_2].
	$$
	The proof of this assertion can be found in Appendix A of \cite{Kr_18}.	

	This implies that, 
	if for each $i, j$,
	the function $a^{i j}$ is a 
	$
		\cP \times B (\bR^{d})
	$-measurable, 
	then
	Assumption \ref{assumption 3.2}  holds with some
	$d_1$, 
	and  
	$
		l_k, \lambda_k, k  =  1, \ldots, d_1
	$.
	Further, 
	by the same theorem, 
	if, for all $i, j$,
	$
		a^{i j} (t, \cdot) \in B^{s},
	$
	then the same holds for all functions
	$
		\lambda_k (t, \cdot).
	$
	\end{remark}

	Here is the main result of this paper.
	
	\begin{theorem}
				\label{theorem 3.2}
	Let 
	$
	  n \in \bZ_{+},
	$
	 $p \geq 2$,
	$T > 0$, 
	and 
	$d_0 \in \bN$
	be numbers.
	Let 
	$
		f \in \bH^n_{p; \, \gamma, h} (T/\gamma),
	$
	 $ 
		g \in \bH^{n + 1}_{p; \, \gamma, h} (T/\gamma).
	$
	In addition, assume that 
	$
		f (t_m, x)$ is an $\cF_{t_{m+1}}
	$-measurable random variable, 
	for any 
	$m \in \gamma\bZ_{+}$ 
	and $x \in h \bZ^d$.
	Let Assumptions \ref{assumption 3.1} -  \ref{assumption 3.8}  hold
	and let $v_{\gamma, h}$ be the solution of  \eqref{3.1} with $u_0 \equiv 0$.
	Then, 
	\begin{equation}
					\label{3.5}
		E  \sum_{r = 0}^{ T/\gamma}  || v_{\gamma, h}  (t_r, \cdot)||^p_{ n + 2, p; \,  h} 
															 				\leq N \fR_n, 
	\end{equation}
	where
	$$
		\fR_n   =  E \sum_{r = 0}^{ (T/\gamma) - 1} (|| f (t_r, \cdot) ||^p_{ n, p; \, h}
	 		 + \sum_{l = 1}^{d_0} || g^l (t_r, \cdot)||^p_{ n + 1, p; \, h}),
	$$
	and the constant $N$
	depends only on 
	$p$, 
	$
	   d, d_0,  d_1,
	$
	 $
		\delta^{*}_1, \delta^{*}_2
	$,
	 $n, K_1$, 
	$
		\Lambda, \rho
	$,
	and
	 $
		T.
	$
	In addition, 
	\begin{equation}
				\label{3.5.1}
		E  \max_{r = 0, \ldots,  T/\gamma}   || v_{\gamma, h} (t_r, \cdot) ||^p_{ n + 1, p; \, h} \leq N  h^2 \mathfrak{R}_n.
	\end{equation}
	\end{theorem}
	
	\begin{remark}
	In the proof of Theorem \ref{theorem 3.1}
	 we  use the stability estimates \eqref{3.5.0} and \eqref{3.5.1}
	 for an equation of type \eqref{3.1}  with 
	$f (t, x)$ replaced with
	 $f (t+\gamma, x)$.
	This is the reason why we need 
	$f (t, x)$ to be 
	$\cF_{t+\gamma}$-measurable in the statement of Theorem \ref{theorem 3.2}.
	Note that, under assumptions of Theorem \ref{theorem 3.2},
	$
		v_{\gamma, h} (t, x)
	$
	 is  an $\cF_t$-measurable random variable, 
	for any $t \in \gamma\bZ_{+}$, 
	$x \in h \bZ^d$.
	\end{remark}

	\textbf{\textit{Rate of convergence.}}
	We fix some numbers 
	$
		n \in \bZ_{+}, s \in (0, \infty)\setminus \bN
	$,
	$s_1 > 0$, 
					 $p > 2$, 
							$\theta \in (1/p, 1/2)$,
		 										$\eta \in (0, 1)$,
																and $T > 0$.	

		In this paper we consider the following SPDE:
	\begin{equation} 
	 \begin{aligned}
				\label{1.2}
		dv (t, x)& = [ \sum_{i, j = 1}^d a^{i j} (t, x)  D_{i j} v (t, x) + \sum_{i = 1}^d  b^i (t, x) D_i v (t, x) + f(t, x)] \, dt\\
													&+ \sum_{l = 1}^{d_0} [\nu^l (t, x) v (t, x) + g^l (t, x)] \, dw^l (t), \quad v (0, x)  = u_0 (x).
	   \end{aligned}
	\end{equation}	
	In the sequel, where it is clear what is the range of the index of summation,
	we  omit  writing the  symbol $\sum$.

		\begin{definition}
		We say that $v \in \cH^s_p (0, T)$ is a solution of \eqref{1.2}
		if 
		$$
			\mathbb{D} v = a^{i j} (t, x)  D_{i j} v (t, x) + b^i (t, x) D_i v (t, x) + f(t, x),
		$$
	 	 $$
			\mathbb{S} v = 	(\nu^l (t, x) u (t, x) + g^l (t, x), l = 1, \ldots, d_0),									
		 $$
		and $v(0, x) = u_0 (x)$.

		\end{definition}

	\begin{assumption}
					\label{assumption 3.4.2}
	There exists a number $K_2 > 0$ such that,
	for
	 any 
	$
		i, j, l, \omega, t
	$, 
	 $$
		||a^{i j} (t, \cdot)||_{ C^{s+1+\eta}} 
								 + ||b^i (t, \cdot)||_{ C^{s+1+\eta} } + ||\nu^l  (t, \cdot) ||_{ C^{s+2+\eta} } < K_2,
	 $$
	 where 
		$\eta \in  (0, 1)$ is a  number 
	such that
	$s + \eta \not \in \bN$. 
	\end{assumption}

	\begin{assumption}
					\label{assumption 3.5}
	
		There exists a number $K_3 > 0$ such that, 
		for any 
		$
			i, j,  l, x, y,\omega, t,
		$
		$$
			||a^{i j}||_{  C^{ \theta - 1/p } ([0, T], C^{n - 1, 1}) }   +  || b^i ||_{  C^{ \theta - 1/p } ([0, T], C^{n - 1, 1}) } + ||\nu^l||_{  C^{ \theta - 1/p } ([0, T], C^{s+2}) } < K_3.
		$$
	\end{assumption}

	\begin{assumption}
					\label{assumption 3.6}
				$
			      	 u_0 \in L_p (\Omega, \cF_0, H^{s+3 - 2/p}_p),
			        $
				 $
					f  \in \bH^{s+1}_p ([0, T]), 
				 $
			  	  $
					g \in \bH^{s + 2}_p ([0, T]).
				  $
	\end{assumption}

				\begin{assumption}
							\label{assumption 3.7}
				$
					 E ||f||^p_{C^{\theta - 1/p} ([0, T],  H^{s_1}_p)} 
					 +  \sum_{ l = 1}^{d_0} E ||g^l||^p_{  C^{\theta - 1/p} ([0, T],  H^{s+2}_p)} <\infty.
				$
	\end{assumption}

	\begin{remark}
				\label{remark 3.3}
	Even if we require more regularity of $f$ and $g$ in the temporal variable,
	by using our methods,
	 we will still get the same rate of convergence in the aforementioned theorem.
	Note that the power $\theta   - 1/p$ in the inequality \eqref{3.3}
	comes from the estimate of the modulus of continuity of a function 
	from a stochastic Banach space (see Remark \ref{remark 2.1}).
	This author believes that one can get the rate of convergence of order $h^p$ in Theorem \ref{theorem 3.1}
	by establishing and using the $l_p$-theory of semi-discrete stochastic equations.
	\end{remark}

	The following theorem is a particular case of Theorem 5.1 of \cite{Kr_99}.

	\begin{theorem}
				\label{theorem 3.5}
	Let Assumptions \ref{assumption 3.1}, 
	\ref{assumption 3.4},
	 \ref{assumption 3.4.2}, \ref{assumption 3.6} hold. 
	Then, by Theorem 5.1 of \cite{Kr_99} 
	there exists a unique solution 
	$
		v \in \cH^{s+3}_p ([0, T])
	$
	of \eqref{1.2}, and  
	$$
		|| v ||_{\cH^{s+3}_p (0, T) } 
			\leq N (|| f ||_{\bH^{s+1}_p (0, T)} + \sum_{l = 1}^{d_0} ||g^l ||_{\bH^{s+2}_p (0, T)} + (E ||u_0||^p_{s+3 - 2/p, p})^{1/p}),
	$$
	where $N$ depends only on 
	$
		p, s, d, d_0, \delta_1, \delta_2, \eta,  K_2
	$
	 and $T$.
	
	\end{theorem}

	\begin{corollary}
				\label{corollary 3.1}
	It follows from Remark \ref{remark 2.1} that $v$ is a classical 
	solution of Eq. \eqref{1.2} i.e., for all $t \in [0, T]$, 
	$v (t, \cdot)$ is a  twice continuously differentiable function, 
	and Eq. \eqref{1.2} is satisfied not in the weak sense,
	but, for all $\omega, t, x$.

	\end{corollary}

	As an application of  Theorem \ref{theorem 3.2} we provide an estimate of the rate of convergence of the finite difference  scheme \eqref{3.1}.
	\begin{theorem}
				\label{theorem 3.1}
	Let
	$p > d \vee 2$,
				$ T > 0$, 
						$n \in \bZ_{+}$,
						$s_1 > n + d/p$,
									$\eta \in (0, 1)$ be numbers.
	Let
	$\mu$ and $\theta$ be constants such that
	$
		1/2 > \mu > \theta > 1/p,
	$ 
	and
	 $s$
	be a number such that 
	$
		s \in (n + p/d + 2\mu, \infty)\setminus \bN.
	$
	Let Assumptions \ref{assumption 3.1} - \ref{assumption 3.7} hold,
	and 
	$
		v \in \cH^{s+3}_p (0, T)
	$
	 be the unique solution of \eqref{1.2}
	 (see   
		Theorem \ref{theorem 3.5}),
	and $v_{\gamma, h}$ be the solution of \eqref{3.1}.
	Then, 
	 \begin{equation}
				\label{3.3} 
		\gamma E  \sum_{m = 0}^{T/\gamma} || v (t_m, \cdot) - v_{\gamma, h} (t_m, \cdot)||^p_{ n+2, p; \, h} 
																							\leq N h^{ 2 p (\theta  - 1/p)} R,
	\end{equation}
	where
	\begin{align*}
		R = &     (|| f  ||^p_{ \mathbb{H}^{s+1}_p (0, T) } 
	+  \sum_{l = 1}^{d_0} ||g^l ||^p_{\mathbb{H}^{s+2}_p (0, T)} 
		+  E  ||f||^p_{C^{\theta  -  1/p}([0, T], H^{s_1}_p)} \\
		& +    \sum_{l = 1}^{d_0} E ||g^l||^p_{C^{\theta -  1/p}([0, T], H^{s+2}_p)}
		+ \gamma  \sum_{l = 1}^{d_0}   \sum_{m = 0}^{(T/\gamma) - 1}  
												E	|| g^l (t_m, \cdot) ||^p_{n+2, p; \, h}),
	\end{align*}
	and $N$  depends only on 
	$p$,
	 $d, d_0,  d_1,$ 
	$\delta_1, \delta_2$,
	$\delta^{*}_1, \delta^{*}_2$,
	 $ n,  s, s_1, \eta$, 
	$
		\Lambda, \rho,
	$
	$
		K_1, K_2, K_3,
	$ 
	$
		\mu, \theta,
	$
	and $T$.
	In addition, 
	$$
		E  \max_{m = 0, \ldots, T/\gamma} \sup_{x \in h \bZ^d} \sum_{|\alpha| \leq n}  | \delta^{\alpha} (v (t_m, x) - v_h (t_m, x))|^p  \leq N h^{2 p (\theta  - 1/p)} R.																		
	$$
	\end{theorem}

	\textit{\textbf{Discrete version of the parabolic Littlewood-Paley inequality.}}	
	Let Assumptions \ref{assumption 3.1} and \ref{assumption 3.2} hold
	and assume that  
	$
		a
	$
	 is independent of  $x$. 	
	Let 
	$\rho > 0$
	 be a number.
	For any $m \in \bZ$
	by
	$
	  G (m, \cdot, \cdot), 
	$
	we denote the solution to the  following equation: 
	\begin{equation}
				\label{3.4}
	\begin{aligned}
		  z (n+1, y) - z (n, y) & = \rho \sum_{k = 1}^{d_1} \lambda_k (n) \Delta_{l_k, 1} z (n, y), \\
																& z (m, y) =  I_{y  = 0}, \,  n \in [m, \infty) \cap \bZ, y \in \bZ^d.
	\end{aligned}
	\end{equation}

	\begin{theorem}
				\label{theorem 3.3}
	 Let 
	Assumptions \ref{assumption 3.1}
	and  \ref{assumption 3.2} hold.
	Fix some number 
	$
		\rho \in (0, (2d_1 \delta^{*}_2)^{-1}).
	$ 
	Assume that the function 
	$a$ is independent of $x$ and $\omega$.
	Let   
	 $
		g :  \bZ^{d+1} \to \bR
	$
	be a
	deterministic function
	 with compact support.  
	Then, for any $p \geq 2$,
	and any $k \in \{1, \ldots, d\}$,
	we have
	$$
		 \sum_{ n  \in \bZ}  || \sum_{ j  = -\infty}^{n - 1} |\delta_{e_k, 1} G (j, n, \cdot) 	\ast_1  g (j, \cdot) (\cdot)|^2||^{p/2}_{p/2; \, 1} \leq 
		N  || g   ||_{  l_p (\bZ^{d+1}) }^p.
	$$
	Here 
	$
		G
	$
	 is the function defined by \eqref{3.4}, 
	and  $N$ depends only on $p$, 
	$
		d, d_1, \delta^{*}_1, \delta^{*}_2, \Lambda, \rho.
	$
	\end{theorem}

	 The next theorem demonstrates how  Theorem \ref{theorem 3.3} is related to Theorem \ref{theorem 3.2}.

	\begin{theorem}
				\label{theorem 3.4}
	Invoke the conditions and the notations of Theorem \ref{theorem 3.2}  
	with
	the function 
	$a$ independent of $x$ and $\omega$,
	and
	$
		b \equiv 0,  \nu  \equiv 0,    f \equiv  0 \equiv u_0 .
	$ 
	Let 
	$v_{\gamma, h}$
	 be the solution of Eq. \eqref{3.1}. 
	Then,  for any $n \in \bZ_{+}$
	$$
		E \sum_{m = 0}^{T/\gamma} [ v_{\gamma, h} (t_m, \cdot) ]^p_{ n + 1, p; \, h}
		 \leq N   \sum_{l = 1}^{d_0} E \sum_{m = 0}^{(T/\gamma) - 1} [ g^l (t_m, \cdot) ]^p_{ n,  p; \, h}, 
	$$
	where $N$ depends only on $n, d, d_0, d_1,  \delta_1^{*}, \delta_2^{*}, \rho, \Lambda$.
	\end{theorem}

	\begin{proof}

	Let $\alpha$ be a multi-index of order $n$.
	 Observe that 
	$
		\delta^{\alpha}_h v_{\gamma, h}
	$  
	satisfies Eq. \eqref{3.1} 
	with 
	$
		b \equiv 0,   \nu \equiv 0,   f \equiv 0 \equiv u_0 
	$	
	and 
	$
		g^l
	$ 
	replaced by $\delta^{\alpha}_h g^l$.
	Hence, we may assume that $n = 0$.

	Next, 
	for any $r, m \in \bZ$, 
	such that $m \geq r$, 
	and $x \in h \bZ^d$, 
	denote
	 \begin{equation}
				\label{3.4.1}
		G_h (r, m, x) = G (r,  m, x/h).
	 \end{equation}
	
	By Duhamel's principle (see Lemma \ref{lemma 11.6}),
	for any $m \geq 2$, 
	$x \in h \bZ^d$,
	 and $i$,
	$$
		\delta_{e_i, h} v_{\gamma, h} (t_m, x) =  \sum_{l = 1}^{d_0}  (\delta_{e_i, h}  g^l (t_{m - 1}, x)  (w^l (t_m) - w^l (t_{m-1})) + \tilde v^l (t_m, x)),
	$$
	where 
	$$	
		\tilde v^l (t_m, x) = 	\sum_{j = 0}^{m-2} \sum_{l = 1}^{d_0}   (\delta_{e_i, h}  G_{ h} (j,  m - 1, \cdot)) \ast_h g^l (t_j, \cdot) (x)  (w^l (t_{j+1}) - w^l (t_j)).
	$$
	 We rewrite $\tilde v^l (t_m, x)$ as a stochastic integral as follows:
	 $$
			\tilde v^l (t_m, x) =  	\sum_{l = 1}^{d_0}	\int_0^{ t_{m - 1} }  \tilde g^l (r, x) \, dw^l (r),
	 $$
	where 
	$$
	  	 \tilde g^l (r, x) = \sum_{ j = 0}^{m - 2}  (\delta_{e_i, h} G_{ h} (j, m - 1, \cdot)) \ast_h g^l (t_j, \cdot) (x) I_{ t_j \leq r < t_{j+1} }.
	$$
	First, since $g^l (t_{m - 1}, x)$ is $\cF_{t_{ m  - 1} }$-measurable,
	$$
		 \sum_{x \in h \bZ^d} E |\delta_{e_i, h}  g^l (t_{m - 1}, x)  (w^l (t_m) - w^l (t_{m - 1}))|^p 
	$$
	 $$
			  \leq N (p,  \rho)   \sum_{x \in h \bZ^d}  E | (T_{e_i, h} - 1)g^l (t_m, x)|^p
		\leq  N (p,  \rho)    ||g^l (t_m, \cdot)||^p_{p; \, 1},
	$$
	and, hence, we may concentrate on the function $\tilde v^l$.

	Since $G_h (j, m-1, x)$ is independent of $\omega$, 
	the random variable
	$
		\tilde g^l (r, x)
	$ 
	is $\cF_{r}$-measurable, for any $r > 0$.
	Then, by the Burkholder-Davis-Gundy inequality
	$$
		E |\tilde v^l (t_m, x)|^p \leq
	 N (p, d_0) \sum_{l = 1}^{d_0} E   ( \int_0^{ t_{m - 1} } |\delta_{e_i, h} \tilde g^l (r, x)|^2 \, dr)^{p/2} 
	$$
	  $$
		 = N (p, d_0) \gamma^{p/2} h^{-p} \sum_{l = 1}^{d_0}  E (\sum_{j = 0}^{m - 2} |(T_{e_i, h} - 1) G_{ h}  (j, m - 1, \cdot) \ast_h g^l (t_j, \cdot) (x)|^2)^{p/2} 
	 $$
	  $$
		= N (p, d_0, \rho) \sum_{l = 1}^{d_0}  E (\sum_{j = 0}^{m - 2} |(T_{e_i, h} - 1) G_{ h} (j, m - 1 \cdot) \ast_h g^l (t_j, \cdot) (x)|^2)^{p/2}.
	  $$
	By what was just proved
	we may  assume
	$
		\gamma = 1 = h
	$
	and 
	  $T \in \bN$.

	 We fix $\omega$ 
	and replace 
	$g^l$
	by the function
	$$
		 g^l_r (k, x) = g^l (k, x) I_{k \in [0, T - 1]} I_{ ||x|| \leq r}
	$$
	which has  compact support.
	By Theorem \ref{theorem 3.3}
	$$
		 \sum_{ m  =  0}^T  || \sum_{ j  = 0}^{m - 2} |\delta_{e_k, 1} G (j, m - 1, \cdot) 	\ast_1  g^l_r (j, \cdot) (\cdot)|^2||^{p/2}_{p/2; \, 1} \leq 
		N   \sum_{ m  = 0}^{T-2} || g^l_r (m, \cdot)  ||_{p; \,  1}^p.
	$$
	Then we  take a limit as $m \to \infty$.
	Note that by Lemma \ref{lemma 11.1} 
	$
	\delta_{e_k, 1} G (j, m-1, \cdot)
	$
	has  compact support.
	This combined with the fact that
	$$
		\lim_{r \to \infty} ||g^l_r (t, \cdot) - g^l (t, \cdot) ||_{ l_{\infty} (\bZ^d) } \to 0, \, \, \forall t \in \{0, 1, \ldots, T-1\}
	$$
	implies that 
	$
		\delta_{e_k, 1} G (j, m - 1, \cdot) \ast_1 g^l_r (j, \cdot)
	$
	converges to 
	$
		\delta_{e_k, 1} G (j, m - 1, \cdot) \ast_1 g^l (j, \cdot) (x)
	$
	as $r \to \infty$,
	for any $j, m, x$.
	Now the assertion follows from Fatou's lemma.

	\end{proof}

	\mysection{Main results for deterministic equations}
				\label{section 4}
		In this section all the functions are assumed to be deterministic i.e. independent of $\omega$.

	{\textit{\textbf{Stability estimates.}}
	Fix some numbers $n  \in \bZ_{+}$,
		 $p > 1$,  
		$h, \gamma \in (0, 1]$,
		and $T > 0$.

	Here is the main result for deterministic finite difference equations.
	\begin{theorem}
				\label{theorem 4.1}
	Let $T > 0$, 
	$n \in \bZ_{+}$,
	$p > 1$
	be numbers.
	Invoke
	all the assumptions of Theorem \ref{theorem 3.2}.
	In addition, we also assume that 
	$
		g^l \equiv 0 \equiv \nu^l, l = 1, \ldots, d_0
	$ 
	and that all the functions 
	$
		a^{i j}, b^i,  \lambda_k, f
	$
	  are independent of $\omega$.
	Let
	 $
		u_{\gamma, h}
	$
	 be the  solution of Eq. \eqref{4.1} with $u_0 \equiv 0$.
	Then,    we have
	\begin{equation}
				\label{4.1.0.0}
		   \sum_{m = 0}^{T/\gamma} || u_{\gamma, h} (t_m, \cdot) ||^p_{ n + 2,  p; \, h}
																									 \leq N   \sum_{m = 0}^{(T/\gamma) - 1} || f (t_m, \cdot)||^p_{ n, p; \, h},																										
	\end{equation}
	and, in addition, if $p \geq 2$, then 
	\begin{equation}
				\label{4.1.0}
		\max_{m= 0, \ldots, T/\gamma} || u_{\gamma, h} (t_m, \cdot) ||^p_{ n + 1, p; \, h} \leq
																				 N h^2  \sum_{m = 0}^{(T/\gamma) - 1} || f (t_m, \cdot)||^p_{ n, p; \, h},
	\end{equation}													
	\end{theorem}
	where $N$ depends only on 
	$p, d,  d_1, 
	\delta^{*}_1, \delta^{*}_2,
	 n,  K_1, 
		\Lambda, \rho
	$
	and 
	$T$.

	\textit{\textbf{Rate of convergence.}}
		Fix some numbers 
		$
			n \in \bZ_{+},   s  \in  (0, \infty)\setminus \bN
		$,
		$s_1 >  0$,  
		$p > 1$,
		$\eta \in (0, 1)$,  
		$h, \gamma \in (0, 1]$,
		and $T > 0$.

	\begin{assumption}
					\label{assumption 4.3.1}
	There exists a constant $\hat K_2 > 0$ such that,
		 for any $i, j, t,$
		$$
			||a^{i j} (t, \cdot)  ||_{  C^{s+2+\eta}  } +   || b^i (t, \cdot)  ||_{  C^{s+2+\eta} } < \hat K_2.
		$$
	\end{assumption}

	\begin{assumption}
					\label{assumption 4.4}
		In addition, there exists a constant $\hat K_3 > 0$ such that,
		 for any $i, j,$
	 	 $$
			||a^{i j}  ||_{   C^{ 1 } ([0, T], C^{n-1, 1} )   }   +  || b^i  ||_{ C^{ 1 } ([0, T], C^{n-1, 1})   } < \hat K_3.
		 $$
	\end{assumption}

	\begin{assumption}
					\label{assumption 4.5}		
	  		  $
			       u_0 \in  H^{s+4 - 2/p}_p,
			  $
				 $
					f  \in L_p  ([0,T], H^{s+2}_p).
				$
	\end{assumption}

	\begin{assumption}
					\label{assumption 4.6}		
		$D_t f \in L_p ([0, T], H^{s_1}_p)$. 
	\end{assumption}

	In this section we consider the following PDE: 
	\begin{equation}
				\label{1.1}
	  \begin{aligned}
		D_t u (t, x) &= a^{i j} (t, x) D_{i j} u (t, x) \\
														&+ b^i (t, x) D_i u (t, x) + f (t, x), \quad u (0, x) = u_0 (x),
	  \end{aligned}
	\end{equation}

	 The following theorem addresses the existence and uniqueness of Eq. \eqref{1.1}.
	The proof can be found, for example, Theorem 5.1 of \cite{Kr_99}.

	\begin{theorem}
				\label{theorem 4.5}
	Let 
	$
		s \in (0, \infty)\setminus \bN
	$, 
	$p >  1, T > 0$ be numbers,
	and  
	Assumptions \ref{assumption 3.1},
	 \ref{assumption 3.4},
	 \ref{assumption 4.3.1},  \ref{assumption 4.5} hold.
	Then, there exists a unique solution 
	$
		u \in \fH^{s+4}_p (0, T)
	$
	 of \eqref{1.1}
	and, in addition, 
	$$
		|| u ||^p_{\fH^{s+4}_p (0, T)} 
			\leq N  (\int_0^T || f (t, \cdot) ||^p_{ H^{s+2}_p } \, dt + || u_0 ||^p_{ H^{ s+4 - 2/p}_p }),
	$$
	where $N$ depends only on $d, p, \delta_1, \delta_2, s, \eta,   K_1$ and $T$.	
	\end{theorem}

	From Theorem \ref{theorem 4.1} we derive the estimate of the rate of convergence of the finite difference scheme \eqref{4.1}.
	
	\begin{theorem}
					\label{theorem 4.2}
	Let $T > 0$, 
	$p > 1$, 
	  $
		n \in \bZ_{+}
	  $
	    $
		s \in (n + d/p, \infty) \setminus \bN,
	    $		
		$
			s_1 > n+ d/p
		$
	 be numbers.
	 Let Assumption \ref{assumption 3.1} - Assumptions \ref{assumption 3.8} 
	hold with 
	$
		g^l \equiv 0  \equiv \nu^l, l = 1, \ldots, d_0
	$
	and all the functions 
	$
		a^{i j}, b^i, \lambda_k, f
	$
	 independent of $\omega$.
	In addition, let Assumptions
	 \ref{assumption 4.3.1} - \ref{assumption 4.6} 
	be satisfied.
	Let 
	$
		u \in \mathfrak{H}^{s+4}_p (0, T)
	$
	 be the unique solution of \eqref{1.1} 
	(see
	Theorem \ref{theorem 4.5}),
	 and 
	$u_{\gamma, h}$
	 be the solution of \eqref{4.1}.
	Then, 
	\begin{equation}
				\label{4.2.0}
		\gamma \sum_{m = 0}^{T/\gamma}  || u (t_m, \cdot) - u_{\gamma, h} (t_m, \cdot)  ||^p_{n+2, p; \, h}  
				\leq N  h^{2  p }  \cR,
	\end{equation}
	where
	\begin{align*}
		\cR = || f ||^p_{ L_p ([0, T], H^{s+2}_p) }  				 +  || D_t f ||^p_{ L_p ([0, T], H^{s_1}_p) }
	+ ||u_0||^p_{  H^{s+4 - 2/p}_p },
	\end{align*}
	and $N$
	depends only
	on 
	$p, d, d_1$, $\delta_1, \delta_2$, $\delta_1^{*}, \delta_2^{*}, s, s_1,  n,  \eta$, $K_1, \hat K_2,  \hat K_3,  \Lambda, \rho$ and $T$.
	In addition, if $p \geq 2$, then															
	\begin{equation}
				\label{4.2.1}
		\max_{ m = 0, \ldots, T/\gamma}  \sup_{x \in h\bZ^d} \sum_{ |\alpha| \leq  n} |\delta^{\alpha}_h  (u (t_m, \cdot) - u_{\gamma, h} (t_m, \cdot))|  \leq N h^{2 p  } \cR.
	\end{equation}
	\end{theorem}

	\textit{\textbf{Discrete variant of the Calderon-Zygmund estimate.}}
	To prove  Theorem \ref{theorem 4.1} we need a discrete version of the Calderon-Zygmund estimate which we present below.

	\begin{theorem}
				\label{theorem 4.3}
	Let Assumptions \ref{assumption 3.1} and \ref{assumption 3.2} hold
	and assume that $a$ is independent of $\omega$ and $x$.
	Take any function $f $ on $\bZ^{d+1}$  with compact support.
	Let 
	$
		\rho \in (0, (2d_1 \delta^{*}_2)^{-1})
	$
	be a number and $G$ be the function defined by \eqref{3.4}.
	Then, for any $p > 1$, 
	we have
	$$
		 \sum_{ n  \in \bZ}  [\sum_{ j = - \infty }^{n-1}    G (j, n, \cdot) \ast_1 f (j, \cdot) (\cdot)]_{2, p; \,  1}^p 
														\leq N   || f  ||_{  l_p (\bZ^{d+1}) }^p,
	$$
	where $N$ depends only on $p, d, d_1, \rho, \delta^{*}_1, \delta^{*}_2, \Lambda$.

		\end{theorem}

	Here we elaborate on the connection between Theorem \ref{theorem 4.3}
	 and Theorem \ref{theorem 4.1}.

	\begin{theorem}
				\label{theorem 4.4}
	Invoke the conditions of Theorem \ref{theorem 4.1}
	and assume that $a$ is independent of $x$.
	 Let $u_{\gamma, h}$ be the solution of \eqref{4.1} with
	 $
		b \equiv 0,   u_0 \equiv 0.
	$
	 Then,
	$$
		\sum_{m = 0}^{T/\gamma} [ u_{\gamma, h} (t_m, \cdot) ]^p_{n+2, p; \, h} 
		\leq N  \sum_{m = 0}^{(T/\gamma) - 1} [ f (t_m, \cdot) ]^p_{n,  p; \, h},
	$$
	where $N$ depends only on $p, d, d_1, n,  \delta_1^{*}, \delta_2^{*}, \rho, \Lambda$.
	\end{theorem}

	\begin{proof}
	First, note that $\delta^{\alpha}_h u_{\gamma, h}$ satisfies Eq. \eqref{4.1}
	with 
	$
		b \equiv 0, u_0 \equiv 0
	$
	 and $f$ replaced
	by $\delta^{\alpha}_h f$. 
	Hence, may assume $n = 0$.
	
	By Lemma \ref{lemma 11.6}
	 for any 
	$
		m \in  \bN, x \in h \bZ^d,
	$ 
	and any multi-index $\alpha$,
	 we have
	$$
		\delta^{\alpha}_h u_{\gamma, h} (t_m, x) 
										= \gamma \delta^{\alpha}_h f (t_{m-1}, x) +  \tilde u (t_m, x), 
	$$
	where
	$$
		\tilde u (t_m, x) = \gamma \sum_{j = 0}^{m-1}   (\delta^{\alpha}_h G_{h} (j, m - 2,  \cdot) \ast_h    f (t_j, \cdot) (x).
	$$
	 Assume that $|\alpha| =2$.
	 Then, by the fact that $\gamma \leq \rho h^2$, 
	we have
	$$
		\gamma ||\delta^{\alpha}_h f (t_{m-1}, \cdot)||_{p; \, h}   \leq N (\rho) || f (t_{m-1}, \cdot)||_{p; \, h},
	$$
	 $$
		\delta^{\alpha}_h \tilde u (t_m, x) = \rho   \sum_{j = 0}^{m - 2}  \sum_{y \in \bZ^d} \delta^{\alpha}_1 G (j, m - 1, y) f (t_j, x - hy).
	 $$
	Hence, may assume that 
	$\gamma = 1 = h$, 
	and $T \in \bN$.
	By  the standard approximation argument (see Theorem \ref{theorem 3.4})
	we may assume that $f$ has compact support.
	Thus, by Theorem \ref{theorem 4.3}
	$$
		\sum_{m = 0}^{T}	||\delta^{\alpha}_1 \tilde u (m, \cdot) ||^p_{ p; \, 1 } \leq N \sum_{n = 0}^{T  - 2} || f (m, \cdot) ||^p_{ p; \, 1 }.
	$$
	\end{proof}

	\mysection{Auxiliary results}
						\label{section 5}
		For any $r \in \bN$, 
	and  any
	$t \in \bZ, x \in \bZ^d$,
	we set
	 $$
		C_r (x) = (\Pi_{i=1}^d (x_i - r, x_i + r)) \cap \bZ^d,  
										\quad C_r  = C_r (0), 
	$$
	 $$
		 Q_r (t, x) = ((t- r^2, t + r^2) \cap \bZ) \times C_r (x),
											 \quad Q_r  = Q_r (0,  0).
	 $$
	For any subset $A$ of $\bZ^d$ 
	 by $|A|$ and  we denote the number of elements in $A$.
	For a function 
	 $
		h : \bZ^{d+1} \to \bR,
	$
	 we define
	  $$
		\cM_x h (t, x) =   \sup_{r \in \bN} |C_r|^{-1} \sum_{  y \in C_r (x)  }  |h (t, y)|,					
	   $$
	    $$
			 \cM_t h (t, x) = \sup_{r \in \bN} (2r + 1)^{-1} \sum_{m = t - r}^{t + r}  |h (m, x)|.
	    $$
	For a function $f$ on $\bZ^d$
	 with compact support we define the Fourier transform of $f$ as follows: 
	\begin{equation}
				\label{5.0}
		\cF f (\xi) := \sum_{x \in \bZ^d} \exp ( - i \xi \cdot x) f (x), \, \xi \in \bR^d.
	\end{equation}
	Some properties of the Fourier transform are stated in Appendix A.

	\begin{lemma}
							\label{lemma 5.1}
	 Theorem \ref{theorem 3.3} holds for $p = 2$.
	\end{lemma}
	\begin{proof}
	Clearly, if suffices to prove the claim for $k = 1$.
	By the Parseval's identity  \eqref{11.1}
	 and the convolution property of the Fourier transform  \eqref{11.2}
	we have 
	$$
		I: =  \sum_{n \in \bZ}   ||\sum_{j = -\infty}^{n - 1} \delta_{e_1, 1} G (j, n, \cdot) \ast_1 g (j, \cdot) (\cdot)||^2_{  l_2 (\bZ^d) }
	$$
	 $$
		=  (2\pi)^{-d} \int_{[-\pi, \pi]^d} \sum_{n \in \bZ} \sum_{ j = -\infty}^{n-1}  |\cF  (\delta_{e_1, 1} G (j, n, \cdot))|^2 (\xi)   |\cF g (i, \xi)|^2  \, d\xi
	 $$
	  $$
		= \sum_{j \in \bZ}  (2\pi)^{-d} \int_{[-\pi, \pi]^d} |\cF g (j, \xi)|^2   \sum_{n = j + 1}^{\infty} | \cF  (\delta_{e_1, 1} G (j, n, \cdot))|^2 (\xi)  \, d\xi.
	  $$ 
	By Parseval's identity if suffices to show that, 
	for a.e. $\xi  \in [-\pi, \pi]^d$, 
	$$
		S (\xi) =  \sum_{n= j+1}^{\infty} | \mathcal{F} (\delta_{e_1, 1} G (j, n, \cdot))|^2 (\xi)  \leq N
	$$
	with $N$ is independent of $\xi$.
	By Lemma \ref{lemma 11.1} $(ii)$
	$$
		S (\xi) = 2 (1 - \cos (\xi_1)) \sum_{n = j  + 1}^{\infty} \Pi_{m = j+1}^n    | 1 - 2\rho \sum_{k = 1}^{d_1}  \lambda_k (m) (1 - \cos (\xi \cdot l_k))|^2.
	$$
	Clearly,
	\begin{equation}
				\label{5.1.1}
		| 1 - 2\rho \sum_{k = 1}^{d_1}  \lambda_k (m) (1 - \cos (\xi \cdot l_k))|  \leq M (\xi), 
	\end{equation}
	where
	\begin{equation}
				\label{5.1.2}
		M (\xi)  =    \max_{r=  1, 2} |1 -  2\rho \delta_r^{*} \sum_{k  = 1}^{d_1} (1 - \cos (\xi \cdot l_k))|.
	\end{equation}
	We claim
	 \begin{equation}
				\label{5.1.3}
		M (\xi) < 1, \, 
		\forall \xi \in [-\pi, \pi]^d\setminus \{0\}.
	 \end{equation}
	Note that in  \eqref{5.1.2} the expression inside the absolute value bars is strictly less than $1$.
	Hence, we only need to show that
	$$
		1 -  2\rho \delta_r^{*} \sum_{k = 1}^{d_1}  (1 - \cos (\xi \cdot l_k)) \geq 1 -  4\rho \delta_2^{*} d_1 > -1,
	$$
	and the latter holds because
	$
		\rho \in (0, (2 \delta_2^{*} d_1)^{-1}).
	$
	Then, 
	 \begin{equation}
				\label{5.1.4}
	  \begin{aligned}
		&S (\xi) \leq   2 (1 - \cos (\xi_1))  \sum_{n = j + 1}^{\infty}    M^{2 (n-j)} (x)\\
		 & = 2  (1 - \cos (\xi_1))  M^2 (\xi) (1 - M^2 (\xi))^{-1}.
	  \end{aligned}
	     \end{equation}
	Further, since $l_1 = e_1$, 
	for any $r \in \{1, 2\}$, 
	 we have
	$$
		1 - |1 -  2\rho \delta_r^{*} \sum_{k = 1}^{d_1}  (1 - \cos (\xi \cdot l_k))|^2
	$$
	 $$
		= 4 \rho \delta_r^{*} (1  - \rho \delta_r^{*} \sum_{k = 1}^{d_1} (1 - \cos (\xi \cdot l_k)))  (\sum_{k = 1}^{d_1} (1 - \cos (\xi \cdot l_k)))
	 $$
	  $$
		  \geq  4 \rho \delta_1^{*} (1 - 2 \rho \delta_2^{*} d_1 ) (1 - \cos (\xi_1))
	  $$
	and, hence $S (\xi) < N$, for a.e. $\xi$.
	This combined with \eqref{5.1.3} and \eqref{5.1.4} proves the assertion.
	\end{proof}

	\begin{lemma}
				\label{lemma 5.2}
	Theorem \ref{theorem 4.3} holds with $p = 2$.
	\end{lemma}
	\begin{proof}
	Denote
	$$
		v (n, x) = \sum_{ j = -\infty }^{n-1} G (j, n, \cdot) \ast_1 f (j, \cdot) (x).
	$$
	First, recall that $f$ has  compact support and then, in 
	the above sum we have a finite number of nonzero terms.
	Each term has  compact support as a function of $x$, 
	because it is a convolution of two functions with compact support.
	The fact that $G (j, n, \cdot)$ vanishes for large $x$
	follows, for example, from Lemma \ref{lemma 11.1} $(i)$.
	Then,  by the convolution property of Fourier transform  \eqref{11.2}
	$$
		\cF v (n, \xi) = 	\sum_{ j = -\infty }^{n-1} \cF G (j, n, \xi)  \cF f (j, \xi).
	$$
	By Lemma \ref{lemma 11.1} $(ii)$ and \eqref{5.1.1}
	$$
		|\cF G (j, n, \xi)|   
		\leq   M^{n - j} (\xi) I_{n > j},
	 $$
	where $M(\xi)$ is defined in \eqref{5.1.2}.
	Then, by this and Young's inequality
	\begin{equation}
				\label{5.2.1}
	  \begin{aligned}
	 &	\sum_{n \in \bZ}  |\cF v (n, \xi)|^2  
		\leq \sum_{n \in \bZ}  |\sum_{ j  = -\infty}^{n-1} M^{n-j} (\xi) \hat f (j, \xi)|^2  \\
	&	\leq (\sum_{ j =  1}^{\infty} M^j (\xi))^2 \sum_{j \in \bZ} |\cF f (j, \xi)|^2.
	 \end{aligned}
	 \end{equation}
	By \eqref{5.1.3}
	$$
		 \sum_{ j =  1}^{\infty} M^j (\xi) \leq
		  M (\xi) (1 - M (\xi))^{-2} 
				\leq (2 \rho \delta_1^{*})^{-1} (\sum_{k = 1}^{d_1}  (1 - \cos (\xi \cdot l_k))^{-1}.
	$$
	Combining this with \eqref{5.2.1} we obtain
	$$
		\sum_{n \in \bZ} \sum_{k = 1}^{d_1} (1 - \cos (\xi \cdot l_k))^2 |\cF v (n, \xi)|^2 
			  \leq N  \sum_{n \in \bZ} |\cF f (n, \xi)|^2.
	$$
	Integrating both parts of the above inequality with respect to $\xi$ and using  Parseval's identity, we prove the claim.
	\end{proof}

	\begin{lemma}
				\label{lemma 5.3}
	Let the assumptions of Theorem \ref{theorem 3.3} hold.
	Let 
	$
		r \in \bN,
	$ 
	 $
		(n, x) \in Q_r,
	 $
	and $j \leq n - 1$.
	Assume that
	$
		g (\cdot, y)   = 0
	$
	 for
	 $
	    y \in C_{2r} (0).
	 $
	Let $\alpha$ be a multi-index of order $1$ or $2$ and denote
	$$
		u (j, n, x) =  \sum_{ y \in \bZ^d} \delta^{\alpha}_1 G (j, n, y) g (j, x - y).
	$$ 
	Then, 
		$$
			|u (j, n, x)| \leq N (d, d_1, \delta_1^{*}, \delta_2^{*}, \rho, \Lambda)  r^{- |\alpha| } \cM_x g (j,  0).
		$$ 
	\end{lemma}

	\begin{proof}
	First, recall that by the argument of Lemma \ref{lemma 11.6} 
	$
		u (j, n, x), n \geq j, x \in \bZ^d
	$ 
	is the solution of the following equation:
	 $$
		v (m+1, x) - v(m, x) = \rho \sum_{k = 1}^{d_1} \lambda_k (m) \Delta_{l_k, 1} v (m, x), 
	$$
	 $$
									 v (j, x) = \delta^{\alpha}_1 g (j, x), m \geq j, x \in \bZ^d. 
	 $$
	Hence, by replacing  $\lambda_k (\cdot)$
	 by 
	$
		\lambda_k (\cdot + j)
	$
	 and  
	$
		g (\cdot, x)
	$ by
	 $
		g (\cdot + j, x),
	$
	 we may assume that
	$j = 0$ and $n \geq 1$.	

	 Next, denote
	$
		\tilde \alpha =  \alpha + (1, \ldots, 1).
	$
	By  Lemma \ref{lemma 9.5} $(ii)$  we get
	\begin{equation}
				\label{5.3.1}
	\begin{aligned}
		|u (0, n, x)| 
		&\leq
					 \sum_{  y \in \bZ^d}  |\delta^{\tilde \alpha}_1  G (0, n, y)|  \, \sum_{z\in \bZ^d: ||z|| \leq ||y||} |g (0, x - z)|\\
	& 
		= \sum_{  y \in \bZ^d}  |\delta^{\tilde \alpha}_1  G (0, n, y)|  \, \sum_{z\in  C_{||y|} (x) } |g (0, z)|.
	\end{aligned}
	 \end{equation}
	 Observe that,
	if
	 $
		||z || \leq ||y||,
	$ 
	and
	$
		||y|| \leq r,
	$
	then,
	since 
	$
		||x|| \leq r,
	$ 
	we have
	$
		||x - z|| \leq 2r. 
	$	
	Therefore,
	 since 
	$
	   g (0, \cdot) = 0
	$ 
	 inside
	$C_{2r}$,
	  we may replace the sum over $y$ in  \eqref{5.3.1} by
	the sum over 
	$|| y || >   r$.
	Next, 
	for
	$
	  ||y|| > r,
	$
	we have
	 $$
	 	   C_{||y||} (x) \subset C_{||y||+ r} \subset C_{2||y||}.
	 $$
	By what was just said
	$$
		|u (0, n, x)| \leq    \cM_{x} g (n, 0) \sum_{   ||y|| \geq r }   ||y||^d \, \delta^{\tilde \alpha}_1   G (0, n, y).
	 $$
	We consider the cases
	$|\alpha| = 1$ and $|\alpha| = 2$ separately.

	In the first case by Lemma \ref{lemma 11.5} $(ii)$
	 with $k = d$ and $m = d+1$ we get
	$$
		\sum_{   ||y|| \geq r }   ||y||^d \, \delta^{\tilde \alpha}_1   G (0, n, y) \leq N r^{-1},
	$$
	and this proves the assertion.

	Next, in case  
	$
		|\alpha| = 2
	$
	the claim follows from  Lemma \ref{lemma 11.5} $(ii)$ with $k = d$
	and 
	$
		m = d + 2.
	$
	\end{proof}

	\begin{lemma}
			\label{lemma 5.4}
	Invoke the assumptions of Theorem \ref{theorem 3.3}.
	Let  $r, m \in \bN$,
	$
		(n, x) \in Q_r
	$
	and $j \in \bZ$ be a number such that  $j \leq n -r^2$.
	Let
	$
		\eta_1, \ldots, \eta_m
	$
	be vectors in $\bZ^d$ and
	denote 
	$$
		L = \delta_{\eta_1, 1} \ldots \delta_{\eta_m, 1},
	$$
	 $$
	  	u (j, n, x) = \sum_{y \in \bZ^d} L G (j, n, y) g (j, x - y). 
	 $$ 
	Then, 
	 $$
		|u (j, n, x)| \leq N   (n-j)^{ -  m/2 } \,  \cM_x g (j, 0),
	 $$
	where $N$ depends only on 
	$
		d, \delta_1^{*}, \delta_2^{*},
	$ 
	  $
		m, \rho, \Lambda,
	 $
	  $
			\eta_1, \ldots, \eta_m.
	   $
	\end{lemma}

	\begin{proof}
	As in the proof of Lemma \ref{lemma 5.3} 
	we may assume that $j = 0$, 
	and 
	$
		n \geq r^2.
	$
	Denote
	$
	\alpha =   (1, \ldots, 1).
	$
	Then, by Lemma \ref{lemma 9.5} $(ii)$ we have
	$$
		|u (0, n, x)| 
	 \leq  \sum_{y  \in \bZ^d}  | \delta^{  \alpha}_1 L G (0,n, y)| 
			 \, \sum_{  z \in  C_{||y||} }  |g (0, x-z)|,
	  $$
		By the fact that
	$
		C_{ ||y||  } (x)  \subset	C_{||y||  +  r},  
	$
	we have
	 $$
		 |u (0, n, x)| \leq  \sum_{ y  \in \bZ^d }  |\delta^{   \alpha }_1  L G (0, n, y)| \, \sum_{ y \in C_{||y||  + r}  }  |g (0, y)|
	 $$
	  $$
		\leq N (d)  \cM_x g (0, 0) \, \sum_{ y \in \bZ^d } (||y||^d +  r^d)\,   |\delta^{  \alpha }_1  L G (0, n, y)|.
	 $$
	Next, we apply Theorem \ref{lemma 11.5} $(i)$ twice
	with $m$ replaced by $m+d$
	and either $k = d$ or $k = 0$.
	We have
	$$
		|u (0, n, x)|   \leq N (n^{ - m/2} + r^d n^{ - (d + m)/2 }) \cM_x g (0, 0).
	$$
	Since $n \geq r^2$, we obtain
	$$
		|u (0, n, x)| \leq N n^{-m/2} \, \cM_x g (0, 0),
	$$
	and this finishes the proof of the lemma.
	\end{proof}

	\begin{lemma}
				\label{lemma 5.6}
	Invoke the assumptions and notations of Theorem \ref{theorem 4.2}.
	Then, for any $t \geq [0, T - \gamma]$,
	$$
		I: =  \int_t^{t+\gamma} (||D_i (u (r, \cdot) - u (t, \cdot))||^p_{n, p; \, h} + || D_{i j} (u (r, \cdot) - u (t, \cdot)) ||^p_{n, p; \, h} \, dr  
	$$
	 $$	
	 \leq N (p, d, \hat K_2) \gamma^p \int_t^{t+\gamma} (|| u (r, \cdot)||^p_{s+2, p; \, h} +  || f (r, \cdot)||^p_{s + 2, p; \, h} \, dr).
	$$

	\end{lemma}

	\begin{proof}
	By Remark \ref{remark 2.2} $u$ is the classical solution of Eq. \eqref{1.1}.
	Hence, for any 
	$
		r \in [t, t+\gamma],
	$
	 and any $x \in h \bZ^d$
	$$
		u (r, x) -  u (t, x)  = \int_r^t (a^{i j} (\tau, x) D_{i j} u (\tau, x) + b^i (t, x) D_i u (\tau, x) + f (\tau, x)) \, d\tau.
	$$
	By Lemma \ref{lemma 9.6}
	$$
		I \leq p^{-1} \gamma^p  \int_t^{ t + \gamma} (|| D_i D_t u (r, \cdot)||^p_{n, p; \, h} + || D_{i j} D_t  u (r, \cdot) ||^p_{n, p; \, h} \, dr).
	$$
	Next, by Lemma \ref{lemma 9.2} 
	$$
		I \leq N  (p)   \gamma^p (I_1 + I_2), 
	$$
	where
	 $$
	  	 I_1 =    \sum_{||\alpha|| \leq 4}  \sup_{t \in [t, t+\gamma]} (||a^{i j} (t, \cdot)||_{ C^{s+2} } + ||b^i (t, \cdot)||_{ C^{s+2} })   \int_t^{ t + \gamma} (|| D^{\alpha}_x u (r, \cdot)||^p_{n, p; \, h}  \, dr),
	 $$
	  $$
		I_2 =  \int_t^{t+\gamma} || D_{i j} f (r, \cdot) ||^p_{n, p; \, h}\, dr.
	  $$		
	By Assumption \ref{assumption 4.3.1} and Lemma \ref{lemma 9.2}
	$$
		I_1 \leq N (p, d, \hat K_2) \int_t^{t+\gamma} || u (r, \cdot)||^p_{ H^{s+2}_p} \, dr.
	$$
	Using  Lemma \ref{lemma 9.2} again, we get
	$$
		I_2 \leq N (p, d) \int_t^{t+\gamma} || f (r, \cdot)||^p_{ H^{s+2}_p } \, dr.
	$$
	The claim is proved.
	\end{proof}
	
	\begin{lemma}
				\label{lemma 5.5}
	Invoke the assumptions and notations of
	Theorem \ref{theorem 3.1} except
	Assumption \ref{assumption 3.4.1}.
	Set $v_0 \equiv u_0$.
	For 
	$
		m \in \bN
	$
	 consider the following SPDEs:
	\begin{equation}
				\label{8.1} 
	\begin{aligned}
		dv_{m} (t, x) &=   [a^{i j} (t, x) D_{i j} v_{m} (t, x) + b^i (t, x) D_i v_{m} (t, x) + f (t, x)] \, dt\\ 
																						 &+ (g^l (t_{m-1}, x) + \nu^l (t_{m-1}, x) v_{ m - 1 } (t_{m-1}, x)) (w^l (t_m) - w^l (t_{m-1})),\\
						& \quad v_{m} (t_{m-1}, x) = v_{m - 1} (t_{m-1}, x), \, t\in [t_{m-1}, t_{m}],
	 \end{aligned}
	\end{equation}
	where $t_m = m \gamma$.
	Then, the following assertions hold.

	$(i)$
	Denote
	$$
		\cI_m = E \int_{0}^{ t_m } || f (r, \cdot) ||^p_{ H^{s+1}_p } \, dr + \gamma \sum_{r  = 0}^{m} E  || g^l (t_r, \cdot) ||^p_{ H^{s+2}_p } + E || u_0 ||^p_{ H^{s + 3  - 2/p}_p }.
	$$
	Then, 
	for any
	$\varepsilon \in (0, 1)$,
	and any
	$
	   m \in \{1, \ldots, T/\gamma\},
	$
	one has
	 $
	  v_m \in \cH^{s+3 - \varepsilon}_p (t_{m-1}, t_m).
	 $
	In addition
	\begin{equation}
					\label{8.1.1}
		|| v_m ||_{\cH^{s+3 - \varepsilon}_p (t_{m-1}, t_m)} \leq N \cI_m,
	\end{equation}
	and
	\begin{equation}
					\label{8.4}
		E || v_m ||^p_{ C^{\theta -1/p} ([t_{m-1}, t_m], H^{s+3 - \varepsilon -  2\mu}_p)} \leq N \cI_m,
	\end{equation}
		where $N$  depends only on $p, d, d_0, T,  \delta_1, \delta_2, s, K_2, \eta$.

	 Denote
	$$ \cJ_m = \cI_m +    E ||g ||^p_{ C^{\theta - 1/p} ([0, T], H^{s+2}_p) }  
		   + E  \int_{0}^{t_m} ||  g (r, \cdot) ||^p_{  H^{s+2}_p } \, dr.
	$$
	$$
		(ii) \, E \max_{m = 1, \ldots, T/\gamma } \sup_{t \in [t_{m-1}, t_m]  } || v_m (t, \cdot) - v (t, \cdot) ||^p_{  H^{s + 2}_p } \leq N \gamma^{\theta p - 1} \cJ_m,
	$$
	
	\end{lemma}

	\begin{proof}
	For the sake of convenience, assume that $T = 1$.

	$(i)$
	We may assume that $\varepsilon  < 1 - 2/p$.
		Denote 
	$
		s_m =  s + 3  -   m \gamma \varepsilon.
	$ 	
	Note that, for any $m \in \{1, \ldots, \gamma^{-1}\}$, 
	we have 
	$
		s_m  - 2/p \geq s+2.
	$
	This will be used in the sequel.	

	First, note that   
	$$
	  g^l (0, x) I_{t \in [0,  t_1]} \in \bH^{s+2}_p (0, t_1).
	$$
	Second,
	$$
	     u_0 (x) I_{  t \in [0, t_1] } \in \bH^{s+2}_p (0, t_1)
	$$
	and, by Lemma 5.2 of \cite{Kr_99} combined with 
	Assumption \ref{assumption 3.3} we have
	$$
		||\nu^l (0, \cdot) u_0 (\cdot) ||_{   \bH^{s+2}_p (0, t_1)}
		 \leq N  \gamma||\nu^l||_{C^{s+2 +\eta} }  || u_0 ||_{ H^{s+2}_p }
		\leq N || u_0 ||_{ H^{s+3 - 2/p}_p }.
	 $$
	Further, for any $x \in \bR^d$,
	$$
		(g^l (0, x) + \nu^l (0, x) u_0 (x) )w^l (t_{1})
															= \int_0^{t_1} (g^l (0, x) + \nu^l (0, x) u_0 (x)) \, dw^l (t).
	$$
	Then,  by Theorem 5.1 of \cite{Kr_99} there exists a unique solution  
	$
		v_1 \in   \cH^{s+3}_p (0, t_1),
	$
	and
	$$
		|| v_1 ||_{\cH^{s+3}_p (0, t_1)} \leq N \cI_1. 
	$$
	
	Next,  by 
	 Remark \ref{remark 2.1}
	$
		v_1 (t_1, \cdot)) \in H^{s_1 - 2/p }_p.
	$ 
	Again, in Eq. \eqref{8.1} with $m = 2$
	 we  rewrite the term involving 
	$
		 (w^k (t_{2}) - w^k (t_1))
	$
	as a stochastic integral
	and 
	 repeat the  argument of the previous paragraph.
	This time, 
	since 
	$
		s_1 - 2/p \geq s+2,
	$
	 by Lemma 5.2 of \cite{Kr_99}
	one has
	  $$
		||\nu^l (t_1, x) v_1 (t_1, x) ||_{ \bH^{s+2}_p (t_1, t_2) } \leq N \gamma || v_1 (t_1, \cdot) ||_{ H^{s+2}_p }
		\leq N ||v_1 (t_1, \cdot) ||^p_{ H^{s_1 -2/p}_p }.
	   $$
	 Then, by Theorem 5.1 of \cite{Kr_99}
	 it follows that the equation
	\eqref{8.1} (with $m = 2$) 
	 has a unique solution 
	$
		v_{2} \in \cH^{s_1}_p (t_1, t_2).
	$ 
	Iterating the above argument, we conclude that,
	for  
	$
		m = \{1, \ldots, \gamma^{-1}\},
	$
	we have 
	$
		v_{m} \in \cH^{ s_{m-1} }_p (t_{m-1}, t_{m}),
	$
	and
	\begin{equation}
				\label{8.3}
	 \begin{aligned}
	 	 &  || v_m ||^p_{   \cH^{ s_{m-1} }_p (t_{m-1}, t_{m})  } 
																		 \leq N  E (\int_{t_{m-1}}^{t_m} || f (r, \cdot) ||^p_{ H^{s+1}_p } \, dr \\
			&	+  \gamma E || g^l (t_{m-1}, \cdot) ||^p_{  H^{s+2}_p }
																	+  E || v_{m-1} (t_{m-1}, \cdot) ||^p_{ H^{s_{m - 1} - 2/p}_p }).
	 \end{aligned}
	 \end{equation}
	By Remark \ref{remark 2.1}, for any function $\xi  \in  \cH^{s_{m-2}}_p (t_{m-2}, t_{m-1}),$
	\begin{equation}
					\label{8.3.1}
									E || \xi (t_{m-1}, \cdot) ||^p_{ H^{s_{m - 1} - 2/p}_p }
														 \leq N ||\xi ||^p_{ \cH^{s_{m-2}}_p (t_{m-2}, t_{m-1}) }.
	\end{equation}
	Combining \eqref{8.3} and \eqref{8.3.1}, and iterating these estimates, we obtain
	 \eqref{8.1.1}.
	From this and Remark \ref{remark 2.1} we conclude that \eqref{8.4} holds..

	$(ii)$
	\textit{Step 1.}
	For
	$
		t \in [t_{m-1}, t_m],
	$
	and 
	$x \in \bR^d$, 
	denote
	$
		z_m (t) = v_m (t, x) - v (t, x).
	$
	Note that 
	$
		z_{0} \equiv 0, 
	$
	and, for $m \geq 1$, 
	$z_m$
	satisfies the following SPDE:
	\begin{equation}
				\label{8.4.1}
	 \begin{aligned}
		dz_m (t, x) &=   [a^{i j} (t, x) D_{i j} z_m (t, x) + b^i (t, x) D_i z_m (t, x)] \, dt\\ 
		 &+ ([g^l (t, x ) - g^l (t_{m-1}, x)]\, dw^l (t) \\
		&+ (\nu^l (t, x) v (t, x) - \nu^l (t_{m-1}, x) v_{m} (t_{m-1}, x)) \, dw^l (t)\\
				& \quad z_{m} (t_{m-1}, x) = z_{ m - 1 } (t_{m-1}, x), \, t \in [t_{m-1}, t_{m}].
	 \end{aligned}
	\end{equation}

	Denote
	$$
		\hat \cJ_m  = \cI_m +  E ||g ||^p_{ C^{\theta - 1/p} ([0, T], H^{s+2}_p) }  + \int_0^{t_m} || v (r, \cdot) ||^p_{ H^{s+2}_p } \, dr.
	$$
	First, we will prove that, for any $\varepsilon > 0$, 
	\begin{equation}
				\label{8.5}
		E    \sup_{r \in [t_{m-1}, t_m]}   || z_m (r, \cdot) ||^p_{  H^{s+3 - \varepsilon -  2\mu}_p }   \leq N \gamma^{\theta p - 1} \hat \cJ_m.
	\end{equation}
	We use the  iterative argument from Step 1.
	Fix some 
	$
		t \in [t_{m-1}, t_m].
	$
	By Theorem 5.1 of \cite{Kr_99}, for $m \geq 1$,
	\begin{equation}
					\label{8.6}
		E   || z_m (t, \cdot) ||^p_{  \cH^{s_{m-1}   }_p (t_{m-1}, t) }  
	 		\leq N \sum_{i = 1}^4 J_i (t)  + N E  || z_{ m - 1 } (t_{m-1}, \cdot) ||^p_{ H^{ s_{m-1} - 2/p}_p },
	 \end{equation}
	where
	 $$
		J_1 (t) =     E \int_{t_{m-1}}^{t} || g^l (r, \cdot) - g^l (t_{m-1}, \cdot)||^p_{ H^{ s+2}_p } \, dr,
	 $$
	  $$	
		J_2 (t) =  E \int_{t_{m-1}}^{t} ||(\nu^l (r, \cdot) -  \nu^l (t_{m-1}, \cdot)) v (r, \cdot) ||^p_{ H^{s+2}_p } \, dr,
	  $$
	   $$
		J_3 (t)  = E \int_{t_{m-1}}^{t} ||\nu^l (t_{m-1}, \cdot) (v (r, \cdot) - v_m (r, \cdot))||^p_{ H^{ s+2}_p } \, dr,
	   $$
	    $$
		J_4 (t) = E \int_{t_{m-1}}^{t} ||\nu^l (t_{m-1}, \cdot) (v_m (r, \cdot) - v_{m - 1} (t_{m-1}, \cdot))||^p_{ H^{ s+2}_p } \, dr.
	    $$	

	By Assumption \ref{assumption 3.6}
	\begin{equation}
				\label{8.7}
		 J_1 (t) \leq    \gamma^{ \theta p }  E ||g||^p_{C^{\theta - 1/p} ([0, T], H^{s+2}_p)}.
	\end{equation}
	Next, due to   Lemma 5.2 of \cite{Kr_99} and Assumption \ref{assumption 3.5}
	\begin{equation}
				\label{8.8}
	\begin{aligned} 
		J_{2} (t) & \leq N  \gamma^{\theta p - 1} E ||\nu^l ||^p_{C^{\theta - 1/p} ([0, T], H^{s+2}_p)}   \int_{t_{m-1}}^t   || v (r, \cdot) ||^p_{ s+2, p }\, dr\\
		&\leq N \gamma^{\theta p - 1} E  \int_{t_{m-1}}^t || v (r, \cdot) ||^p_{ H^{s+2}_p }\, dr.
	\end{aligned}
	 \end{equation}
	
	Next, repeating the above argument and using  Assumption \ref{assumption 3.4.2}, 
	we get
	\begin{equation}
				\label{8.9} 
		J_3 (t) \leq   N E \int_{t_{m-1}}^{t} ||  z_m (r, \cdot) ||^p_{ H^{s+2}_p} \, dr 
	 \end{equation}
	 Further, by the fact that
	$
		v_m (t_{m-1}, x) = v_m (t_{m-1}, x)
	$
	$$
		J_4 (t) \leq E  \int_{t_{m-1}}^t || v_m (r, \cdot) - v_m (t_{m-1}, \cdot)||^p_{ H^{s+2}_p } \, dr \leq N \gamma^{\theta p} E || v_m ||^p_{ C^{\theta -  1/p} ([0, T], H^{s+2}_p) }.
	$$	
	To estimate the last term we fix some $\varepsilon < 1 - 2\beta$.
	In that case 
	$
		s +2 < s + 3 - \varepsilon - 2\mu,
	$ 
	and by \eqref{8.4}
	 we get
	\begin{equation}
				\label{8.10}
		 J_4 (t) \leq    N \gamma^{\theta p } \cI_m.
	\end{equation}

	Combining the estimates \eqref{8.6} - \eqref{8.10} with Remark \ref{remark 2.1},
	for any 
	$
		m \in \{ 2, \ldots, \gamma^{-1}\},
	$ 
	and any $t \in [t_{m-1}, t_m]$,
	 we get
	\begin{equation}
				\label{8.11}
	\begin{aligned}
	&	E || z_m (t, \cdot) ||^p_{ H^{ s+2 }_p } \leq  N  || z_m  ||^p_{  \cH^{s_{m-1}}_p (t_{m-1}, t)  }\\
		& \leq  N     \gamma^{\alpha p} (  \cI_m +  E ||g ||^p_{ C^{\theta - 1/p} ([0, T], H^{s+2}_p) }     \\ 
	&    N \gamma^{\alpha p - 1} E  \int_{t_{m-1}}^t || v (r, \cdot) ||^p_{ H^{s+2}_p }\, dr   + E ||z_{m-1} ||^p_{ \cH^{ s_{m - 1} - 2/p }_p (t_{m-1}, t_m)  }\\
	&
		+ N  E \int_{t_{m-1}}^{t} ||  z_{m} (r, \cdot) ||^p_{  H^{s+2}_p } \, dr,
	\end{aligned}
	\end{equation}
	where $N$ is independent of  $t$.
	By Gronwall's lemma
	 we may get rid of the integral term containing $z_m$
	 on the right hand side of we \eqref{8.11}.
	By this and \eqref{8.3.1}
	\begin{align*}
		E  & || z_m  ||^p_{  \cH^{s_{m-1}}_p (t_{m-1}, t_m)  } 
		 \leq  N \gamma^{\theta p}  (\cI_m +  E ||g ||^p_{ C^{\theta - 1/p} ([0, T], H^{s+2}_p) })  \\
 		&   N \gamma^{\theta p - 1} E  \int_{t_{m-1}}^{t_m} || v (r, \cdot) ||^p_{ H^{s+2}_p }\, dr   
			+   N   E ||z_{m-1} ||^p_{ \cH^{ s_{m - 2}  }_p (t_{m-1}, t_m)  }.
	 \end{align*}
	Iterating this estimate
	we obtain that, 
	for all 
	$
		m \in \{1, \ldots, \gamma^{-1}\}
	$,
	$$
		 || z_m  ||^p_{  \cH^{s_{m-1}}_p (t_{m-1}, t_m)  }
	  \leq  N \gamma^{\theta p - 1} \hat \cJ_m.
	$$
	Due to Remark \ref{remark 2.1}  this implies \eqref{8.5}.

	\textit{Step 2.}
	Denote 
	$$
	      \tilde g^l (t, x) =  g^l (t_{m-1}, x)  - g^l (t, x ) +
	$$ 
	 $$
		+ \nu^l (t_{m-1}, x) v_{m-1} (t_{m-1}, x) -  \nu^l (t, x) v (t, x), \,  t \in [t_{m-1}, t_m], x \in \bR^d,
	 $$
	and consider the equation
	\begin{equation}
				\label{8.17}
	  \begin{aligned}
		du (t, x) & = [a^{i j} (t, x) D_{i j} u (t, x)\\
		& + b^i (t, x) D_i u (t, x)]\, dt + \tilde g^l (t, x) \, dw (t), \quad u (0, x) = 0.
	   \end{aligned}
	\end{equation}
	Note that 
	$$
		E \int_0^1 ||\tilde g^l (r, \cdot) ||^p_{ H^{s+2}_p }  \, dr  
									\leq N \gamma \sum_{m = 0}^{1/\gamma  } \sum_{i = 1}^4 J_{i} (t_m).
	$$
	Combining  \eqref{8.7} - \eqref{8.10}, we get
	 $$
		E \int_0^1 ||\tilde g^l (r, \cdot) ||^p_{  H^{s+2}_p } \, dr  \leq N  \gamma^{\theta p  - 1} \hat \cJ_{1/\gamma}.
	 $$
	Then, by Theorem 5.1 of \cite{Kr_99} 
	there exists a unique solution 
	$
		\tilde z \in \cH^{s+3}_p (0, T)
	$
	 of  \eqref{8.17}, and
	$$
		|| \tilde z ||^p_{ \cH^{s+3}_p (0, 1)}
										 \leq  N \gamma^{ \theta p  -  1 } \hat \cJ_{1/\gamma}.
	$$
	Observe that  by Theorem 5.1 of \cite{Kr_99}  and Remark \ref{remark 2.1}
	$
		z_1  \equiv  \tilde z,
	$
	for $\Omega \times [0, t_1] \times \bR^d$. 
	By induction
	one can easily show that
	$
		z_m  = \tilde z,
	$
	for 
	$
	   \Omega \times   [t_{m-1}, t_m] \times \bR^d.
	$ 
	Finally, by Remark \ref{remark 2.1} and Theorem \ref{theorem 3.5}
	$$
		E \int_0^{t_m} || v (r, \cdot)||^p_{ H^{s+2}_p } \, dr \leq  \cJ_m,
	$$
	and then, $\hat \cJ_m \leq \cJ_m$, for any $m$.
	Now the assertion follows what was just said and Remark \ref{remark 2.1}.

	\end{proof}

	\mysection{Proof of Theorem \ref{theorem 3.3} and Theorem \ref{theorem 4.3}.}
											\label{section 6}
	
	For a function  $h$ on $\bZ^{d+1}$
	 and any set 
	$
		A \subset \bZ^{d+1}
	$
	 denote
	$$
		  h_{A} =|A|^{-1} \sum_{ (t, x) \in A  }  |h (t, x)|.	     
	 $$
	For any
	 $(t_0, x_0) \in \bZ^{d+1}$,
	denote
	    $$	 	
		   h^{\#} (t_0, x_0) =  \sup_{r \in \bN} |Q_a|^{-1} \sum_{ (t, x) \in Q_r (t_0, x_0) } |h (t, x) - h_{ Q_r (t_0, x_0) }|.
	    $$	
	where $Q_r$ and $C_r$ are defined in Section \ref{section 5}.

	\textit{Proof of Theorem \ref{theorem 3.3}.}
	It suffices to prove the theorem for $k = 1$.

	Denote
	$$
		\cG g (n, x) =  (\sum_{ j = -\infty}^{n - 1} |\delta_{e_1, 1} G (j, n, \cdot) \ast_1  g (j, \cdot) (x)|^2)^{1/2}.
	$$

	\textit{Step 1.} Here we show that	
	 for any 
	$
		(t, x) \in \bZ^{d+1}
	$
	 the following inequality holds:
	\begin{equation}
				\label{6.1}
		(\cG g)^{\#} (t, x) \leq N (d, d_1,  \Lambda, \rho, \delta_1^{*}, \delta_2^{*}) (\cM_t \cM_x |g|^2 (t, x))^{1/2}.
	\end{equation}

	First, we note that  by Cauchy-Schwartz inequality, it suffices to prove that, 
	 for any $r \geq 1$  
	\begin{equation}
				\label{6.2}
		\sum_{ (n, y) \in Q_r (t, x)} |\cG g (n, y)  - (\cG  g)_{Q_r (t, x)}|^2 \leq N  |Q_r|  \cM_t \cM_x |g|^2 (t, x).
	\end{equation}
	By the fact that  convolution commutes with the shift
	we may assume $x = 0$.
	Further, by replacing  
	$\lambda_{k} (a(\cdot))$
	 with 
	$\lambda_{k} (a (\cdot + t))$
	 we may assume that $t = 0$.
	Hence, it suffices to prove \eqref{6.1} with $(t, x) = (0, 0)$.

	Next, fix some $r \geq 1$,
	and denote
	$$
		g_1 (t, x) = g (t, x)   I_{ t \in (-2r^2, 2r^2)} I_{ ||x|| \leq  3 r},
	$$
	 $$
		g_2 (t, x) = g (t, x) I_{ t \in (-2r^2, 2r^2)} I_{ ||x || >  3 r },
	 $$
	  $$
		g_3 (t, x) =  g (t, x) I_{t \geq 2r^2},
	 $$
	 $$
		g_4 (t, x) = g (t, x) I_{t \leq -2r^2},
	  $$
	 $$
		I_k =   \sum_{ (t, x) \in Q_r } |\cG g_k (t, x)|^2, \, k = 1, 2, 3,
	$$
	 $$
		I_4 = \sum_{ (t, x) \in Q_r } |\cG g_4 (t, x)  - \cG g_4 (0, 0)|^2.
	 $$
	Observe that, since $(\cG g)_{Q_r}$ approximates $g$ with the least mean-square error, we have
	\begin{equation}
				\label{6.3}
		\sum_{ (t, x) \in Q_r } |\cG g (t, x)  - (\cG  g)_{Q_r }|^2   
		 \leq   \sum_{ (t, x) \in Q_r } |\cG g (t, x)  - \cG g (0, 0)|^2
															  \leq N \sum_{k = 1}^4 I_k.
	 \end{equation}

	\textit{Estimate of $I_1$.}
	By Lemma \ref{lemma 5.1}
	\begin{equation}
	  			\label{6.4}
	 \begin{aligned}
		I_1  &= \sum_{(t, x) \in \bZ^{d+1}} |\cG g_1 (t, x)|^2 
		\leq N \sum_{(t, x) \in \bZ^{d+1}} |g_1 (t, x)|^2 \\
		& =   N \sum_{(t, x) \in Q_{2r}} |g (t, x)|^2
																				\leq N |Q_{r}| \cM_t \cM_x |g|^2 (0, 0).
	 \end{aligned}
	 \end{equation}

	\textit{Estimate of $I_2$.}
	Fix any $(n, x) \in Q_r$, 
	$
		j \leq n  - 1.
	$
	By
	 Lemma \ref{lemma 5.3} 
	$$
		 |\delta_{e_1, 1} G (j, n, \cdot) \ast_1  g_2 (j, \cdot) (x)| \leq N r^{-1} \cM_x g (j, 0).
	$$
	By Cauchy Schwartz inequality
	$
		 |\cM_x g (j, 0)|^2 \leq  \cM_x |g|^2 (j, 0),
	$
	and, then, since $g (t, \cdot)$ vanishes  for $|t| > 2r^2$, we get
	 $$
		  |\cG g_2 (n, x)|^2 
					\leq N r^{-2} \, \sum_{j= -2r^2}^{2r^2} \cM_x |g|^2 (j, 0) 
																		\leq N \cM_t \cM_x |g|^2 (0, 0).
	 $$
	Hence,
	\begin{equation}
				\label{6.5}
			I_2 \leq N |Q_r| \, \cM_t \cM_x g (0, 0).
	\end{equation}

	\textit{Estimate of $I_3$.}
	Observe that, for 
	$
		n \in (-r^2, r^2),
	$ 
	 we have 
	$
		\cG  g_3 (n, \cdot) = 0,
	$
	and, then, 
	\begin{equation}
				\label{6.6}
		I_3 = 0.
	\end{equation}

	\textit{Estimate of $I_4.$}
	Since $r^2$ is the diameter of $Q_r$, we have
	\begin{equation}
				\label{6.7}
		I_4 \leq r^2 |Q_r| \max_{ (n, x) \in Q_r} (| \cG g_4 (n+1, x) - \cG g_4 (n, x)|^2  
																+ \sum_{k = 1}^d |\delta_{e_k, 1} \cG g_4 (n, x)|^2).
	\end{equation}
	Fix any 
	$
		(n, x) \in Q_r.
	$
	By Minkowski inequality
	$$
		|\delta_{e_k, 1} \cG g_4 (n, x)|^2 \leq  \sum_{j = -  \infty}^{-2r^2} |\delta_{e_k, 1} \delta_{e_1, 1} G (j, n, \cdot) g (j, \cdot) (x)|^2. 
	$$
	Observe that, 
	for 
	$j \leq -2r^2$, 
	we have
	$n - j \geq r^2$
	so that Lemma \ref{lemma 5.4} is applicable.
	We use this lemma with $m = 2$
	and 
	$
		\eta_1 = e_1, \eta_2  = e_k.
	$
	 We obtain 
	$$
		|\delta_{e_k, 1} \cG g_4 (n, x)|^2
		\leq N \sum_{ j = -\infty}^{-2r^2} (n - j)^{-2} \,  \cM_x |g|^2 (j, 0).  
	$$
	Since
	 $n \geq  -r^2$,
	 we may replace 
	$
		(n - j)^{-2}
	$
	 by
	 $
		(-r^2 - j)^{-2}.
	$ 
	Then we change the index of summation as follows: 
	$
	   j \to -r^2 - j.
	$
	 By this and Cauchy-Schwartz inequality we get
	$$
		|\delta_{e_k, 1} \cG g_4 (n, x)|^2
		 \leq N \sum_{ j = r^2}^{\infty} j^{-2} \,  \cM_x |g|^2 (-r^2-j, 0).  
	$$
	Now we use Lemma \ref{lemma 9.5} $(i)$ with 
	$
		\hat f (j) = j^{-2},
	$
	 $
		\hat g (j) = \cM_x |g|^2 (-r^2 - j, 0),
	 $
	$
		t_0 = r^2 = t_1,
	$
	 $
		t_2 > r^2
	$
	and we 
	 take a limit as $t_2 \to \infty$.
	The boundary term in \eqref{9.5.0} vanishes
	because 
	$
		\hat g (j), j \in \bZ
	$
	has compact support.
	  By what was just said the last sum is bounded above by
	$$
		N \sum_{j = r^2}^{\infty}  j^{-3} \sum_{m = r^2}^j \cM_x |g|^2 (-r^2 - m, 0)
	     $$
	 Clearly, in the above sum 
	$
		[-r^2 - m, 0] \subset [-2j, 0].
	$
	Then,  
	we have
	  \begin{equation}
				\label{6.8}
	   \begin{aligned}
		|\delta_{e_k, 1} \cG g_4 (n, x)|^2&  \leq N  (\sum_{j = r^2}^{\infty}  j^{-2}) \cM_t \cM_x |g|^2 (0, 0)  	\\																	
	&\leq	 N r^{-2} \,  \cM_t \cM_x |g|^2 (0, 0).
	  \end{aligned}
	 \end{equation}
	
	Next, we estimate the first term on the right hand side of \eqref{6.7}.
	Fix any $(n, x) \in Q_r$ and note that by Minkowski inequality
	$$
		 |\cG g_4 (n+1, x) - \cG (n, x)|^2 
			\leq  \sum_{j = -\infty}^{n-1}  |\delta_{e_1, 1} (G (j, n+1, \cdot) - G (j, n, \cdot)) \ast_1 g (j, \cdot) (x)|^2.
	$$
	Since $G (j, n, \cdot)$ satisfies Eq. \eqref{3.4},
	 by Assumption \ref{assumption 3.2}
	 we have
	$$
		|\delta_{e_1, 1} (G (j, n+1, \cdot) - G (j, n, \cdot)) \ast_1 g (j, \cdot) (x)|
	$$
	 $$
		\leq \delta_2^{*} \sum_{k = 1}^{d_1} |\Delta_{l_k, 1} \delta_{e_1, 1} G (j, n, \cdot) \ast_1 g (j, \cdot) (x)|.
	 $$
	Observe that 
	$
		\Delta_{l_k, 1} = -\delta_{ - l_k, 1} \delta_{ l_k, 1}. 
	$
	Now we use Lemma \ref{lemma 5.4} with 
	$
		m  = 3, 
	$
	and
	$
	\eta_1 = - l_k, \eta_2 = l_k,  \eta_3  = e_1.
	$
	We get 
	$$
			 | \Delta_{l_k, 1} \delta_{e_1, 1} G (j, n, \cdot) \ast g (j, \cdot) (x)|^2 \leq N (n - j)^{-3} \,  \cM_x |g|^2 (j, 0).
	$$
	After that one repeats  the above argument  and obtains
	\begin{equation}
				\label{6.9}
		 |\cG g_4 (n+1, x) - \cG (n, x)|^2 	 \leq  N  r^{-2} \cM_t \cM_x |g|^2 (0, 0).
	\end{equation}
	By \eqref{6.7} - \eqref{6.9}  we get
	\begin{equation}
				\label{6.10}
		I_4 \leq N  \cM_t \cM_x |g|^2 (0, 0).
	\end{equation}
	
	Next,
	combining \eqref{6.3} - \eqref{6.6} with \eqref{6.10}, 
	we conclude that 
	 the estimate \eqref{6.2} holds with $t_0 = 0$, $x_0 = 0$,
	and, then \eqref{6.1} holds.
	
	Finally, we raise both sides of \eqref{6.1} to the power $p > 2$ and 
	sum with respect to 
	$
		(t, x) \in \bZ^{d+1}.
	$
	By Theorem \ref{theorem 12.2} and Corollary \ref{corollary 12.1} with $p/2 > 1$
	 we obtain
	$$
		\sum_{ (t, x) \in \bZ^{d+1} } |\cG g (t, x)|^p \leq N \sum_{ (t, x) \in \bZ^{d+1} } |\cM_t \cM_x |g|^2 (t, x)|^{p/2}
	$$
	 $$
		\leq N \sum_{ (t, x) \in \bZ^{d+1} } |g (t, x)|^p.
	 $$

	\textit{Proof of Theorem \ref{theorem 4.3}.}

	\textit{Case $p \geq 2$.}
	Fix any multi-index $\alpha$ of order $2$ and denote
	$$
		\cR g (t, x) = \sum_{j = -\infty}^{t-1}  \delta^{\alpha}_1 G (j, t, \cdot) \ast_1 g (j, \cdot) (x), \,  t\in \bZ, x  \in \bZ^d.
	$$ 
	By Lemma \ref{lemma 5.2} the theorem holds for $p = 2$, and hence, we may assume that $p > 2$.

	We follow the argument of Theorem \ref{theorem 3.3} very closely
	making only necessary changes.
	First, we  prove that, 
	for any 
	$
		(t, x) \in \bZ^{d+1}
	$
	$$
			(\cR g)^{\#} (t, x)  \leq N  \,   (\cM_t \cM_x |g|^2 (t, x))^{1/2}.
	$$ 
	From the proof of Theorem \ref{theorem 3.3}
	 we know that this implies the assertion of the theorem for $p > 2$.
	
	By Cauchy-Schwartz  inequality it suffices to show that
	for any 
	$
		(t, x) \in \bZ^{d+1},
	$
	and $r \geq 1$,
	\begin{equation}
				\label{6.11}
		\sum_{ (k,  y) \in Q_r (t, x) } |\cR g (k, y)  - (\cR  g)_{Q_r (t, x)}|^2 \leq N  |Q_r| \,  \cM_t \cM_x  |g|^2 (t, x).
	\end{equation}
	As in the proof of Theorem \ref{theorem 3.3} we may assume that $t = 0, x = 0$.
	
	Next, let
	$
		g_i, i = 1, \ldots, 4
	$ 
	be the functions defined  in the proof of Theorem \ref{theorem 3.3} 
	and denote 
	$$
		I_i = \sum_{(n, x) \in Q_r} |\cR g_i (n, x)|^2, \, i = 1, 2, 3,
	$$
	 $$
		I_4 = \sum_{ (n, x) \in Q_r } |\cR g_4 (n, x)  - \cR g_4 (0, 0)|^2.
	 $$
	Then, we have
	\begin{equation}
				\label{6.11.1}
		\sum_{ (n, x) \in Q_r} |\cR g (n, x)  - (\cR  g)_{Q_r}|^2 \leq \sum_{i=1}^4 I_i.
	\end{equation}
	Clearly, $I_3 = 0$ so that we only need to consider $I_1, I_2, I_4$.
	
	\textit{Estimate of $I_1$.}  
	By Lemma \ref{lemma 5.2}
	\begin{equation}
				\label{6.12}
		I_1 \leq N \sum_{(t, x) \in Q_{2r} }  |g (t, x)|^2 \leq N |Q_r|\,  \cM_t \cM_x |g|^2 (0, 0).
	\end{equation}
	This time one needs to use Lemma \ref{lemma 5.2} instead of Lemma \ref{lemma 5.1}.

	\textit{Estimate of $I_2$.}
	Fix any $(n, x) \in Q_r$
	and 
	$
		j \leq n - 1.
	$
	By Lemma \ref{lemma 5.3}
	$$
			|\delta^{\alpha}_1 G (j, n, \cdot) \ast_1 g (j, \cdot) (x)| \leq N r^{-2} \, \cM_x g (j, 0).
	$$
	Then, we get 
	$$
		|\cR g_2 (n, x)| \leq  N   r^{-2}  \sum_{j = -2r^2}^{2r^2}  \,  \cM_x g (j, 0)
				 \leq N \cM_t \cM_x g (0, 0).
	$$
	By the fact that  
	$
		|\cM_t \cM_x g|^2 \leq  \cM_t \cM_x |g|^2
	$
	\begin{equation}
				\label{6.13}
		I_2 \leq  |Q_r|  \cM_t \cM_x  |g|^2 (0, 0).
	\end{equation}

	\textit{Estimate of $I_4$}.
	Observe that the inequality \eqref{6.7} holds
	with $\cG$ replaced by $\cR$.
	We fix any $(n, x) \in Q_r$
	and note that
	$$
		\delta_{e_k} \cR g (n, x) = \sum_{j = -\infty}^{-2r^2} \delta_{e_k, 1} \delta^{\alpha}_1 G (j, n, \cdot) \ast_1 g (j, \cdot) (x).
	$$
	
	Next, for each 
	$
		j \leq -2r^2,
	$
	we have $ n - j \geq r^2$.
	 Then, by
	Lemma \ref{lemma 5.4} $(i)$ with
	$
		m = 3
	$
	 we get
	$$
				|\delta_{e_k, 1} \cR g (n, x)| \leq N \sum_{j = -\infty}^{-2r^2} (n - j)^{-3/2} \cM_x g (j, 0).
	$$
	Repeating the corresponding part of the  proof of Theorem \ref{theorem 3.3}, we obtain
	 \begin{align*}
		& |\delta_{e_k} \cR g (n, x)| 
	\leq N  \sum_{j = r^2}^{\infty}  j^{-5/2} \sum_{m = -2j}^0 \cM_x g (m, 0) \\
		&\leq  N  (\sum_{j = r^2}^{\infty}  j^{-3/2}) \cM_t \cM_x g (0, 0) \leq N r^{-1}  \cM_t \cM_x g (0, 0),
	\end{align*}
	which implies
	\begin{equation}
				\label{6.14}
		|\delta_{e_k} \cR g (n, x)|^2  \leq N r^{-2} \cM_t \cM_x |g|^2 (0, 0).
	\end{equation}

	Next, since $G (j, n, \cdot)$ satisfies Eq. \eqref{3.4},
	and  Assumption \ref{assumption 3.2} hold,
	 we have
	$$
		|\cR g (n+1, x) - \cR g (n, x)| \leq  
			\delta_2^{*}  \sum_{k = 1}^{d_1} \sum_{j = - \infty}^{-2r^2} |\Delta_{l_k, 1} \delta^{\alpha}_1 G (j, n, \cdot) \ast_1 g (j, \cdot) (x)|.
	$$
	By  Lemma \ref{lemma 5.4} with $m = 4$ 
	$$
		|\cR g (n+1, x) - \cR  g (n, x)| \leq  N  \sum_{j = - \infty}^{-2r^2} (n-j)^{-2} \cM_x g (j, 0).
	$$
	As in the previous paragraph, this yields
	\begin{equation}
				\label{6.15}
		|\cR g (n+1, x) - \cR g (n, x)|^2  \leq N r^{-2} \cM_t \cM_x |g|^2 (0, 0).
	\end{equation}
	Combining \eqref{6.7} with \eqref{6.14} and \eqref{6.15} 
	we get
	\begin{equation}
				\label{6.16}
		I_4 \leq N |Q_r| \cM_t \cM_x |g|^2 (0, 0).
	\end{equation}
	Now \eqref{6.11} follows from \eqref{6.11.1} -  \eqref{6.13}, \eqref{6.16} and the fact that $I_3 = 0$.

	\textit{Case $p \in (1, 2).$}
	For functions $f$ and $g$ on $\bZ^{d+1}$ with compact support
	we denote 
	$$
		B (f, g) = \sum_{ (t, x) \in \bZ^{d+1} } (\delta^{\alpha}_1 \cR g) (t, x) f (t, x).
	$$ 
	We will show that
	\begin{equation}
				\label{6.17}
		| B (f, g)| \leq N   || g  ||_{ l_p (\bZ^{d+1}) }
				||f  ||_{  l_{p'} (\bZ^{d+1}) },
	\end{equation}
	where $p' =p/(p-1)$.
	After that by the standard  approximation argument 
	(see the last paragraph of the proof of Theorem \ref{theorem 3.4})
	the above inequality holds with 
	$
		g \in l_p (\bZ^{d+1})
	$
	 and 
	$
		f \in l_{p'} (\bZ^{d+1}).
	$
	By duality this will imply the claim.

	Next, for a function $h$ on $\bZ^d$ we denote 
	$A h (x) = h (-x)$.
	We note that
	$$
		B (f, g) = \sum_{ t  \in \bZ }  \sum_{x \in \bZ^d}  \sum_{j = -\infty}^{n-1}  A \delta^{\alpha}_1  G (j, t, \cdot) \ast_1 f (t,  \cdot) (x) g (j, x)
	$$
	 $$
		 = \sum_{j \in \bZ} \sum_{x \in \bZ^d}   g (j, x)  \bar \cR f (j, x),
	 $$
	where
	$$
	  	 \bar \cR f (j, x) =  \sum_{t = j+1}^{\infty}   A \delta^{\alpha}_1  G (j, t, \cdot)   \ast_1  f (t,  \cdot) (x).
	$$
	By Lemma \ref{lemma 11.1} $(i)$ we have 
	$
		A G (j, t, \cdot) = G (j, t, \cdot).
	$
	and, then,
	$$
		\bar \cR f (j, x) = \sum_{t = j+1}^{\infty}    \delta^{\alpha}_1  G (j, t, \cdot)   \ast_1  f (t,  \cdot) (x).
	$$
	By inspecting the  argument of this theorem for the case $p \geq 2$, 
	we conclude that 
	$$
		  || \bar \cR f||_{ l_p (\bZ^{d+1}) } \leq N    || f ||_{  l_p (\bZ^{d+1}) } .
	$$
	Finally, by H\"older's inequality and what was just proved we obtain \eqref{6.17}.
	\mysection{Proof of Theorem \ref{theorem 3.2}}

				\label{section 7}
	\begin{proof}
	
	The proof is divided into six steps.
	In the first four steps our goal is to show that, if $s  = 0$,
	then, for any $m \in \bN$
	 the  following estimate holds:
	\begin{equation}
				\label{7.0.1}
	\begin{aligned}
		E \sum_{r = 0}^{m} & [ v_{\gamma, h} (t_r, \cdot)]^p_{2, p; \, h} \leq N E \sum_{r = 0}^{m-1} ( ||f (t_r, \cdot)||^p_{p; \, h} \\ 
	& + [g(t_r, \cdot)]^p_{1, p; \, h}  +   || v_{\gamma, h} (t_r, \cdot) ||^p_{1, p; \, h}).
	\end{aligned}
	\end{equation}
	In Step 1 and Step 2 we prove \eqref{7.0.1} in case of  coefficients  independent of $x$.
	In Steps 3 and 4 we prove \eqref{7.0.1} 
	in  case of variable coefficients.
	In Step 5 by using the bootstrap method and interpolation inequality for discrete Sobolev norms,
	 we prove that 
	$$
		E \sum_{r = 0}^{m}  [ v_{\gamma, h} (t_r, \cdot)]^p_{n + 2, p; \, h} 
		\leq N \gamma \fR_n +  E \sum_{r = 0}^{m-1}   || v_{\gamma, h} (t_r, \cdot) ||^p_{p; \, h}.
	$$  
	Now \eqref{3.5} follows from the following estimate:
	\begin{equation}
				\label{7.0.2}
		E \sum_{r = 0}^{T/\gamma} || v_{\gamma, h} (t_r, \cdot)||^p_{p; \, h} \leq \fR_0,
	\end{equation}
	which is proved in Step 5 by using Lemma \ref{lemma 9.4}.
	In  Step 6 we prove \eqref{3.5.1}.
	
	For the sake of convenience, we assume $d_0 = 1$
	and denote 
	$
		g := g^1, \nu : = \nu^1.
	$

	 \textit{Step 1.
	Case 
	$n = 0$, 
	$b \equiv 0$,
	 $\nu  \equiv 0$,  
	and $a$  independent of $t$.
	}
	 Since $a$  is independent of $t$,
	 by Assumption \ref{assumption 3.2}  $(iv)$
	the same holds for
	 $
		\lambda_{k}, k  =1, \ldots, d_1.
	$
	Hence, $v_{\gamma, h}$ solves the equation
	\begin{equation}
				\label{7.4} 
	\begin{aligned}
		z (t+\gamma, x) &- z (t, x) =\gamma  [\lambda_{k} (t) \Delta_{l_k, h} z (t, x) 
		+ f (t, x)] \\
	&+  g (t, x)  (w (t+\gamma) - w (t)), \, z (0, x) = 0, \, \,	 t \in \gamma \bZ_{+}, h \in h\bZ^d.	
	 \end{aligned}
	\end{equation}	

	Here is the outline of the argument.
	We split $v_{\gamma, h}$
	into two parts $v^{(1)}_{\gamma, h}$
	and $v^{(2)}_{\gamma, h}$.
	The first part is a solution of the stochastic
	finite difference equations,
	and the second one is a solution 
	 of a finite difference equation without stochastic part.
	Then we fix $\omega$ and use  
	 Theorem \ref{theorem 3.4}  and  Theorem \ref{theorem 4.4} respectively.
	 After that we take expectations and obtain the desired estimate.
	Note that, however, Theorem \ref{theorem 3.4} requires the function $G$ 
	given by \eqref{3.4}
	 to be independent of $\omega$.
	This holds if the leading coefficients of the stochastic finite difference equations are independent of $\omega$.
	This is why we choose
	$v^{(1)}_{\gamma, h}$ 
	and
	 $v^{(2)}_{\gamma, h}$ 
	to be the solutions
	   of the following equations:
	 \begin{align*}
		& v^{(1)}_{\gamma, h} (t+\gamma, x)  -  v^{(1)}_{\gamma, h} (t, x) = \gamma \Delta_{e_i, h}  v^{(1)}_{\gamma, h} (t, x)\\
																			&		+    g (t, x)  (w (t+\gamma) - w (t)), \, v_{1, h} (0, x) = 0,	
	  \end{align*}
		\begin{align*}
		& v^{(2)}_{\gamma, h} (t+\gamma, x) - v^{(2)}_{\gamma, h} (t, x)  = \gamma  \lambda_{k} (t) \Delta_{l_k, h} v^{(2)}_{\gamma, h} (t, x)  \\
		 &+ \gamma [ f (t, x) +  \lambda_{k} (t) \Delta_{l_k, h} v^{(1)}_{\gamma,  h} -  \Delta_{e_i, h} v^{(1)}_{\gamma, h} (t, x)], \, v^{(2)}_{\gamma, h} (0, x) = 0,
		\, t \in \gamma \bZ_{+}, x \in h \bZ^d.		
	        \end{align*}	
	Note that
	$$
	 v_{\gamma, h} = v^{(1)}_{\gamma, h} + v^{(2)}_{\gamma, h}.
	$$
	Further, observe that  Assumption \ref{assumption 3.1}
	is satisfied with
	$
		a^{i j} \equiv \delta_{i j}
	$	
	and
	Assumption \ref{assumption 3.2}
	is satisfied with
	$d_1 = d$,  
	 $l_k = e_k$,  
	   $\lambda_k = 1$, 
	    $k = 1, \ldots, d$.
	Then, by  Theorem  \ref{theorem 3.4} $(ii)$ with $n = 1$,
	we have
	\begin{equation}
				\label{7.1}
		E \sum_{r = 0}^m [v^{(1)}_{\gamma, h} (t_r, \cdot) ]^p_{2, p, h} \leq N  E \sum_{r = 0}^{m-1} [  g (t_r, \cdot)]^p_{1, p, h}.
	\end{equation}
	By Theorem  \ref{theorem 4.4}, 
	\begin{equation}
				\label{7.2}
	 \begin{aligned}
		E \sum_{r = 0}^m  [v^{(2)}_{\gamma, h} (t_r, \cdot) ]^p_{2, p; \,  h} & \leq  
																		N  E \sum_{r = 0}^{m-1} ( || f (t_r, \cdot)||^p_{p; \,  h} \\
	& [ v^{(1)}_{\gamma, h} (t_r, \cdot) ]^p_{2, p; \,  h} + ||\Delta_{l_k, h} v^{(1)}_{\gamma, h} (t_r, \cdot)||^p_{p; \, h}).
	 \end{aligned}
	\end{equation}
	We need to replace the finite differences with respect to $l_k$
	 by the ones with respect to 
	$
		e_i, i  =1, \ldots, d
	$
	 only.
	To this end observe that, 
	for any function $\kappa$ 
	on $h \bZ^d$, and 
	any
	$
		\xi \in \bZ^d,
	$ 
		\begin{equation}
					\label{7.16}
								\delta_{\xi, h} \kappa (x) = \kappa (x + \xi h) - \kappa (x) =  \sum_{ i =1}^d  \sum_{ y \in \bZ^d: ||y|| \leq ||\xi|| } c_i (y)  \delta_{e_i, h} \kappa (x + h y),    
		\end{equation}
		where $c_i (y) \in \{-1, 0, 1\}$.
		Since 
		\begin{equation}
					\label{7.25}
			\Delta_{\xi, h} \kappa = -\delta_{\xi, h} \delta_{-\xi, h} \kappa,
		 \end{equation}
		we have 
		$$
			\Delta_{\xi, h} \kappa (x)  = \sum_{i, j = 1}^d  \sum_{||y|| \leq 2||\xi||} c_{i j} (y) \delta_{e_i, h} \delta_{e_j, h} \kappa (x + h y),
		$$
		where 
		$
			c_{i j} (y) \in \{-1, 0, 1\}.
		$
		Then,
		 \begin{equation}
																					\label{7.17}
							||\Delta_{\xi, h} \kappa||_{p, h} \leq N (d, p, \xi) [ \kappa ]_{2, p, h}.
		 \end{equation}
	By what was just said
	 combined with \eqref{7.1} and \eqref{7.2} we obtain
	\begin{equation}
				\label{7.3}
		\begin{aligned}
		E \sum_{r = 0}^m    [v_{\gamma, h} (t_r, \cdot) ]^p_{2, p; \,  h} 
																\leq N  E \sum_{r = 0}^{m-1} ( || f (t_r, \cdot)||^p_{p; \,  h} 
		+ [ g (t_r, \cdot)]^p_{1, p; \, h}).
		\end{aligned}
	\end{equation}

	 \textit{Step 2. 
	Case $ n =0$, 
	and
	$a$ and $\nu$  independent of $x$ and  $b \equiv 0$}.
	Denote
	$$
		\tilde g (t, x) = g (t, x)  +  \nu (t) v_{\gamma, h} (t, x).
	$$ 
	Observe that that 
	$
		v_{\gamma, h}
	$ 
	solves Eq. \eqref{7.4} with
	  $g$ replaced by  $\tilde g$.
	Then, by \eqref{7.3}  and Assumption \ref{assumption 3.4.1} we obtain
	\eqref{7.0.1}.
	
	\textit{Step 3.
	Case $n = 0$, 
	$b \equiv 0$, 
	and
	$a$ and $\nu$ depend on $\omega, t, x$.} 
	In this case we use the method of freezing the coefficients.
	Fix some $\varepsilon > 0$, which we will choose later,
	and let 
	$
		\kappa_{\varepsilon}
	$
	be the number from Assumption \ref{assumption 3.3}.
	Let 
	$
	  \zeta \in C^{\infty}_0 (\bR^d)
	$
	be a function  such that
	it is supported
	on the ball
	$
		\{y \in \bR^d: |y| \leq \kappa_{\varepsilon}\},
	$
	and
	$
		\int_{\bR^d} |\zeta (y)|^p \, dy = 1.
	$
	First, we write a version of the product rule for finite differences,
	which we will use in the sequel. 
	For any vector 
	$\xi \in \bZ^d$, and any functions
	 $u$ and  $v$ on $h \bZ^d$,
	\begin{equation}
				\label{7.9}
		\delta_{\xi, h} (v u) = (T_{\xi, h} v) \delta_{\xi, h} u + (\delta_{\xi, h} v) u.
	\end{equation}
	Then, for any 
	$
		\xi, \eta \in \bZ^d,
	$ 
	 $
		t \in \gamma \bZ_{+},
	 $
	 $
		x \in h\bZ^d,
	$ 
	and $y \in \bR^d$, 
	we have
	\begin{equation}
				\label{7.11}
	  \begin{aligned} 
		\zeta (x - y)  & \delta_{\eta, h} \delta_{\xi, h} v_{\gamma, h} (t, x) = \delta_{\eta, h} \delta_{\xi, h} (\zeta (x-y) v_{\gamma, h} (t, x))\\
		&- (T_{\eta, h} T_{\xi, h} v_{\gamma, h} (t, x)) (\delta_{\eta, h} \delta_{\xi, h} \zeta ( x - y))\\
		& - (T_{\xi, h} \delta_{\eta, h} v_{\gamma, h} (t, x)) (\delta_{\eta, h} \zeta (x - y))\\
		& - (T_{\eta, h} \delta_{\xi, h} v_{\gamma, h} (t, x)) (\delta_{\xi, h} \zeta (x-y)).
	  \end{aligned}
	 \end{equation}
	We plug 
	$
		\xi = e_j, 
	   \nu = e_i,
 	$
	and
	raise the above identity to the $p$-th power, 
	sum all the terms with respect to 
	$x \in h \bZ^d$,
	and  integrate with respect to 
	$y \in \bR^d$.
	Then, by the mean value theorem
	we get
	\begin{equation}
				\label{7.10}
	 \begin{aligned}
	&	|| \delta_{e_j, h} \delta_{e_i, h} v_{\gamma, h} (t, \cdot)||^p_{p; \,  h} \\
		&\leq  N \int ||\delta_{e_j, h} \delta_{e_i, h} (v_{\gamma, h} (t, \cdot) \zeta (\cdot - y))||^p_{p; \, h} \, dy + N || v_{\gamma, h} (t, \cdot)||^p_{1, p; \,  h}.
	\end{aligned}
	\end{equation}
	
	To estimate the first term on the right hand side of \eqref{7.10}
	we introduce a function
	$
	     z_{\gamma, h, y} (t, x)= v_{\gamma, h} (t, x) \zeta (x - y), t \in \gamma \bZ_{+}, x \in h \bZ^d.
	$
	By plugging 
	$
		\eta = -l_k, \xi = l_k
	$
	 in \eqref{7.11},
	and \eqref{7.25} and the fact that
	$$
		 T_{\xi, h} \delta_{-\xi, h} = - \delta_{\xi, h},
	$$
	 we obtain the following identity:
	$$
		\zeta (x - y) \Delta_{l_k, h} v_{\gamma, h} (t, x) = \Delta_{l_k, h} (v_{\gamma, h} (t, x) \zeta (x-y))
	$$
	 $$
		-  v_{\gamma, h} (t, x) (\Delta_{l_k, h} \zeta ( x - y))
	 $$
	  $$
		  - (\delta_{l_k, h} v_{\gamma, h} (t, x)) (\delta_{-l_k, h} \zeta (x - y))\\
		 - (\delta_{-l_k, h}  v_{\gamma, h} (t, x)) (\delta_{l_k, h} \zeta (x-y)).
	   $$
	By using this identity
	we get that
	$
		z_{\gamma, h, y}
	$ 
	solves the following difference equation with  coefficients 
	constant in the temporal variable:
	\begin{equation}
				\label{7.12}
	  \begin{aligned}
		& u (t+\gamma, x) - u (t, x) = \gamma [ \lambda_k (t, y) \Delta_{l_k, h} u (t, x)  + \tilde f (t, x)] \\
																				&	+ [\nu (t, y) u (t, x) + \tilde g (t, x)] (w (t+\gamma) -  w (t)), \quad u (0, x) = 0,
	  \end{aligned}
	 \end{equation}
	 where 
	 $$
		\tilde f (t, x) = f (t, x) +  (\lambda_k (t, x) - \lambda_k (t, y)) \Delta_{l_k} z_{\gamma, h, y} (t, x) 
	$$
	 $$
		-   \lambda_k (t, x) v_{\gamma, h} (t, x) \Delta_{l_k, h} \zeta (x) 
	 $$
	 $$
		- \lambda_k (t, x) (\delta_{ - l_k, h}  v_{\gamma, h} (t, x)) (\delta_{l_k, h} \zeta (x - y))
	$$
	 $$
		 - \lambda_k (t, x) (\delta_{ l_k, h} v_{\gamma, h} (t, x)) (\delta_{ - l_k, h } \zeta (x - y)),
	 $$
	  $$
		\tilde g (t, x) = g (t, x) +  [\nu (t, x) - \nu (t, y)] z_{\gamma, h, y} (t, x).
	$$
	By what was just said we  may apply our results that we obtained in Step 2.
	 By \eqref{7.0.1} we get
	 \begin{equation}
				\label{7.12.1}
		E \sum_{r = 0}^{m}  [ v_{\gamma, h} (t_r, \cdot) \xi (\cdot - y) ]^p_{ 2, p, h} \leq 
																				 N \sum_{i = 1}^4 E \sum_{r = 0}^{m-1}  I_{i} (t_r) + \gamma \cR_0,
	 \end{equation}
	where, for $t \in \gamma \bZ$, 
		 $$
			I_{1} (t) =  ||(\lambda_k (t, \cdot) - \lambda_k (t, y)) \Delta_{l_k, h} (v_{\gamma, h} (t, \cdot) \xi (x - y)) ||^p_{ p; \, h},
		  $$
		   $$
			 I_{2} (t) =  || \lambda_k (t, \cdot) v_{\gamma, h} (t, \cdot) \Delta_{l_k, h} \xi  (\cdot - y) ||^p_{ p; \, h},
		    $$
		     $$
			  I_{3} (t) =  ||\lambda_k (t, \cdot) (\delta_{ - l_k, h} v_{\gamma, h} (t, \cdot)) (\delta_{l_k, h} \xi (\cdot - y)) ||^p_{p; \, h}
		     $$
		      $$	
					+ ||\lambda_k (t, \cdot) (\delta_{l_k, h } v_{\gamma, h} (t, \cdot)) (\delta_{- l_k, h} \xi (\cdot - y))||^p_{ p;  \, h},
		      $$
		 	$$
				I_{4} (t) =    || (\nu (t, \cdot) - \nu (t, y))   v_{\gamma, h} (t, \cdot) \xi (\cdot - y)||^p_{1, p;  \,  h}.
			$$
		In the sequel $\tilde \zeta$ is a $C^{\infty}_0 (\bR^d)$ function supported on the ball of radius $\kappa_{\varepsilon}$ 
		which might change from inequality from inequality.
		 By  the fact that $\xi$ is supported on the ball of radius $\kappa_{\varepsilon}$ and Assumption   \ref{assumption 3.3} 
		we get
		\begin{equation}
					\label{7.13}
		\begin{aligned}
			I_{1} (t)   &\leq \varepsilon^p || \lambda_k (t, \cdot) \Delta_{l_k, h} (v_{\gamma, h} (t, \cdot) \zeta (\cdot - y)) ||^p_{ p; \,  h} \\
			&\leq \varepsilon^p [  v_{\gamma, h} (t, \cdot) \zeta (\cdot - y) ]^p_{ 2, p; \,  h},
		\end{aligned}
		\end{equation}
		where the last inequality is due to \eqref{7.17}.
		Similarly, by \eqref{7.9}, Assumption \ref{assumption 3.4.1}  and the mean value theorem one has
		\begin{equation}
					\label{7.14}
		\begin{aligned}
			I_{4} (t) &\leq  N  \varepsilon^p   (||v_{\gamma, h} (t, \cdot)  \tilde \zeta (\cdot - y) ||^p_{p; \, h} \\
			&
			 +    \sum_{j = 1}^d ||    (\delta_{e_j, h} v_{\gamma, h} (t, \cdot))  \zeta (\cdot - y)  ||^p_{p; \, h}).
		\end{aligned}
		\end{equation}
		Next, by Assumption \ref{assumption 3.2} and the mean value theorem
		\begin{equation}
					\label{7.15}
		  \begin{aligned}
			I_{2} (t) & \leq   N  ||    v_h (t, \cdot)  \tilde \zeta (\cdot - y) ||^p_{ p, h}  \\
										&                                   I_{3} \leq ||   (\delta_{l_k, h} v_{\gamma, h} (t, \cdot))  \tilde \zeta (\cdot - y) ||^p_{p, h}  
														+ ||  (\delta_{-l_k, h} v_{\gamma, h} (t, \cdot))  \tilde \zeta (\cdot - y) ||^p_{p, h}.
		  \end{aligned}
		\end{equation}
		Further, by 	\eqref{7.16} 
		\begin{equation}
					\label{7.18}
			I_{3} (t)  \leq N \sum_{j = 1}^d \sum_{ \xi \in \bZ^d: ||\xi|| \leq L} || (\delta_{e_j, h} v_{\gamma, h} (t, \cdot + h\xi)) \tilde \zeta (\cdot - y)   ||^p_{p, h},
		\end{equation}
		where 
		$
			L = \max_{k = 1, \ldots, d_1} ||l_k||.
		$

		Combining \eqref{7.12.1} with \eqref{7.13} - \eqref{7.18} 
		 and integrating both sides of the inequality with respect to $y \in \bR^d$,
		 we obtain 
		   \begin{align*}
			E \sum_{r  = 0}^{m} & \int  [  v_{\gamma, h} (t_r, \cdot) \xi (\cdot - y) ]_{ 2, p; \, h} \, dy   
			\leq N \cR_0  + N E \sum_{r = 0}^{m-1}  || v_{\gamma, h} (t_r, \cdot) ||^p_{1, p; \,  h}
					\\ 
			&  + (N_1 \varepsilon)^p   E \sum_{r = 0}^{m-1}    \int [ v_{\gamma, h} (t_r, \cdot) \xi (\cdot - y) ]^p_{ 2,  p; \,  h} \, dy. \\
		\end{align*}
		To get rid of the last term on the right hand side of the above inequality we choose
		$
		   \varepsilon  < (2 N_1)^{-1}.
		$ 
		By this and \eqref{7.10}
		\begin{equation}
					\label{7.20}
								  E \sum_{r  = 0}^{m}  [ v_{\gamma, h} (t_r, \cdot) ]^p_{2, p, h} \leq  
																							N \cR_0 +  N E \sum_{  r  =  0  }^{m-1}    || v_{\gamma, h} (t_r, \cdot) ||^p_{1, p; \, h}.
		\end{equation}

		\textit{Step 4.
		Case $n = 0$, 
		and
		$a, b, \nu$ 
		depend on $t, x, \omega$}.
		The function $v_{\gamma, h}$
		 solves Eq. \eqref{3.1} with
		$b \equiv 0$, 
		and
		$f (t, x)$
		 replaced by 
		$$
			\tilde f (t, x) = f (t, x) + b^i (t, x) \delta_{e_i, h} v_h (t, x).
		$$
		Then, by what we proved in the previous step
		$$
			E \sum_{r = 0}^{m} [v_{\gamma, h} (t_r, \cdot)]^p_{2, p; \, h} \leq N \cR_0 
			+ E \sum_{r = 0}^{m-1} || (b^i  \delta_{e_i, h} v_{\gamma, h}) (t_r, \cdot) ||^p_{p; \, h}.
		$$
		Now \eqref{7.0.1} follows from Assumption \ref{assumption 3.4.1}.

		\textit{ Step 5. Case $n > 0$.}
		Fix any multi-index $\alpha$ such that
		$
			||\alpha|| \leq n
		$
		and apply the operator 
		$
		   \delta^{\alpha}_h
		$
		to
		 Eq. \eqref{3.1}.
		Then, by \eqref{7.9}
		$
		  \delta^{\alpha}_h v_{\gamma, h}
		$ 
		satisfies  
		Eq. \eqref{3.1} with 
		$ f (t, x)$ 
		replaced by $\tilde f (t, x)$ 
		and $g (t, x)$ replaced by $\tilde g (t, x)$, where
		$$
			\tilde  f (t, x) =  \delta^{\alpha}_h f (t, x) + \sum_{\alpha' : \alpha' \neq 0, \alpha' \leq \alpha} N (\alpha', \alpha)  (T^{\alpha'}_h \delta^{\alpha'}_h \lambda_k (t, x)) (\delta^{\alpha - \alpha'}_h \Delta_{l_k, h} v_{\gamma, h} (t, x))
		$$
		 $$
			+ \sum_{\alpha' : \alpha' \neq 0, \alpha' \leq \alpha} N (\alpha', \alpha) (T^{\alpha'}_h \delta^{\alpha'}_h b^i (t, x)) (\delta^{\alpha - \alpha'}_h  \delta_{e_i, h} v_{\gamma, h} (t, x)),
		 $$ 
		   $$
			\tilde g (t, x) = \delta^{\alpha}_h g (t, x) + \sum_{\alpha' : \alpha' \neq 0, \alpha' \leq \alpha} N (\alpha', \alpha) (T^{\alpha'}_h \delta^{\alpha'}_h \nu^i (t, x)) (\delta^{\alpha - \alpha'}_h  v_{\gamma, h} (t, x)).
		   $$		  
		Then, the results of the previous case are applicable, and by \eqref{7.0.1} we get
		$$
			E \sum_{r = 0}^{m} [ \delta^{\alpha}_h v_{\gamma, h} (t_r, \cdot) ]^p_{2, p; \,  h} 
				\leq N   E  \sum_{r = 0}^{m-1} (|| \tilde  f (t_r, \cdot) ||^p_{p; \,  h} 
		$$  
	   	 $$
		 	+ || \tilde g (t_r, \cdot)||^p_{1, p; \,  h} + ||v_{\gamma, h} (t_r, \cdot)||^p_{|\alpha|+1, p;\, h}).
		 $$
		By Assumption \ref{assumption 3.4.1} and \eqref{7.17}, 
		for any 
		$
			t \in \gamma \bZ_{+},
		$
		 we get
		$$
				|| \tilde  f (t, \cdot) ||_{p; \,  h} \leq N  (|| f (t, \cdot)||_{ |\alpha|, p; \, h} + || v_{\gamma, h} (t, \cdot) ||_{ |\alpha| + 1, p; \, h}).
		$$
		Using \eqref{7.9},
		 we obtain
		$$
			|| \tilde g (t, \cdot)||^p_{1, p; \,  h} \leq N || g (t, \cdot) ||^p_{ |\alpha| + 1, p; \, h} +  || v_{\gamma, h} (t, \cdot) ||_{|\alpha|, p; \, h}.
		$$
		Then, by the above,
		 for any 
		$
			n_1 \in \{0, \ldots, n\},
		$ 
		there exists a constant $N_{s_1}$ such that
		\begin{equation}
					\label{7.21}
			E  \sum_{n = 0}^{m}  [  v_{\gamma, h} (t_n, \cdot) ]^p_{n_1 + 2, p; \, h}
																			 \leq N_{s_1}  (\cR_{n_1} + E  \sum_{n = 0}^{m-1} || v_{\gamma, h} (t_n, \cdot) ||^p_{n_1 + 1, p; \,  h}).
		\end{equation}

		To simplify the right hand side of the above inequality 
	we use an interpolation inequality which we describe below 
	(see Theorem 4.0.4 in \cite{D_12} or Propositions 1 and 4 in \cite{HY_19}).
	 For any 
	function
	$
	  \kappa \in l_p (h \bZ^d),
	$ 
	and any $\varepsilon > 0$, 
	one has 
	\begin{equation}
				\label{7.8}
	 	[ \kappa ]_{1, p; \,  h} \leq
									  N_0 (p, d) (\varepsilon [ \kappa]_{2, p; \,  h} + \varepsilon^{-1} ||\kappa||_{p; \,  h}).  	
	\end{equation}

	 Next, by applying
	  inequality \eqref{7.8} repeatedly and choosing $\varepsilon$ small enough, we get	
		\begin{equation}
					\label{7.21.1}
			E  \sum_{r = 0}^{m}  [  v_{\gamma, h} (t_r, \cdot) ]^p_{n_1 + 2, p; \, h}
																			 \leq \tilde N_{n_1}  (\cR_{n_1} +  E \sum_{r = 0}^{m-1} || v_{\gamma, h} (t_r, \cdot) ||^p_{ p; \,  h}).
		\end{equation}

		Next, we estimate the $l_p$ norm of $v_{\gamma, h}$.
		By  Lemma \ref{lemma 9.4}, for any 
		$
			m \in \{1, \ldots, T/\gamma\},
		$
		\begin{equation}
					\label{7.22}
			E   ||v_{\gamma, h} (t_m , \cdot)||^p_{p; \, h} 
												 \leq N \gamma  (\cR_0 +   \sum_{k = 1}^3 E   \sum_{r = 0}^{m-1}    J_k (t_r)),
		 \end{equation}
		where
		$$
			J_1 (t) = || (\lambda_{k} \Delta_{l_k, h} v_{\gamma, h}) (t, \cdot)||^p_{p; \, h},
		$$
	 	$$
			J_2 (t) = || (b_i \delta_{e_i, h}) (t, \cdot)||^p_{p; \, h}
		$$
		 $$
			J_3 (t) = || (\nu  v_{\gamma, h}) (t, \cdot)||^p_{1, p; \, h}.
		 $$
		By Assumption \ref{assumption 3.4.1} and \eqref{7.17}
		\begin{equation}
					\label{7.23}
			J_1 (t) \leq N [ v_{\gamma, h} (t, \cdot) ]^p_{2, p; \, h},
		\end{equation}
		\begin{equation}
					\label{7.23.1}
			J_2 \leq  ||v_{\gamma, h} (t, \cdot)||^p_{1, p; \, h}.
		\end{equation}
		Next, by \eqref{7.9} and Assumption \ref{assumption 3.4.1}
		$$
			J_3 (t) \leq N || v_{\gamma, h} (t, \cdot) ||^p_{1, p; \, h}.
		$$
		Combining this with \eqref{7.8}, we get
		\begin{equation}
					\label{7.24}
			J_3 (t) \leq N ([ v_{\gamma, h} (t, \cdot) ]^p_{2, p; \, h} + || v_{\gamma, h} (t, \cdot)||^p_{p; \, h}).
		\end{equation}
		Combining \eqref{7.22} - \eqref{7.24} with \eqref{7.21.1} (with $n_1 = 0$), 
		 we get
		$$
			E   ||v_{\gamma, h} (t_m, \cdot)||^p_{p; \, h}  \leq
		 N \gamma \mathcal{R}_0  + N_1 \gamma  E \sum_{r = 0}^{m - 1} ||v_{\gamma, h} (t_r, \cdot)||^p_{ p; \, h},
		$$
		where $N$ is independent of $m$.
		Then, by Lemma \ref{lemma 9.1} and the fact that $m \leq T/\gamma$ 
		$$
			E  \sum_{r = 0}^{T/\gamma} ||v_{\gamma, h} (t_r, \cdot)||^p_{p; \, h}  \leq N  \cR_0.
		$$
		 This combined with \eqref{7.21.1} implies  \eqref{3.5}.

		\textit{Step 6.} 
		Finally, we prove the inequality \eqref{3.5.1}.
		By Lemma \ref{lemma 9.4} 
		$$
			E \max_{r = 1, \ldots, T/\gamma} || v_{\gamma, h} (t_r, \cdot) ||^p_{ n + 1, p; \, h}
		 \leq N  h^2 \cR_n +  N h^2  \sum_{r = 0}^{(T/\gamma) - 1} E (|| (\lambda_k  \Delta_{l_k, h} v_{\gamma, h}) (t_r, \cdot) ||^p_{n, p; \, h} 
		$$
		 $$
			+ || (b^i \delta_{e_i, h} v_{\gamma, h}) (t_r, \cdot) ||^p_{n, p; \, h} +  || (\nu^l  v_{\gamma, h}) (t_r, \cdot) ||^p_{n + 1, p; \, h}).
	   	$$
		To estimate each term we use the argument similar to the one above. 
		By Lemma \ref{lemma 9.2}  combined with Assumption \ref{assumption 3.4.1}  we get
		$$
			E \max_{r = 1, \ldots, T/\gamma} || v_{\gamma, h} (t_r, \cdot) ||^p_{ n + 1, p; \, h}
			 \leq  N h^2  \cR_n +   N h^2 \sum_{r = 0}^{(T/\gamma) - 1} E (|| v_{\gamma, h} (t_r, \cdot) ||^p_{ n + 1, p; \, h} 
		$$
	  	 $$
			+ ||\Delta_{l_k, h} v_{\gamma, h} (t_r, \cdot) ||^p_{n, p; \, h}).
		 $$
		By \eqref{7.17} we may replace the term involving $\Delta_{l_k, h}$  by 
		 $
			[ v_{\gamma, h} ]^p_{n + 2, p; h}.
		$
		Now  \eqref{3.5.1} follows from   \eqref{3.5}.

		\textbf{Proof of Theorem \ref{theorem 4.1}}
		 Note that $u_{\gamma, h}$ is the solution of \eqref{3.1} with 
		$
		 	 g^l \equiv 0, l \geq 1.
		$ 
		Then, one repeats the argument of the proof of Theorem \ref{theorem 3.2} almost word for word. 
		However, in Step 1 we use Theorem \ref{theorem 4.4} with $p > 1$,
		 and that is why \eqref{4.1.0.0} holds with $p > 1$.
		In Step 6 we use 
		 Lemma \ref{lemma 9.4} $(ii)$ where $p$ is required to be greater that $2$,
		and, hence the estimate \eqref{4.1.0} is valid only for $p \geq 2$. 
	\end{proof}

	\mysection{Proof of Theorem \ref{theorem 4.2}}
										\label{section 10}

	The proof is standard. We construct a finite difference equation
	that is satisfied by the error of the approximation.
	After that both claims follow
	 from the stability estimates of Theorem \ref{theorem 4.1}.

	First, since $u \in \fH^{s+4}_p (0, T)$, 
	 by Remark \ref{remark 2.2}
	 $u$ is a four times differentiable function,
	and, hence, it is  a classical solution of \eqref{1.1}.
	Then, due to Assumption \ref{assumption 3.2}, for any $t, x$, 
	\begin{equation}
				\label{10.0}
		a^{i j} (t, x) D_{i j} u (t, x) = \lambda_k (t, x) D_{(\lambda_k) (\lambda_k)} u (t, x).
	\end{equation}

	Next, 
	denote 
	$
		e_{\gamma, h} (t, x) =   u (t, x) - u_{\gamma, h} (t, x), \, t \in \gamma \bZ_{+}, x \in h\bZ^d.
	$
	By \eqref{10.0}
	$e_{\gamma, h}$ is the solution of the following difference equation:
	\begin{equation}
					\label{10.1}
	 \begin{aligned}
		& z (t+\gamma, x) - z (t, x)  = \gamma \lambda_k (t, x) \Delta_{l_k, h} z (t, x) + \gamma b^i (t, x) \delta_{e_i, h} z (t, x) \\
														&   + \gamma  \sum_{q = 1}^3 f_q (t+\gamma, x), \, \,  z (0, x) = 0,  \, t \in \bZ_{+}, x \in h \bZ^d,
	 \end{aligned}
	\end{equation}
	where
	$$
		f_1 (t+\gamma, x) = \gamma^{-1} \int_t^{t+\gamma} (f (r, x) - f(t, x)) \, dr,
	$$
	 $$
		f_2 (t+\gamma, x) = \gamma^{-1} \int_t^{t+\gamma} (b^i (r, x)  D_i u (r, x) - b^i (t, x) \delta_{e_i, h} u_{\gamma, h} (t, x))\, dr,
	 $$
	  $$
		f_3 (t+\gamma, x) = \gamma^{-1} \int_t^{t+\gamma} (a^{i j} (r, x) D_{i j} (r, x)  - \lambda_k (t, x) \Delta_{l_k, h} (t, x)) \, dr.
	  $$

	Then, by Theorem \ref{theorem 4.1}
	\begin{equation}
				\label{10.2}
	  	\gamma \sum_{m = 0}^{T/\gamma} || e_{\gamma, h} (t_m, \cdot)||^p_{ n + 2, p; \, h} 
		\leq N \gamma \sum_{m = 0}^{(T/\gamma)-1} (F_{1} (t_m) + F_{2} (t_m) + F_{3} (t_m)),
	\end{equation}
	where
	$$
		F_{q} (t) = || f_q (t, \cdot) ||^p_{ n, p; \, h}, \,  t \in \gamma \bZ_{+},  q = 1, 2, 3.
	$$
	Fix any 
	$
		t \in  \gamma\bZ_{+} \cap [0, T].
	$ 
	By  Minkowski and H\"older's inequalities
	\begin{equation}
				\label{10.3}
		F_{2} (t) \leq N (F_{2, 1} (t) + F_{2, 2} (t) + F_{2, 3} (t)),
	\end{equation}
	where
	$$
		F_{2, 1}  (t) = \gamma^{-1} \int_t^{t+\gamma} ||(b^i (r, \cdot) - b^i (t, \cdot))  D_i u (r, \cdot)||^p_{n, p; \, h} \, dr,
	$$
	 $$
		F_{2, 2} (t) = \gamma^{-1} \int_t^{t+\gamma} || b^i (t, \cdot) (D_i u (r, \cdot) - D_i u (t, \cdot)) ||^p_{n, p; \, h} \, dr,
	 $$	
	  $$
		F_{2, 3} (t) =  || b^i (t, \cdot) (D_i u (t, \cdot) - \delta_{e_i, h} u (t, \cdot))||^p_{n, p, \, h}.
	  $$
	Similarly, by \eqref{10.0}
	\begin{equation}
				\label{10.4}
		F_{3} (t) \leq N (F_{3, 1} (t) + F_{3, 2} (t) + F_{3, 3} (t)),
	\end{equation}
	where
	$$
		F_{3, 1} (t) = \gamma^{-1} \int_t^{t+\gamma} ||(a^{i j} (r, \cdot)  - a^{i j} (t, \cdot)) D_{i j} u (r, \cdot) ||^p_{n, p; \, h} \, dr,
	$$
	 $$
		F_{3, 2} (t) = \gamma^{-1}   \int_t^{t+\gamma} || a^{i j} (t, \cdot)  (D_{i j} u (r, \cdot)  - D_{i j} u (t, \cdot))||^p_{n, p; \, h}  \, dr,
	 $$	
	  $$
		F_{3, 3} (t) =  || \lambda_k (t, \cdot) (D_{ (\lambda_k) (\lambda_k) } u (t, \cdot) - \Delta_{l_k, h} u (t, \cdot))   ||^p_{n, p; \, h}.
	  $$
	
	Our goal is to show that, for any $t \in [0, T-\gamma]$, the following estimates hold: 
	\begin{equation}
				\label{10.5}
		 F_1 (t)  \leq   N  \gamma^{p-1 }  \int_t^{t+\gamma} ||D_t f  (r) ||^p_{ H^{s_1}_p } \, dr,
	\end{equation}
	 \begin{equation}
				\label{10.6}
		F_{2, 1} (t) + F_{3, 1} (t) \leq   N \gamma^{p - 1   }   \int_t^{t+\gamma}  ||  u (r, \cdot) ||^p_{ H^{s + 2  }_p  } \, dr,
	 \end{equation}
	\begin{equation}
				\label{10.7}
	F_{2, 2} (t) + F_{2, 3} (t) \leq N  \gamma^{p - 1   } \int_t^{t+\gamma}  ||  u (r, \cdot) ||^p_{ H^{s + 2  }_p  }  + ||f (r, \cdot)||^p_{  H^{s+2}_p } \, dr,
	\end{equation}
	  \begin{equation}
				\label{10.8}
		F_{2, 3 } (t) + F_{3, 3} (t) \leq N h^{2p} || u ||^p_{\fH^{s+4}_p (0, T)}.
	  \end{equation}
	These inequalities combined with \eqref{10.2} prove \eqref{4.2.0}.
	Next, to prove the second assertion of the theorem we use by Theorem \ref{theorem 4.1}.
	By \eqref{4.1.0} and  \eqref{10.1} we get
	$$
		E \max_{m  = 1, \ldots, T/\gamma} || e_{\gamma, h} (t_m, \cdot) ||^p_{ n, p; \, h }  
	$$
	$$
		\leq  N \gamma \sum_{m = 0}^{T/\gamma } (F_{1} (t_m) + F_{2} (t_m) + F_{3} (t_m)) \leq N h^{2 p (\beta - 1/p)  } \cR.
	$$
	Now the estimate \eqref{4.2.1}  the embedding theorem for discrete Sobolev spaces
	(see, for example, in Section 4.3 of \cite{D_12}).
	One can also obtain find this embedding estimate by combining Proposition 1 with Proposition 5 of \cite{HY_19}.

	\textit{Estimate of $F_1$.}
	By  Lemma \ref{lemma 9.2} 
	and Lemma \ref{lemma 9.6} 
	with $X = H^{s_1}_p$ 
	   we have
	$$
		F_1 (t)  
			\leq N \gamma^{-1} \int_t^{t+\gamma} || f (r, \cdot) - f (t, \cdot) ||^p_{n, p; \, h} \, dr 
		\leq N \gamma^{-1} \int_t^{t+\gamma} || f (r, \cdot) - f (t, \cdot) ||^p_{ H^{s_1}_p } \, dr.
	 $$
	Now \eqref{10.5} follows from Lemma \ref{lemma 9.6} with $X= H^{s_1}_p$.
	
	\textit{Estimate of
	$F_{2, 1}$ and $F_{3, 1}$.}
	First, we use Lemma \ref{lemma 9.2} combined with Assumption \ref{assumption 4.4}.
	 We get
	$$
		F_{3, 1} (t) \leq  N (n)\gamma^{-1} \int_t^{t+\gamma}  || (a^{i j} (r, \cdot) - a^{i j} (t, \cdot))||^p_{ C^{n-1, 1}} \,  || D_{i j} u (r, \cdot) ||^p_{ n, p; \, h }  \, dr 
	$$
	 $$
		\leq  N \gamma^{p   - 1 }   \int_t^{t+\gamma}  || D_{i j} u (r, \cdot) ||^p_{ n, p; \, h  } \, dr.
	 $$
	By Lemma \ref{lemma 9.3} we obtain the estimate \eqref{10.6} for $F_{3, 1}$ only. 
	The argument for $F_{2, 1}$ is similar.

	\textit{Estimate of
	$F_{2, 2}$ and $F_{3, 2}$.}
	Again, by Lemma \ref{lemma 9.2} and Assumption \ref{assumption 4.4} 
	$$
	 F_{2, 2 } (t) + F_{3, 2} (t)
					\leq N \gamma^{-1} \int_t^{t+\gamma} (|| D_i (u (r, \cdot) - u (t, \cdot))||^p_{n, p;\, h}
	$$
	 $$
	 + ||D_{i j} (u (r, \cdot) -  u (t, \cdot)) ||^p_{n, p; \, h } \, dr).
	 $$
	Further, by Lemma \ref{lemma 9.6}
	$$
		 F_{2,2 } (t) + F_{3, 2} (t) 	 \leq  N \gamma^{ p  - 1 } \int_t^{t+\gamma} (||D_i D_t u (r, \cdot)||^p_{n, p; \, h} + || D_{i j} D_t u (r, \cdot) ||^p_{ n, p; \, h}) \, dr. 
	$$
	Now  \eqref{10.7}  follows from Lemma \ref{lemma 5.6}.

	\textit{Estimate of
	$F_{2, 3}$ and $F_{3, 3}$.}	
	It suffices to prove the estimate for $F_{3, 3}$.
	Recall that $u (t, \cdot) \in C^4$  for each $t$.
	Then,
	 by  Taylor's formula and Minkowski inequality we have
	$$
		 F_{3, 3} (t) \leq N h^{2p} \sum_{||\alpha|| = 4} \int_0^1  (|| \lambda_k (t, \cdot)  D^{\alpha}_x u (t, \cdot + r h l_k) ||^p_{ n, p; \, h }  
	$$
	  $$
		+|| \lambda_k (t, \cdot) D^{\alpha}_x u (t, \cdot - r h l_k) ||^p_{ n, p; \, h }) \, dr.
	  $$
	By  Lemma \ref{lemma 9.2} and Assumption \ref{assumption 3.4.1}
	$$
		|| \lambda_k (t, \cdot)  D^{\alpha}_x u (t, \cdot + r h l_k) ||_{n, p; \, h } 
					  \leq N || \lambda_k (t, \cdot) ||_{ n - 1, 1}\,   ||  D^{\alpha}_x u (t, \cdot) ||_{ n, p; \, h }  
	$$
	 $$
		\leq  N ||  D^{\alpha}_x u (t, \cdot) ||_{ n, p; \, h }.
	 $$
	By Lemma \ref{lemma 9.3}
	 and Remark \ref{remark 2.2} combined with the fact that
	 $
		s > n + d/p + 2/p
	$
	 the last expression is dominated by
	$$
		   N h^{2p} || u (t, \cdot) ||^p_{n+4, p; \, h} 
	\leq N h^{2p} ||u ||^p_{ \fH^{s+4}_p (0, T) }.
	$$
	Thus, \eqref{10.8} is proved.

	\begin{remark}
		For the argument involving Taylor's formula to work,
		 we needed $u (t, \cdot) \in C^4$, for each $t$.
		This
		led us
		to imposing conditions like Assumption \ref{assumption 4.3.1} and \ref{assumption 4.5} which ensure that $u \in \fH^{s+4}_p (0, T)$ 
		(see
		Theorem \ref{theorem 4.5}
		).
	\end{remark}

	\mysection{Proof of Theorem \ref{theorem 3.1}}
										\label{section 8}
	 
	We point out that one cannot apply directly the argument of the proof of Theorem \ref{theorem 4.2}.
	To show this we note that
	$$
		\int_t^{t+\gamma} g^l (r, x)  \, dw^l (r) - g^l (t, x) (w^l (t+\gamma) - w^l (t)) = \int_t^{t+\gamma} (g^l (r, x)  -g^l (t, x)) \, dw^l (r).
	$$
	In general, one cannot expect  this integral to be equal to 
	$
		\tilde g (t, x)  (w^l (t+\gamma) - w^l (t)),
	$ 
	where  $\tilde g (t, x)$ is some $\cF_t$-measurable random variable. 
	Then,  the function
	$
		v (t, x) - v_{\gamma, h} (t, x) , t\in \gamma \bZ_{+}, x \in h \bZ^d
	$
	does not satisfy an equation of type \eqref{3.1}.
	To handle this issue we introduce a sequence of auxiliary SPDEs (see \eqref{14.2}).
	We show that the $m$-th equation has a unique solution 
	$
		v_m \in \cH^{s+3 - \varepsilon}_p,
	$ 
	where $\varepsilon > 0$ is small.
	 In addition, we prove that each solution $v_m$ approximates $v$ with the error of  order  $\gamma^{\theta  p -1}$.
		It will be seen from the argument
	 that this error comes from the estimate of the modulus of continuity of $v$.
	Third, the difference between $v_{\gamma, h}$ and the solution $v_m$  satisfies the difference equation of type \eqref{3.1}.
	Then,  by using Theorem \ref{theorem 3.1} and by following the argument
	given in the proof of Theorem \ref{theorem 4.2},
	 we show that
	the difference between
	  $v_{\gamma, h}$ and $v_{m}$ is order $h^{2 (\theta p -1)}$.
	Now the assertion follows
	from  this and what was said about the difference between $v_m$ and $v$.

        First, note that by
	Remark \ref{remark 2.2}, 
	for any $\omega$ and
	 $t \in [0, T]$, 
	the function
	$v (t, \cdot) \in C^3$, 
	and, then, $v$ is classical solution of \eqref{1.2}.
	 Further, by Assumption \ref{assumption 3.2}, 
	for any $\omega, t$ and $x$, 
	we have
	\begin{equation}
				\label{14.1}
		a^{i j} (t, x) D_{i j} v  (t, x) =  \lambda_k (t, x)   D_{(\lambda_k) (\lambda_k)} v (t, x).
 	\end{equation}	

	Next, we set $v_0 \equiv u_0$
	and consider the following sequence of SPDEs:
	\begin{equation}
				\label{14.2}
	\begin{aligned}
		dv_{m} (t, x) &=   [a^{i j} (t, x) D_{i j} v_{m} (t, x) + b^i (t, x) D_i v_{m} (t, x) + f (t, x)] \, dt\\ 
																						 &+ (g^l (t_{m-1}, x) + \nu^l (t_{m-1}, x) v_{ m - 1 } (t_{m-1}, x)) (w^l (t_m) - w^l (t_{m-1})),\\
						& \quad v_{m} (t_{m-1}, x) = v_{m - 1} (t_{m-1}, x), \, t\in [t_{m-1}, t_{m}].
	 \end{aligned}
	\end{equation}
	According to  Lemma \ref{lemma 5.5} $(i)$, for any $\varepsilon \in (0, 1)$, 
	and any $m \in \{1, \ldots, T/\gamma\}$, 
	there exists a unique solution 
	$
		v_m \in \cH^{s+3}_p (t_{m-1}, t_m),
	$
	and in addition, \eqref{8.1.1} and \eqref{8.4} hold.

	Next, for 
	$
	  m \in \{0, 1, \ldots, \gamma^{-1}\}
	$
	and
	 $
		x \in h \bZ^d,
	$
	denote
	$
		e_{\gamma, h} (t_m, x) =  v_m (t_m, x) -  v_{\gamma, h} (t_m, x).
	$ 
	Since
	$v_m$ is a classical solution of \eqref{14.2} on $[t_{m-1}, t_m]$,
	for any 
	$
		m \in \{0, 1, \ldots, \gamma^{-1} - 1\}
	$
	  the function $e_{\gamma, h}$ satisfies the following equation:
	\begin{equation}
				\label{14.2.1}
	\begin{aligned}
		e_{\gamma, h} (t_{m+1}, x) &- e_{\gamma, h} (t_m, x) =  \gamma  [\lambda_k (t_m, x) \Delta_{l_k, h} e_{\gamma, h} (t_m, x) 
																										+ b^i (t_m, x) e_{\gamma, h} (t_m, x)\\
		&  f_1 (t_m, x) + f_2 (t_m, x) ] \\
		&+ \nu^l (t_m, x) e_{\gamma, h} (t_m, x) (w^l (t_{m+1}) - w^l (t_m)), \,  e_h (0, x) = 0, 
	 \end{aligned}
	\end{equation}
	where
	$$
		f_1 (t_{m+1}, x) =  \gamma^{-1}   \int_{t_m}^{t_{m+1}}  (a^{i j} (t, x) D_{i j} v_m (t, x) -  \Delta_{l_k, h} v_m (t_m, x))\, dt,
	$$
	 $$
		f_2 (t_{m+1}, x) = \gamma^{-1}   \int_{t_m}^{t_{m+1}} (b^i (t, x) D_i v_m (t, x) - b^i (t_m, x) \delta_{e_i, h} v_m (t_m, x))\, \, dt,
	 $$
	  $$
		f_3 (t_{m+1}, x) = \gamma^{-1} \int_{t_m}^{t_{m+1}}  (f (t, x) -  f (t_m, x)) \, dt.
	  $$
	By Theorem \ref{theorem 3.2} 
	\begin{equation}
				\label{14.3}
		\gamma E \sum_{m = 0}^{  T/\gamma  } || e_{\gamma, h} (t_m, \cdot) ||^p_{ n + 2, p; \,  h} \leq N \gamma  \sum_{m = 0}^{ T/\gamma } \sum_{ i = 1}^3 F_{i} (t_m),
	\end{equation}
	where
	$$
		F_q (t_m) = E || f_{q} (t_m, \cdot)   ||^p_{n, p;  \, h}, \, q = 1, 2, 3.
	$$
	By  Minkowski inequality  and H\"older's inequalities and \eqref{14.1} we have
	$$
		F_{1} (t_m) \leq \sum_{q=1}^3 F_{1, q} (t_m),
	$$
	where
	$$
		F_{1, 1} (t_m) =  \gamma^{-1}   \int_{t_m}^{t_{m+1}} E  ||(a^{i j} (t, \cdot)  - a^{i j} (t_m, \cdot))  D_{i j} v_m (t, \cdot)||^p_{n, p; \, h}   \, dt,
	$$
	 $$
		F_{1, 2} (t_m) =  \gamma^{-1}   \int_{t_m}^{t_{m+1}}  E ||a^{i j} (t_m, \cdot)  (D_{i j} v_m (t, \cdot) - D_{i j} v_m (t_m, \cdot))||^p_{n, p; \, h}   \, dt,
	 $$
	  $$
		F_{1, 3} (t_m) = E ||   \lambda_{k} (t_m, \cdot)   (D_{(\lambda_k) (\lambda_k) } v_m (t_m, \cdot)  - \Delta_{l_k, h} v_m (t_m, \cdot))||^p_{n, p; \, h}.
	  $$
	Similarly, 
	$$
		F_2 (t_m) \leq F_{2, 1} (t_m) + F_{2, 2} (t_m), 
	$$
	where
	$$
		F_{2, 1} (t_m) \leq  \gamma^{-1}   \int_{t_m}^{t_{m+1}} E || (b^i (t, \cdot) - b^i (t_m, \cdot)) D_i v_m (t, \cdot)||^p_{n, p; \, h} \, dt,
	$$
	 $$
		F_{2, 2} (t_m) \leq \gamma^{-1}   \int_{t_m}^{t_{m+1}} E || b^i (t_m, \cdot) (D_i v_m (t, \cdot) - D_i v_m (t_m, \cdot))||^p_{n, p; \, h} \, dt,
	 $$
	  $$
	 	F_{2, 3} (t_m) = E  || b^i (t_m, \cdot) (D_i v_m (t_m, \cdot) - \delta_{e_i, h} v_m (t_m, \cdot)||^p_{n, p; \, h}.
	  $$

	\textit{Estimate of $F_{1, 1}, F_{2, 1}$}.
	By Lemma \ref{lemma 9.2} and Assumption \ref{assumption 3.5} 
	$$
		F_{1, 1} (t_m) \leq  \gamma^{-1}    \int_{t_m}^{t_{m+1}} E  ||(a^{i j} (t, \cdot)  - a^{i j} (t_m, \cdot))||_{ C^{n-1, 1} } || D_{i j} v_m (t, \cdot)||^p_{ n, p; \, h }  \, dt
	$$
	  $$
		\leq  \gamma^{\theta p - 2}  E \int_{t_m}^{t_{m+1}}  || D_{i j} v_m (t, \cdot)||^p_{ n, p; \, h }  \, dt.
	  $$
	Further, by Lemma \ref{lemma 9.3} and Lemma \ref{lemma  5.5} $(i)$ we get
	\begin{equation}
				\label{14.4}
		F_{1, 1} (t_m) \leq  N  \gamma^{\theta p - 1}    E \sup_{t \in [t_{m-1}, t_m]} ||v_m (t, \cdot) ||^p_{ H^{s+2}_p }
		\leq N \gamma^{\theta p - 1} \cI_m,
	\end{equation}
	where $\cI_m$ is defined in Lemma \ref{lemma  5.5} $(i)$.
	By the same argument
	\begin{equation}
				\label{14.5}
	F_{2, 1} (t_m)  \leq N \gamma^{\theta p - 1} \cI_m.
	\end{equation}

	\textit{Estimate of $F_{1, 2}$ and $F_{2, 2}$.}
	Repeating the argument of the previous paragraph we get
	$$
		F_{1, 2} (t_m) \leq N \gamma^{-1} 
	 \int_{t_{m-1}}^{t_m} E || a^{ i j} (t_m, \cdot) ||^p_{ C^{n - 1, 1} } 
																\,  || v (t_m, \cdot) -  v_m (t, \cdot) ||^p_{  H^{s+2}_p } \, dt.
	$$
	By using Assumption \ref{assumption 3.4.2} and \eqref{8.4} with $\varepsilon < 1 - 2 \mu$
	 we obtain
	\begin{equation}
				\label{14.6}
		F_{1, 2} (t_m) \leq N \gamma^{\theta p - 1}  E  || v_m ||^p_{ C^{\theta - 1/p} ([0, T], H^{s+2}_p) }
		 \leq   N \gamma^{\theta p - 1} \cI_m.
	\end{equation}
	Similarly, 
	$$
		F_{2, 2} (t_m) \leq  N \gamma^{\theta p - 1} \cI_m.
	$$
	
	\textit{Estimate $F_{1, 3}$ and $F_{2, 3}$.}
	First,  by Lemma \ref{lemma 9.2} and Assumption \ref{assumption 3.4.1}
	$$
		F_{1, 3} (t_m) \leq N  E ||\lambda_k (t, \cdot) ||^p_{C^{n-1, 1}}  || D_{ (\lambda_k) (\lambda_k)} v_m (t_m, \cdot)  - \delta_{l_k, h} v_m (t_m, \cdot) ||^p_{  n, p; \, h }.
	$$
	  $$
		\leq  N  || D_{ (\lambda_k) (\lambda_k)} v_m (t_m, \cdot)  - \delta_{l_k, h} v_m (t_m, \cdot) ||^p_{  n, p; \, h }.
	  $$
	Since 
	$
		v \in \cH^{s+3}_p (t_{m-1}, t_m),
	$ 
		by Remark \ref{remark 2.1}
	we have
	  $
		v_m (t, \cdot) \in C^3.
	  $
	Then, by Taylor's formula with the remainder in the integral form combined with Minkowski inequality we get
	$$
		F_{1, 3} (t_m) \leq  N h^p  \sum_{ ||\alpha|| = 3} E || D^{\alpha}_x  v_m (t_m, \cdot) ||^p_{ n, p; \, h }.
	$$
	Next, by Lemma \ref{lemma 9.3}, for any $\varepsilon > 0$,
	$$
		F_{1, 3} (t_m) \leq N h^p E || v_m (t_m, \cdot) ||^p_{ H^{n+3 + d/p + \varepsilon}_p },
	$$
	Taking  
	$
		\varepsilon < s - n - d/p  - 2 \mu
	$ 
	and using  \eqref{8.4}, we get
	\begin{equation}
				\label{14.7}
		F_{1, 3} (t_m) \leq N h^p \cI_m.
	\end{equation}
	By the same argument
	\begin{equation}
				\label{14.8}
	F_{2, 3} (t_m) \leq N h^p \cI_m.
	\end{equation}

	Finally, by Lemma \ref{lemma 9.3} and Assumption \ref{assumption 3.7}
	$$
		F_4 (t_m) \leq N \gamma^{\theta p - 1} E || f ||^p_{ C^{\theta p - 1}  ([0, T], H^{s_1}_p) }. 
	$$
	Combining this with \eqref{14.3} - \eqref{14.8}
	 and keeping in mind that 
	$
		\gamma \leq \rho h^2,
	$ we get
	$$	
		\gamma  E \sum_{m = 0}^{T/\gamma} || e_h (t_m, \cdot) ||^p_{ n + 2, p; \, h} \leq N h^{2 (\theta p - 1)} \cI_{T/\gamma} + E || f ||^p_{ C^{\theta p - 1}  ([0, T], H^{s_1}_p) }. 
	$$
	Now the first assertion of this theorem follows from what was just proved and Lemma \ref{lemma 5.5} $(ii)$.

	To prove the second assertion we apply Theorem \ref{theorem 3.2} to \eqref{14.2}
	and we get
	$$
		E  \max_{ m = 1, \ldots, T/\gamma} || e_{\gamma, h} (t_m, \cdot)||^p_{n+1, p; \, h} 
			\leq N \gamma \sum_{m = 0}^{T/\gamma} \sum_{i = 1}^3 || F_i (t_m) ||^p_{  n+2, p; \, h}.
	$$
	By \eqref{14.4} - \eqref{14.8} we conclude that
	$$
		E  \max_{ m = 1, \ldots, T/\gamma} || e_{\gamma, h} (t_m, \cdot)||^p_{n+1, p; \, h}  \leq   N h^{2 (\theta p - 1)} \cI_{T/\gamma} +E || f ||^p_{ C^{\theta p - 1}  ([0, T], H^{s_1}_p) }. 
	$$
	Combining this with \ref{lemma 5.5} $(ii)$ we prove the desired claim.

	\mysection{Appendix A}
						\label{section 11}
	In this section we prove several results
	 for the discrete heat kernel $G$ defined by \eqref{3.4}.
	One of the objectives of this section is to establish difference estimates of $G$.
	In case $d_1 = d$, 
	$
		l_k = e_k, \lambda_k \equiv 1
	$, 
	$
	    k \in \{ 1, \ldots, d\}
	$,  
	such estimates can be found, for example, in Section 2.3 of \cite{LL} (see Theorem 2.3.6 of \cite{LL}).
	To prove  this result  (see Corollary \ref{corollary 11.3}) we will use the argument of Section 2.3 of \cite{LL}.

	Let  $g$ be a function on $\bZ^d$ with compact support.
	Recall that Fourier transform is defined by \eqref{5.0}.
	It is well-known that the following versions of the inversion formula and the Parseval's identity hold:
	\begin{equation}
				\label{11.0}
		g (x) = (2 \pi)^{-d/2} \int_{[-\pi, \pi]^d} \exp ( i \xi \cdot x) \cF g (\xi) \, d\xi,
	\end{equation}
	 \begin{equation}
				\label{11.1}
		\sum_{x \in \bZ^d} |g|^2 (x) = \int_{[-\pi, \pi]^d} |\cF g (\xi)|^2 \, d \xi.
	  \end{equation}
	In addition, if $f: \bZ^d \to \bR$ has  compact support, then
	\begin{equation}
				\label{11.2}
	\cF {(f \ast_1 g)} (\xi) = \cF f (\xi) \cF g (\xi).
	\end{equation}

	\begin{lemma}
				\label{lemma 11.1}
	Let $d_1$ be a positive integer, 
	$\rho > 0$ be a number,
	$
		l_1, \ldots, l_{d_1} \in \bZ^d,
	$
	and
	 $
		\lambda_1, \ldots \lambda_{d_1}
	$ 
	be nonnegative functions on $\bZ$
	such that, 
	for any 
	$
		t \in \bZ,
	$ 
	 $
		 2 \rho \sum_{k = 1}^{d_1} \lambda_k (t)  \leq 1.
	 $
	Let $G$ be the function given by \eqref{3.4}.
	Then, the following assertions hold.

	$(i)$ Let $X_n, n \in \bZ$ be a sequence of independent random vectors such that
		$$
			P (X_n = l_k) = P (X_n = -  l_k) =  \rho \lambda_k (n), \, \, 
			 k = 1, \ldots, d_1, 
		$$
		 $$
					 P (X_n = 0) =  1 - 2 \rho \sum_{k = 1}^{d_1} \lambda_k (n).
		 $$
		 For any $n, m \in \bZ$ such that $n > m$,
		and any $x \in \bZ^d$, 
		$$
			G (m, n, x) = P (\sum_{j = m + 1}^n X_j = x).
		$$

 	   $(ii)$ For any $\xi \in \bR^d$, $n > m$, 
		$$
			  \cF G (m, n, \xi)   =  \Pi_{ j = m + 1}^n  (1 - 2\rho \sum_{k = 1}^{d_1} \lambda_k (j))  (1 -  \cos ( \xi \cdot l_k)).
		$$
	\end{lemma}

	\begin{proof}
	$(i)$  Observe that by \eqref{3.4}, for $n > m$, 
		$$
			G (m, n, x) = G (m, n - 1, x) (1 -  2 \rho \sum_{k = 1}^{d_1} \lambda_k (n - 1)) 
		$$
		 $$
			+  \rho \sum_{k = 1}^{d_1} \lambda_k (n - 1) (G (m, n - 1, x + l_k) + G (m, n - 1, x - l_k)).
		 $$
		Then, by the independence of $X_n, n \in \bZ$
		and the fact that 
		$
			G (m, m, x) = I_{x = 0},
		$ 
		we obtain
		 \begin{equation}
						\label{11.1.1}
			G (m, n, x) = E G (m, n - 1, x + X_n)
		 \end{equation}
		 $$	
			 = E G (m, m, x + \sum_{j = m+1}^n X_j) = P (\sum_{j = m+1}^n X_j = x).
		 $$

	$(ii)$ By the first assertion
		$$
			\cF G (n, m, \cdot) (\xi) = E \exp (-i \xi \cdot \sum_{ j = m+1 }^n X_j) 
		$$
		 $$
			=  \Pi_{j= m+1}^n E \exp (-i \xi \cdot X_j)  = \Pi_{ j = m + 1}^n  (1 - 2\rho \sum_{k = 1}^{d_1} \lambda_k (j))  (1 -  \cos ( \xi \cdot l_k)).
		 $$

	\end{proof}

	In the next lemma we present a discrete variant of Duhamel's principle.

	\begin{lemma}
				\label{lemma 11.6}
	Assume the conditions of Lemma \ref{lemma 11.1}.
	Let  $\zeta$ be a function on 
	$
		    \bZ^{d+1},
	$ 
	such that, 
	for any $t \in  \bZ$, 
	$
		\zeta (t, \cdot) \in l_1 (\bZ^d) 
	$
	and let $u$ be the solution
 	 the following  equation:
	\begin{equation}
				\label{11.6.1}
	 \begin{aligned}
		 &z (n+1, x)  - z (n, x)  = \rho  \sum_{j = 1}^{d_1} \lambda_j (n) \Delta_{l_j, h} z (n, x)   +   \zeta (n, x), \\
		 &z (m, x) = 0, \,  n \in [m, \infty) \cap \bZ, x \in  \bZ^d. 
	\end{aligned}
	\end{equation}

	Then, for any integer $n > m$
	and
	$x \in  \bZ^d$,
	\begin{equation}
				\label{3.5.0}
		 u (n, x)  =  \sum_{j = m}^{n - 1}    G (j, n-1, \cdot) \ast_1  \zeta (j, \cdot) (x).
	\end{equation}
	In particular, if $n \geq m + 2$,
	$$
		u (n, x)  =  \sum_{j = m}^{n - 2}    G (j, n-1, \cdot) \ast_1  \zeta (j, \cdot) (x) + \zeta (n-1, x).
	$$

	\end{lemma}

		\begin{proof}
	
	Let $X_n, n \in \bZ$ be the random vectors defined in Lemma \ref{lemma 11.1}.
	Then, we can rewrite \eqref{11.6.1} as follows:
	$$
		z (n+1, \cdot) = E z (n, \cdot + X_n) + \zeta (n, \cdot), \, z (m, \cdot) = 0, n \geq m.
	$$
	It suffices to check that the function $u$ satisfies this equation.
	We do it by   Lemma \eqref{11.1.1} as follows: 
	$$
			 \sum_{j = m}^{n - 1}  E  G (j, n-1, \cdot + X_n) \ast_1  \zeta (j, \cdot) (x) + \zeta (n, x)
	$$
	 $$
		= \sum_{j = m}^{n - 1} G (j, n, \cdot) \ast_1 \zeta (j, \cdot) + G (n, n, \cdot) \ast_1 \zeta (n, \cdot) (x) = u (n+1, x).
	 $$
	
	\end{proof}

	\begin{lemma}
				\label{lemma 11.2}
	Let Assumptions \ref{assumption 3.1} and \ref{assumption 3.2} hold.
	Assume that $a$ is independent of $\omega$ and $x$,
	and
	$
		2 \rho d_1 \delta_2^{*} < 1.
	$
	Denote 
	$$
		A_n  =  \rho \sum_{j = 1}^n a (j)/n,
	$$ 
	$$
		 P_n (x) = (2\pi n)^{-d/2} (\det A_n )^{-1/2} \exp ( - x \cdot A^{-1}_n x/n).
	$$
	Then, 
	for any 
	$
		\varepsilon \in (0, 1),
	$ 
	there exists  a number
	$
		\theta \in (0, 1),
	$ 
	depending only on 
	$
		d_1, \delta_1^{*}, \delta_2^{*}, \rho, \varepsilon
	$,
	and a continuous function $\psi$,
	and 
	$
		N (d, d_1, \rho, \delta_1^{*}, \delta_2^{*}, \varepsilon)
	$ 
	such that as $n \to \infty$, 
	we have
	$$
		G (0, n, x)  = P_n (x) + R_n (x) + O (\theta^n),
	$$
	where
	$$
		 R_n (x)  = (2 \pi n)^{-d/2} \int_{ ||\xi|| \leq \varepsilon \sqrt{n} } \exp ( i \xi \cdot x/\sqrt{n}) \exp (-  \xi \cdot A(n) \xi) \psi (\xi/\sqrt{n})  \, d\xi,
	$$
	and
	$$
		|\psi (\xi)| \leq N |\xi|^4, \, \forall || \xi|| \leq \varepsilon.
	$$
	\end{lemma}

	\begin{proof}
	Fix some 
	$
		\varepsilon \in (0, \pi),
	$
	and $n \geq 1$.
	By \eqref{11.0} and the change of variables 
	$
		\xi \to \xi/\sqrt{n},
	$
	 \begin{equation}
				\label{11.2.1}
	  \begin{aligned}
		& G (0, n, x) 
		= G_1 (0, n, x) + G_2 (0, n, x)\\
		&:= (2\pi n)^{-d/2}  (  \int_{  \varepsilon \leq ||\xi/\sqrt{n}|| \leq \pi  +  ||\xi/\sqrt{n}|| \leq \varepsilon} )
		 \exp ( i  \xi \cdot x/\sqrt{n}) \cF G (0, n, \xi/\sqrt{n}) \, d\xi.
	 \end{aligned}
	\end{equation}
	By Lemma \ref{lemma 11.1} $(ii)$, 
	for any 
	$
		\xi \in \bR^d, 
	$
	$$
		 E \exp ( - i \xi \cdot X_n)  = 1  - \sum_{k = 1}^{ d_1 } 2 \rho \lambda_k (n) (1 -  \cos (\xi \cdot l_k)),
	$$
	 $$
		\cF G (0, n, \xi)  
															=  \Pi_{ j = 1}^n  [1  - \sum_{k = 1}^{ d_1 } 2 \rho \lambda_k (j) (1 -  \cos (\xi \cdot l_k)) ].
	 $$
	We start with $G_1$.
	Note that, 
	for any 
	$
		n \in \bZ_{+},
	$ 
	$$
		E \exp ( - i \xi \cdot X_n)  
											\geq  1 - 4 \rho d_1 \delta^{*}_2 >  - 1, \, \, \forall \xi. 
	$$ 
	Fix any 
	$
		\xi \in \{\varepsilon \leq ||\xi|| \leq \pi\}.
	$
	Since 
	$
		l_k = e_k, k = 1, \ldots, d,
	$
	 there exists $k$ (depending on $\xi$) such that 
	$
		\xi \cdot l_k \in [\varepsilon,  \pi],
	$
	and, then
	$$
		E \exp ( - i \xi \cdot X_n)  \leq  1 - 2 \rho \delta_1^{*}    (1 -  \cos (\varepsilon)).
	$$
	Hence, there exists 
	$
		\theta \in (0, 1)
	$
	independent of $n$
	 such that
	$$
		|\cF G (0, n, \xi)| < \theta^n, \quad \forall \,  \xi \in \{ \varepsilon \leq  ||\xi|| \leq  \pi\}, \, \forall n \in \bN,
	$$
	and this yields that, for any $n \in \bN$ and  $x \in \bZ^d$,
	\begin{equation}
				\label{11.2.2}
		G_1 (0, n, x) \leq N  \theta^n.
	\end{equation}

	Next, we move to $G_2$.
	 By Taylor's formula, for 
	$
		||\xi|| \leq \varepsilon,
	$ 
	 $$
		|E \exp (- i \xi \cdot X_n) -
															 (1 - \sum_{k=1}^{d_1}  \rho \lambda_{k} (n) (\xi \cdot l_k)^2)| \leq N  |\xi|^4,
	 $$
	and, then,
	$$
		\cF G (0, n, \xi/\sqrt{n})  	=   \exp ( -  \sum_{j = 1}^n \sum_{k=1}^{d_1}  \rho \lambda_{k} (j)  (\xi \cdot l_k)^2/n) ( 1+ \psi (\xi/\sqrt{n})),
	 $$
	where $\psi$ is a continuous function such that 
	 $
		|\psi (\xi)|  \leq N |\xi|^4.
	$
	By Assumption \ref{assumption 3.2}  $(iii)$ 
	$$
		  \sum_{k=1}^{d_1}   \lambda_{k} (n) (\xi \cdot l_k)^2 =   \xi  \cdot a (n)  \xi,
	$$
	and, then, 
	$$
		\cF G (0, n, \xi/\sqrt{n})  = 	\exp ( - \xi \cdot A_n  \xi) ( 1+ \psi (\xi/\sqrt{n})).
	$$
	By what was just said
	\begin{equation}
				\label{11.2.3}
		G_2 (0, n, x) = R_n (x) + G_{2, 1} (0, n, x) - G_{2, 2} (0, n, x), 
	\end{equation}
	 where
	$$
		G_{2, 1} (0, n, x) =    (2\pi n)^{-d/2} \int_{\bR^d} 
		 \exp ( i  \xi \cdot x/\sqrt{n}) \exp ( - \xi \cdot A_n  \xi)    \, d\xi.
	$$
	 $$
		G_{2, 2} (0, n, x) =   (2\pi n)^{-d/2} \int_{   ||\xi/\sqrt{n}|| \geq \varepsilon}
		 \exp ( i  \xi \cdot x/\sqrt{n}) \exp ( - \xi \cdot A_n  \xi)    \, d\xi.
	 $$
	By the change of variables $\xi \to A_n^{-1} \xi$ 
	\begin{equation}
				\label{11.2.4}
		G_{2, 1} (0, n, x) =  P_n (x),
	\end{equation}
	Next, in the expression for $G_{2, 2}$ we replace the domain 
	of integration by 
	$
		\{|\xi/\sqrt{n}| \geq  \varepsilon\}.
	$
	and use the change of variables  $\xi \to A^{1/2}_n \xi$.
	Then, by the fact  that the eigenvalues of $A_n$ are bounded from below by 
	$
		\rho \delta_{1}^{*}
	$
	we get
	$$
		 |G_{2, 2} (0, n, x)|
		 \leq N  (\rho, \delta_1^{*}, d) \int_{  |\xi| \geq  (\rho \delta_1^{*})^{-1/2} \varepsilon \sqrt{n} }  \exp ( -  |\xi|^2)  \, d\xi.
	   $$
	It is well know that the last integral is bounded by 
	$$
		N_1  \exp (-N_2 n^2)
	$$
	where $N_1$ and $N_2$ depend only on $\rho, \delta_1^{*}, \varepsilon$.
	 This combined with \eqref{11.2.1} - \eqref{11.2.4} proves the claim.
	\end{proof}
	
	\begin{corollary}
				\label{corollary  11.3}
	Invoke the assumptions and notations of Lemma \ref{lemma 11.2}.
	Let $m \in \bN$ be a number, 
	and 
	$
		\eta_1, \ldots \eta_m \in \bZ^d
	$
	 be  vectors.
	Denote
	$$
		L = \delta_{\eta_1, 1} \ldots \delta_{\eta_m, 1}.
	$$
	Then,  
	for any $n \in \bN$, and $x \in \bZ^d$, 
	$$
		|L G (0, n, x)| 
	$$
	 $$
		\leq N n^{- (d + m)/2} (||x/\sqrt{n}||^{m} + 1) \exp (- (\rho \delta_2^{*} n)^{-1} ||x||^2)  +  N n^{-  (d + m)/2 - 2},
	 $$
	where $N$  depends only on $d,  \delta_1^{*}, \delta_2^{*}, \rho, \eta_1, \ldots, \eta_m$.
	
	\end{corollary}

	\begin{proof}
	Fix some $n \in \bN, x \in \bZ^d$.
	By Lemma \ref{lemma 11.2}
	 we only need to show
	\begin{equation}
				\label{11.3.1}
		| L  P_n (x)| \leq  N n^{- (d + m)/2} (||x/\sqrt{n}||^{m} + 1) \exp (- (\rho \delta_2^{*} n)^{-1} ||x||^2),
	\end{equation}
	\begin{equation}
				\label{11.3.2}
		| L R_n (x)| \leq N n^{-(d + m)/2 - 2}.
	\end{equation}
	
	Let us prove \eqref{11.3.2} first.
	By the mean value theorem  
	$$
		| L R_n (x)| \leq N   (2 \pi n)^{-d/2}  n^{-2} \int_{ ||\xi/\sqrt{n}|| \geq \varepsilon} | L \exp (i \xi \cdot x/\sqrt{n})| \,  |\xi|^4 \exp (-  \xi \cdot A_n \xi)  \, d\xi
	$$
	 $$
		 \leq N   (2 \pi n)^{-d/2}  n^{- m/2 -  2} \int_{ \bR^d }  (|\xi|^{4+m} + 1) \exp (-  \xi \cdot A_n  \xi)  \, d\xi.
	 $$
	Then, by the  change of variables 
	$
		\xi \to A_n^{1/2} \xi
	$  
	 and Assumption \ref{assumption 3.2} 
	  we have
	 $$
		|L R_n (x)| \leq   N n^{ - (d+m)/2 - 2}.
	 $$

	Next, we prove \eqref{11.3.1}.
	By the mean value theorem 
	$$
	 	|L  P_n (x)| \leq  N  \max_{ y \in \bR^d: ||x - y|| \leq K} | D_{(\eta_1)} \, \ldots D_{(\eta_m)} P_n (y)|,
	$$ 
	where 
	$
		K = \sum_{k = 1}^m ||\eta_k||.
	$
	Further, observe that by the triangle inequality we only need to show that,
	\begin{equation}
				\label{11.3.3}
		| D_{(\xi_1)} \, \ldots D_{(\xi_m)}  P_n (x)| \leq N  n^{- m/2} (|x/\sqrt{n}|^{- m/2} +1) P_n (x).
	\end{equation}
	where $N$ is independent of $n$ and $x$.
	 One can reduce the last estimate to the case when $A_n$ is a $d\times d$ identity matrix.
	To justify this, we set
	$
		 y = A_n^{-1/2} x
	$
	and note that the following identities hold:
	$$
		x \cdot A_n x = |y|^2, \quad
												\frac{\partial}{\partial x_i} =  \sum_{j = 1}^d (A_n^{-1/2})^{i j} \frac{\partial}{\partial y_j}, 
		\quad
		|(A_n^{-1/2})^{i j}| \leq (\delta_1^{*})^{-1/2}.
	$$
	
	Finally, to prove \eqref{11.3.3} in case $A_n$ is an identity matrix it suffices to show that, 
	for any
	 $k \in  \bN$,  
	$$
		d^k/dt^k \exp (-  t^2/n) \leq  N (k) n^{-k/2} (|t/\sqrt{n}|^{k} + 1) \exp (-  t^2/n).
	$$
	This easily follows from the fact that 
	$$
		\exp (t^2) \frac{d^k}{dt^k} \exp (-  t^2) 
	$$ 
	is a polynomial of order $k$. 
	This proves \eqref{11.3.1}.
	\end{proof}

	\begin{lemma}
				\label{lemma 11.4}
	Let assumptions of Lemma \ref{lemma 11.1} hold.
	Let $n, r \in \bN$ be numbers.
	 Then,  there exists $\beta > 0$
	depending only on 
	$
		l_1, \ldots, l_{d_1}
	$
	 such that
	$$
		\sum_{||x|| \geq r} G (0, n, x) \leq 2d  \exp (- \beta r^2/n).
	$$
	\end{lemma}
	
	\begin{proof}
	By Lemma \ref{lemma 11.1}
	 $
		G (0, n, x) = P (S_n =  x),
	$
	 where 
	$
		S_n  = \sum_{k = 1}^n X_k, n \in \bN,
	$
	 and $X_k$ is defined in the aforementioned lemma.
	By what was just said
	$$
		\sum_{||x|| \geq r} G (0, n, x) =  P (||S_n|| \geq r) \leq \sum_{j = 1}^d P (|S_n^j| \geq r).
	$$
	Note that $S_n^j $ is a sum of independent random variables   bounded by 
	$
		L: =  \max_{k = 1, \ldots, d_1} ||l_k|| > 0.
	$
	Then, by the concentration inequality (see Theorem 2 of \cite{H}) we have
	$$
		P (|S_n^j| \geq r) \leq 2 \exp (-  r^2 ( 2 n  L)^{-1}).
	$$
	The lemma is proved.
	\end{proof}

		\begin{lemma}
					\label{lemma 11.5}
 	Let the assumptions of 
	 Lemma \ref{lemma 11.2} hold.
	Let
	$
		r, k \in \bZ_{+},
	$ 
	$m, n \in \bN$,
	 be numbers,
	and 
	$
		\eta_1, \ldots, \eta_m \in \bZ^d
	$
	 be  vectors.
	Assume that $m \geq k$.
	Denote
	$$
		L = \delta_{\eta_1, 1} \ldots \delta_{\eta_m, 1},
	$$
	$$
		S (r, n) = \sum_{ x \in \bZ^d: ||x|| \geq r  } ||x||^k |L G (0, n, x)|.
	$$
	Then, the following assertions hold.

	$$
		 (i) \, S (0, n)   \leq N n^{ (k  - m)/2 }.
	$$

	 $$
		(ii) \, S (r, n)   \leq    N
	 (n^{  (k  - m + 1)/2} r^{-1}) \wedge
		 (n^{(k - m  + 2)/2} r^{-2}).
	 $$

	In both assertions $N$ is a constant depending only on 
	$d, d_1, \delta_1^{*}, \delta^{*}_2$, 
	$\Lambda, \eta_1, \ldots, \eta_m, m, k$.
	\end{lemma}

	\begin{proof}
	Both claims are derived from Corollary \ref{corollary 11.3}.
	 In particular, because of the error term in the local limit theorem, 
	we  to split the sum $S (r)$ into two parts:
	the sum over $x$ with the large norm (the `tail' term) and the sum of  Gaussian terms.
	The `tail' term is handled by Lemma \ref{lemma 11.4}
	and 	the remainder is estimated by Corollary \ref{corollary 11.3}.

	\textit{Proof of $(ii)$.}
	Fix some number 
	$r \in \bN$,
	and  denote
	$
		\delta  = (k+d)^{-1}.
	$
	Observe that due Lemma \ref{lemma 11.1}
	we have
	$
		G (0, n, x) = 0,
	$
	if 
	$
		||x|| > n \max_{j = 1, \ldots, d_1} ||l_j|| = : n  L.
	$
	Hence, in the expression for $S(r)$ 
	the sum is taken over the domain
	$
		r \leq || x || \leq  L n.
	$

	Next, note that in the case
	 $
		r \geq n^{1/2 + \delta},
	$
	we split the sum into two parts as follows:
	$$
		S (r, n) \leq R (n): = \sum_{  n^{1/2 + \delta  } \leq || x|| \leq n L }  ||x||^k |L G (0, n, x)|.
	$$
	Further, if 
	$
		r < n^{1/2 + \delta}
	$, 
	then we split $S$ into two parts as follows:
	 $$
		S (r, n) = (\sum_{r \leq ||x|| \leq n^{1/2 + \delta} } + \sum_{  n^{1/2 + \delta  } \leq || x|| \leq n L })  \ldots =: S_1 (n, r) + R (n).
	 $$
	Hence, it suffices to consider
	 the case 
	$
		r < n^{1/2 + \delta}.
	$ 

	We start with $S_1 (r, n)$. 
	By Corollary \ref{corollary 11.3}
	\begin{equation}
				\label{11.5.3}
		S_1 (r, n) \leq  N (S_{1, 1} (r) + S_{1, 2} (r)),
	\end{equation}
	where
	\begin{equation}
				\label{11.5.4}
		S_{1, 1} (r, n)  =\sum_{ r \leq ||x|| \leq n^{1/2 +\delta} } n^{ - (d + m)/2} ||x||^k  (||x/\sqrt{n}||^{m}  +  1) \exp (-  \varkappa ||x||^2/n),
	\end{equation}
	 \begin{equation}
				\label{11.5.5}
		S_{1, 2} (r, n) =  n^{ - (d+m)/2 - 2}   \sum_{  ||x|| \geq n^{1/2 +\delta} } || x ||^k \leq N n^{ (k  - m)/2   + (k + d)\delta - 2},
	 \end{equation}
	and 
	$
		\varkappa  =  (\rho \delta_2^{*})^{-1}.
	$
	Using the fact that
	 the summand in 
	$
	  S_{1, 1} (r, n)
	$
	 is constant on $||x|| = t$,
	 we get
	$$
		S_{1, 1} (r, n) \leq
		 N n^{ -(d + m)/2 } \sum_{t = r}^{ n^{1/2 + \delta} }  t^k ( t + 1)^{d-1} ((t/\sqrt{n})^{ m  } + 1) \exp (-\varkappa t^2/n).
	$$
	To simplify this sum we use the fact that, for $t \geq 1$, 
	$$
		 t^k ( t + 1)^{d-1} ((t/\sqrt{n})^{ m  } + 1)  \leq N (k, d, m, \varkappa) n^{(k + d - 1)/2} \exp (- \varkappa t^2/(2n)).
	$$
	By what was just said 
	\begin{equation}
				\label{11.5.6}
	  \begin{aligned}
	&	S_{1, 1} (r, n)  \leq N n^{(k - m -1)/2} \int_{t = r - 1}^{\infty}  \exp (-\varkappa t^2/(2n)) \, dt \\\
		& \leq    N n^{ (k - m)/2 } \int_{t = (r - 1)/\sqrt{n} }^{\infty}  \exp (-\varkappa t^2/2) \, dr.
	 \end{aligned}
	 \end{equation}
	It is easy to check that, for any $c > 0$,
	$$
		\int_c^{\infty} \exp (- \varkappa t^2/2) \, dt \leq N (\varkappa)  (c^{-1} \wedge c^{-2}).
	$$
	Hence, by what was just said and \eqref{11.5.6}
	\begin{equation}
				\label{11.5.7}
		S_{1, 1} (r, n) \leq N    (n^{(k - m + 1)/2} r^{-1}) \wedge (n^{ (k - m+2)/2} r^{-2}).
	\end{equation}
	Further, since 
	$
		r \leq   n L,
	$ 
	and
	$
		\delta  = (k+d)^{-1},
	$ 
	we have
	\begin{equation}
				\label{11.5.8}
		S_{1, 2} (r, n) \leq N   (n^{ (k - m)/2}/r) \wedge (n^{(k - m + 2)/2}/r^2). 
	\end{equation}

	We move to the term $R (n)$.
	 Denote 
	$
		K = \sum_{k = 1}^m ||\eta_k||.
	$
	By Lemma \ref{lemma 11.4}
	 there exists $r_0 > 1$ depending only $K$
	such that for $n \geq r_0$, 
		 we have
	\begin{equation}
				\label{11.5.1}
	 \begin{aligned}
		R  (n) & \leq N n^{ k  } \sum_{ ||x|| \geq n^{1/2 + \delta }} \sum_{ y: ||x - y|| \leq   K } G (0, n, y) \\
	&	\leq N n^{ k  } \sum_{ ||x|| \geq n^{1/2 + \delta }/2  }  G (0, n, x)  \leq N n^k \exp (- \beta n^{2\delta}/4 ).
	\end{aligned}
	 \end{equation}
	Adjusting a constant $N$, we conclude that the last inequality holds for all $n \in \bN$.
	Now the assertion $(ii)$ follows from \eqref{11.5.7}, \eqref{11.5.8} and \eqref{11.5.1} combined with the fact that $r \leq n L$.

	\textit{Proof of $(i)$.} We
	 repeat almost word for word the above argument.
	First, by \eqref{11.5.5}
	and 
	  \eqref{11.5.1}
	$$
		S_{1, 2} (0, n) + R (n)  \leq N n^{(k - m)/2}.
	$$
	Next,
	by Corollary \ref{corollary 11.3}
	we have 
	$$
		S_{1, 1} (0, n) \leq  N n^{ - (d + m)/2}    + \tilde S_{1, 1} (0), 
	$$
	where
	 $$
		\tilde S_{1, 1} (0, n) =  \sum_{ ||x|| \geq 1 } n^{ - (d + m)/2} ||x||^k  (||x/\sqrt{n}||^{m}  +  1) \exp (-\tilde \varkappa ||x||^2/n).
	$$
	 Repeating the argument used in \eqref{11.5.6}, we obtain
	$$
		\tilde S_{1, 1} (0, n) \leq N n^{(k - m)/2} \int_0^{\infty} \exp (-\tilde \beta t^2/2) \, dt \leq N n^{(k - m)/2},
	$$
	and this finishes the proof of $(i)$.
	\end{proof}

	\mysection{Appendix B}
					\label{section 9}
	\begin{lemma}
				\label{lemma 9.1}
	Let $T, \gamma >  0$
	be numbers
	such that 
	$
		T/\gamma \in \bN.
	$
	Let  
	$
	  a_k,  k = 1, \ldots, T/\gamma,
	$
	 and $b, K$ be nonnegative numbers.
	Assume 
	$a_1 \leq b$
	and that, 
	 for any 
	$n \in \{2, \ldots, T/\gamma\}$, 
	$$
		 a_n \leq b + K  \gamma \sum_{k = 1}^{n-1} a_k.   
	$$
	Then, 
	for any 
	$
		n  \in \{ 1, \ldots,  T/\gamma\}
	$
	 we have
	$$
		a_n \leq   \exp (KT) b.
	$$
	\end{lemma}	

	\begin{proof}
	 We claim that for $n \geq 1$ 
	$$
		a_n \leq b (1 + K \gamma)^{n-1}.
	$$ 	 
	Proof by induction. 
	First, the claim holds for $n = 1$.
	Next, by summing the geometric progression we get
	$$
		a_{n+1} \leq b +  b K \gamma \sum_{ j = 1}^{n} (1+ K\gamma)^{ j - 1} 
		 \leq b (1 + K\gamma)^{n}.
	$$ 
	Finally, 
	 for any 
	$
		n \in \{1, \ldots, T/\gamma\}
	$
	we have
	$
		(1 + K\gamma)^{n} \leq \exp ( K T),
	$
	and this finishes the proof.
	\end{proof}
	
	\begin{lemma}
				\label{lemma 9.2}
	Let $s \in \bN$,
	 $a \in C^{s - 1, 1}$,
	 and 
	$u \in l_p (h \bZ^d)$.
	Then, 
	$$
		|| a u ||_{ s, p; \, h} \leq N (s) ||a||_{C^{s-1, 1}}  || u ||_{s, p; \, h}.
	$$
	\end{lemma}
	
	\begin{proof}
	By \eqref{7.9}, for any multi-index $\alpha$ and $x \in h \bZ^d$, 
	$$
		\delta^{\alpha}_h (a (x) u (x)) = \sum_{\alpha'  \leq \alpha} N (\alpha', \alpha) (T^{\alpha'}_h \delta^{\alpha'}_h a (x)) (\delta^{\alpha - \alpha'}_h u (x)).
	$$
	Now the claim follows from the fact that $a \in C^{s - 1, 1}$.
	\end{proof}

	\begin{lemma}
				\label{lemma 9.5}
	$(i)$
	Let $f$ and $g$ be  functions on $\bZ$,
	and
	  $t_i \in \bZ, i = 1, 2$ be numbers such that $t_2 >  t_1$.
	Then,  we have
	\begin{equation}
				\label{9.5.0}
		\sum_{ t = t_1 }^{t_2} f (y) g (y) = f (t_2)  \sum_{x = t_1}^{t_2} g (x)   = \sum_{x = t_1}^{t_2 - 1} (f (y+1) - f (y)) \sum_{t = t_1}^x  g (y). 
	\end{equation}

	$(ii)$ Let $g \in l_1 (\bZ^d)$ and $f$ be function on $\bZ^d$ such that
	$
		f (x) \to 0
	$
	as $||x|| \to \infty$.
	Then, 
	$$
		 \sum_{ x \in \bZ^d} |f (x) g (x)| \leq \sum_{x \in \bZ^d} |\delta_{e_1, 1} \ldots \delta_{e_d, 1} f (x)| \sum_{ y \in \bZ^d: || y || \leq ||x|| + 1 } |g (y)|.
	$$
	\end{lemma}

	\begin{proof}
	$(i)$ The proof is standard.

	$(ii)$  \textit{Case $d  = 1$}.
		We fix arbitrary positive integer $t$.
		  By the assertion $(i)$ with $t_1 = 0 = t$ and $t_2 = t > 0$
		and the fact that 
		$
			|a| - |b| \leq |a - b|, 
		$
		for any $a, b \in \bR$,
		 we get
		$$
			I_{+}: = \sum_{ x= 0 }^{t} |f (x) g (x)| \leq |f (t)| \sum_{x = 0}^{t} |g (x)|  + \sum_{x = 0}^{t - 1} |f (y+1) - f (y)| \sum_{y = 0}^{x} |g (y)|.
		$$
		 We take a limit as $t \to \infty$.
		Since $f (x) \to 0$ as $|x| \to \infty$,
		 and 
		$
			g \in l_1 (\bZ),
		$
		 we have
		$$
			\sum_{ x= 0 }^{t} |f (x) g (x)| \leq \sum_{x = 0}^{\infty} |f (y+1) - f (y)| \sum_{y = 0}^{x} |g (y_1, y_2)|.
		$$

		Next, replacing $f (x)$ and $g (x)$ by 
		$
			\hat f (x) = f (-x)
		$ and
		 $
			\hat g (x) = g (-x)
		$ 
		and using what we just proved,  
		 we get
		$$
			I_{-}: = \sum_{x= - \infty}^{0} |f (x) g (x)| = \sum_{x = 0}^{\infty}  |\hat f (x) \hat g (x)| 
		$$
		 $$
			\leq  \sum_{x = - \infty}^{0} |f (x) - f (x - 1)| \sum_{y = x}^{0} |g (y)|
			\leq \sum_{x = - \infty}^{0} |f (x+1) - f (x)| \sum_{  y   =   x }^{ - x } |g (y)|.
		$$
		The claim in this case follows from the bounds for $I_{+}$ and $I_{-}$.

		\textit{Case $d > 1$.}
		 We fix $x^i, i = 2, \ldots, d$ 
		and use the assertion $(ii)$ with $d = 1$
		 for the functions
		 $x_1 \to f (x)$, 
		$x_1 \to g (x)$
		We get
		$$
			\sum_{x_1 \in \bZ} |f (x) g (x)|  \leq   \sum_{x_1 \in \bZ} |\delta_{e_1, 1} f (x)| \sum_{   y  = - |x_1|   }^{|x_1|} |g (y, x_2, \ldots, x_d)|.
		$$
		Next, we fix $x_1, x_3, \ldots, x_d$ and 
		apply the assertion $(ii)$ with $d = 1$ for 
		the functions 
		$
			x_2 \to  |\delta_{e_1, 1} f (x)|, 
		$
		and 
		$
			x_2 \to  \sum_{y = -|x_1|}^{|x_1|} |g (y, x_2, \ldots, x_d)|.
		$
		We get
		$$
			\sum_{x_1, x_2 \in \bZ} |f (x) g (x)|   \leq 
			\sum_{x_1, x_2 \in \bZ} |\delta_{e_1, 1} f (x)|  \sum_{   |y| = - |x_1|   }^{ |x_1| } |g (y, x_2, \ldots, x_d)|
		$$
		 $$
		  	\leq  \sum_{x_1, x_2 \in \bZ}   |\delta_{e_1, 1}  \delta_{e_2, 1} f (x)| \sum_{y_1, y_2 \in \bZ: ||y_i|| \leq ||x_i||, i  =1, 2 } |g (y)|.
		 $$
		Iterating this inequalities, we prove the desired estimate.
	\end{proof}

	The next lemma is a version of Lemma 8.0.1 of \cite{D_12}. A similar lemma also appears in \cite{Y_98}.
		\begin{lemma}
				\label{lemma 9.3}
	Let  
	$
	  p > d,  
	$
	$n \in \bZ_{+}$,
	$\varepsilon > 0$,
	and
	 $
		s \in  n + p/d.
	$
	be numbers
	Then, for any  number $h \in (0, 1]$ 
	and any function $f \in H^s_p$,  
	$$	
		|| f ||_{ n, p; \,  h} \leq N (p, d, s) || f ||_{H^s_p}.
	$$
	\end{lemma}

	\begin{proof}
	It suffices to prove the claim for $s \in (n, n+1)$.

	\textit{Case $n = 0$.}
	Let $\Omega$ be a domain.
	We say that  
	$
	  u \in W^{s, p} (\Omega)
	$ 
	if  $u \in L_p$ and 
	$$
		[ u ]_{W^{s, p} (\Omega)} : =   (\int_{\Omega} \int_{\Omega} \frac{|f (x) - f(y)|^p}{|x-y|^{s p + d}} \, dx dy)^{1/p} < \infty. 
	$$
	Fix some  $s_1 \in (d/p, 1)$.
	For $r > 0$ and $x \in \bR^d$, 
	we denote
	$
		B_r (x) = \{y \in \bR^d: |x-y| < r\}.
	$
	 By the embedding theorem (see Theorem $(i)$ in Section 3.3.1 of \cite{T})
	$$
		|f (x)|^p \leq \sup_{ y \in B_1 (0)} |f (x + h y)|^p
	$$
	 $$	 \leq 
				N (p, d, s) (||f (x + h \cdot) ||^p_{ L_p (B_1 (0))} 
																						+ | f (x + h \cdot) |^p_{W^{s_1, p} (B_1 (0))}).
	$$
	By the change of variables $y \to x+hy$ and the fact that  $h \in (0, 1]$, we get
	$$
		||f (x + h \cdot) ||^p_{ L_p (B_1 (0))}  = h^{-d} || f ||^p_{ L_p (B_h (x)) },
	$$
	 $$
		 | f (x + h \cdot) |^p_{ W^{s_1, p} (B_1 (0)) } 
		 = h^{ ps_1 - d } \int_{ B_h (x) } \int_{ B_h (x) } \frac{ |f (y_1) - f (y_2)|^{p} }{ | y_1 -  y_2|^{p s_1 +d} } \, dy_1 dy_2
	 $$
	  $$
			\leq h^{-d} [ f ]^p_{ W^{s_1, p} (B_h (x)) }.
	  $$
	Next, observe that the balls $B_h (x), x \in h \bZ^d$ do not intersect and, then,
	we multiply  the above inequalities by $h^d$ and take a  sum with respect to $x \in h \bZ^d$. 
	We obtain
	$$
		||f ||^p_{p; \,  h} \leq || f ||^p_{ p} + [f]_{W^{s_1, p} (\bR^d)}.
	$$
	Since $s_1 \leq s$, one has (see  Proposition 2 in Section 2.3.2 of \cite{T})
	$$
		 || f ||_{ p} + [f]_{W^{s_1, p} (\bR^d)} \leq N || f ||_{ H^s_p},
	$$
	and this proves the claim in the first case.

	\textit{Case $n \geq  1$}.
	Let $\alpha$ be a multi-index such that $||\alpha|| \leq n$.
	By what was proved in the previous case we have
	$$
		|| \delta^{\alpha}_h f ||_{p;\, h} \leq N || \delta^{\alpha}_h f ||_{   H^{\{s\}}_p  }
		=  ||  (1 - \Delta)^{\{s\}/2} \delta^{\alpha}_h  f ||_{   L_p  }.
	 $$
	The operator $(1-\Delta)^{\{s\}/2}$ commutes with the shifts, 
	and, then one has
	$$
		(1 - \Delta)^{\{s\}/2} \delta^{\alpha}_h  f = \delta^{\alpha}_h (1 - \Delta)^{\{s\}/2} f.
	$$
	Since  $H^k_p$ and $W^k_p$ coincide as sets for $k \in \bZ_{+}$, 
	we have 
	$
		(1 - \Delta)^{\{s\}/2}  f \in W^{\lfloor s \rfloor}_p,
	$
	where the latter is the usual Sobolev space.
	By Theorem 9.1.1 of \cite{Kr_08} applied to $(1 - \Delta)^{\{s\}/2}  f $ we get
	$$
		 ||   \delta^{\alpha}_h (1 - \Delta)^{ \{s\}/2 }  f ||_{   L_p  } \leq || (1 - \Delta)^{\{s\}/2} f ||_{ W^{\lfloor s \rfloor}_p }.
	$$
	Finally,  by the equivalence of norms of $H^k_p$ and $W^k_p$ for $k \in \bZ_{+}$ the last expression is dominated by
	 $$
		 N || (1 - \Delta)^{\{s\}/2} f ||_{ H^{\lfloor s \rfloor}_p } = ||  f ||_{ H^s_p }.
	 $$
	The claim is proved.
	\end{proof}

	\begin{lemma}
				\label{lemma 9.6}
	Let $p, a, b \in \bR$ be numbers
	such that
	$p > 1$
	and $a < b$.
	Let $f$ be an $X$-valued function such that
	$D_t f \in L_p ([a, b], X)$.
	Then, 
	$$
		I: = \int_a^b ||f (r) - f (a) ||^p_X \, dr \leq p^{-1} (b-a)^p \int_a^b || D_t f (r) ||^p_X \, dr. 
	$$
	\end{lemma}

	\begin{proof}

	For any $r \in [a, b]$ by Minkowski and H\"older inequalities
	$$
		||f (r) - f (a) ||^p_X  = ||\int_a^r D_t f (\tau) \, d\tau||^p_X \leq (r-a)^{p-1} \int_a^r || D_t f (\tau) ||^p_X \, d\tau.
	$$
	Then, 
	$$
		I \leq  (\int_a^b (r-a)^{p-1} \, dr) \int_a^b || D_t f (r) ||^p_X \, dr
	$$
	  $$
		\leq p^{-1} (b-a)^p \int_a^b || D_t f (r) ||^p_X \, dr.
	  $$
	\end{proof}

		\begin{lemma}
				\label{lemma 9.4}
	Let $d_0 \geq 1, s \in \bZ_{+}$, 
	  $p \geq 2, T > 0$,  $\gamma, h \in (0, 1]$ 
	be numbers
	and assume that
	$T/\gamma \in \bN$.
	Let
	 $
		f, g^l \in \bH_{s, p; \gamma, h} (T/\gamma),
	$ 
	 $
		l = 1, \ldots, d_0.
	 $
	Assume that, for any 
	$
		\omega
	$
	and
	$
	 n \in \{ 1, \ldots, T/\gamma\}
	$,
	and
	   $x \in h \bZ^d$, 
	we have
	\begin{equation}
				\label{9.4.1}
		 u (t_n, x) =  \sum_{k = 0}^{n-1}  (\gamma f (t_k, x) + (w^l (t_{k+1}) - w^l (t_k)) g^l (t_k, x)).
	\end{equation}
	Then the following assertions hold.

	$(i)$ For any $n \in \{1, \ldots, T/\gamma\}$, 
	$$
		E \max_{m = 1, \ldots, n } [ u (t_m, \cdot) ]^p_{s, p; \, h} \leq N (p, T, s)  \gamma E \sum_{m = 0}^{n - 1} ([ f (t_m, \cdot)]^p_{s, p; \, h} + \sum_{l = 1}^{d_0} [ g^l (t_m, \cdot)]^p_{s, p; \, h}).
	$$ 
	In the case $g^l \equiv 0, l   = 1, \ldots, d_0$ the above inequality holds for all $p \in [1, \infty)$.
	
	$(ii)$ Assume that $\gamma/h^2 < (2d)^{-1}$.
		 Then,
		$$
			E \max_{m = 1, \ldots, T/\gamma} [  u (t_m, \cdot) ]^p_{s + 1, p; \, h}
			 \leq N (d, p, \gamma/h^2,  T, s) h^2 \sum_{m = 0}^{(T/\gamma) - 1} E ([ f (t_m, \cdot) ]^p_{s, p; \, h} 
		$$
	    	 $$
			+ \sum_{l = 1}^{d_0} [ g^l (t_m, \cdot) ]^p_{s + 1, p; \, h} + \sum_{i  = 1}^d [ \Delta_{e_i, h} u (t_m, \cdot) ]^p_{s, p; \, h}).
		 $$
	\end{lemma}		
	
	\begin{proof}
	For the sake of simplicity, we assume $d_0 = 1$
	and denote $w^1 = w$, $g^1 = g$.
	Let $\alpha$ be a multi-index of order $s$.
	Note that the function $\delta^{\alpha}_h u$ 
	satisfies Eq. \eqref{9.4.1} with
	$
		\delta^{\alpha}_h f
	$  
	and 
	$
		\delta^{\alpha}_h g
	$
	instead of $f$ and $g$.
	Hence, we may assume that $s = 0$.

	$(i)$
	First, 
	by convexity of the function $t \to t^p$, 
	for any 
	$
		n \in \{1, \ldots, T/\gamma\},
	$ 
	$x \in h \bZ^d$, 
	$$
	 	|\sum_{k = 0}^{n-1}  \gamma f (t_k, x)|^p  \leq  (T/\gamma)^{p - 1} \gamma^p	\sum_{k = 0}^{n - 1}  |f (t_k, x)|^p  \leq 
																											T^{p-1} \gamma \sum_{k = 1}^{n}  |f (t_k, x)|^p,
	$$
	and this proves the second claim of assertion $(i)$.

	Next, to handle the stochastic term we note that it can be rewritten as a stochastic integral. 
	After that we apply the Burkholder-Davis-Gundy inequality and use the convexity of the function $t \to t^{p/2}$.
	We get
	$$
			E |\sum_{k = 0}^{n-1}   g (t_k, x) (w (t_{k+1}) - w (t_k))|^p 
														= E |\int_0^{t_{n}} \sum_{k = 0}^{ n - 1 }  g (t_k, x) I_{ t_k \leq s < t_{k+1} }  \, dw (s)|^p
	$$
	 $$
		\leq N (p)  (E \sum_{k = 0}^{ n - 1 } \gamma |g (t_k, x)|^2)^{p/2}
															 \leq N (p) \gamma^{p/2} (T/\gamma)^{p/2 - 1} E \sum_{k = 0}^{n - 1} |g (t_k, x)|^{p}.
	 $$

	$(ii)$	It suffices to  consider two cases: $f \equiv 0$ and $g \equiv 0$.
		In the first case the claim  follows from applying $\delta_{e_i, h}$ to $u$ and using the assertion $(i)$
		with $g$ replaced by $\delta_{e_i, h} g$.

		To handle the second case we 
		 note that $u$ is the solution of the following finite difference equation:
		$$
			z (t + \gamma, x) - z (t, x) =   \gamma \sum_{k = 1}^d \Delta_{e_k, h} z (t, x) + \gamma \tilde f (t, x),  \, \, z (0, x) = 0, t \in \bZ_{+}, x \in \bZ^d,
		$$
		where 
		$$
			\tilde f (t, x) = f (t, x) - \sum_{k = 1}^d \Delta_{e_k, h} u (t, x).
		$$
		Then, the assertion is proved if we show that, for any $i$,
		$$
			\max_{m = 1, \ldots, T/\gamma} || \delta_{e_i, h} u  (t_m, \cdot) ||^p_{ p; \, h} \leq \gamma \sum_{m = 0}^{T/\gamma - 1} || \tilde f (t_m, \cdot) ||^p_{p; \, h}.
		$$
		We may assume that $T/\gamma \geq 2$.
		By  Duhamel's principle (see Lemma \ref{lemma 11.6}), 
		for any 
		$
			n \in \{2, \ldots, T/\gamma\},
		$
		and any $i \in \{1, \ldots, d\}$,
		 $$
			\delta_{e_i, h} u (t_n, x) = \gamma \delta_{e_i, h} \tilde f (t_{n - 1}, x) +   \gamma/h \sum_{j = 0}^{n - 2}  ((T_{e_i, h} - 1) G_h (j, n - 1, \cdot))  \ast_h \tilde f (t_j, \cdot) (x),
		 $$
		where 
		$G_h$ 
			is  defined by \eqref{3.4.1}  with 
		 $d_1 = d$, 
		$\lambda_k = e_k, k  = 1, \ldots, d$,
		$
			\delta_1^{*} = \delta_2^{*} = 1,
		$
		and 
		$\rho = \gamma/h^2$.
		Then, by Young's inequality 
		$$
			|| \delta_{e_i, h} u (t_n, \cdot) ||_{p; \, h} 
				\leq   2\gamma/h ||  \tilde f (t_{n-1}, \cdot)||_{p; \, h}
		$$
		 $$
		 	 + \gamma/h \sum_{j = 0}^{n - 2} (\sum_{x \in \bZ^d} | (T_{e_i, h} - 1) G_h (j, n - 1, x)|) \,  || \tilde f (t_j, \cdot)||_{p; \, h}. 
		 $$
		 Note that the assumptions of Lemma \ref{lemma 11.5} hold because $\gamma/h^2 < (2d)^{-1}$.
		Then, by Lemma \ref{lemma 11.5} $(i)$  with $k = 0$ and $m = 1$ 
		and $j \leq n - 2$
		 we get
		$$ 
			\sum_{x \in \bZ^d} | \delta_{e_i, 1} G (j, n - 1,  x)| \leq N (d, \gamma/h^2) (n - j-1)^{-1/2}.
		$$
		By what was just proved,
		 H\"older's inequality and the fact that $\gamma \leq N h^2$  we have
		\begin{equation}
					\label{9.4.2}
			||\delta_{e_i, h} u (t_n, \cdot) ||^p_{1, p; \, h} \leq N (h^p || \tilde f (t_{n-1}, \cdot)||^p_{p; \, h} 
			+   c    \sum_{j = 0}^{ n - 2 }|| \tilde f (t_{j}, \cdot) ||^p_{p; \, h}),
		\end{equation}
		where
		$$
			c = \gamma^p h^{-p} \, ( \sum_{j = 2}^{n } j^{-p (2(p-1))^{-1} })^{p-1}.
		$$
		Finally, 
		since $\gamma, h \in (0, 1)$, 
		and 
		$
			\gamma/h^2  \leq N
		$, 
		we get
		$$
			c \leq \gamma^p h^{-p} \,  (\int_1^{T/\gamma} t^{ - p (2(p-1))^{-1}}  \, dt)^{ p - 1 } 
			\leq (\gamma/h)^p  (T/\gamma)^{ p/2 - 1}  
		$$
		 $$
			\leq  N (p, d, T) \gamma \leq N h^2.
		 $$
		By \eqref{9.4.2} and what was just proved we get
		$$
			\max_{n = 1, \ldots, m} || \delta_{e_i, h} u (t_n, \cdot) ||^p_{1, p; \, h} \leq N (h^p + h^2) \sum_{n = 1}^{m - 1}|| f (t_n, \cdot) ||^p_{p; \, h}.
		$$
		Now the claim follows from the fact that $p \geq 2$ and $h \leq 1$.

		\end{proof}

		\mysection{Appendix C}
						\label{section 12}
	Here we state two results from harmonic analysis that are used in Section \ref{section 6}.
		\begin{theorem}[Hardy-Littlewood inequality]
					\label{theorem 12.1}
	For $p > 1$, 
	and 
	$
		f \in l_p (\bZ^{d}),
	$
	one has
	$$
		||\cM_x f ||_{l_p (\bZ^d) } \leq N (p, d) ||f||_{l_p (\bZ^d)}^p. 
	$$
	\end{theorem}
		A direct proof can be found, for instance, in Chapter 5 of \cite{D_12} (see Theorem 5.1.1).
	
	\begin{corollary}
				\label{corollary 12.1}
	$$
		||\cM_t \cM_x f ||_{l_p (\bZ^{d+1}) } \leq  N (p, d) \sum_{ (t, x) \in \bZ^{d+1} } || f ||_{l_p (\bZ^{d+1}) }.
	$$
	\end{corollary}
	Indeed, by using Theorem \ref{theorem 12.1} twice we get
	$$
		\sum_{(t, x) \in \bZ^{d+1}} |\cM_t \cM_x f (t, x)|^p  \leq N (p)  \sum_{ (t, x) \in \bZ^{d+1}  } | \cM_x  f (t, x) |^p
	$$
	 $$
		\leq   N (p, d) \sum_{ t \in \bZ } | f (t, x) |^p.
	 $$

	\begin{theorem}[Fefferman-Stein inequality]
						\label{theorem 12.2}
	Under the conditions of Theorem \ref{theorem 12.1},
	$$
		\sum_{(t, x) \in \bZ^{d+1}} |h (t, x)|^p  \leq  N (p, d) \sum_{(t, x) \in \bZ^{d+1}} | h^{\#} (t, x)|^p.
	$$
	\end{theorem}
	
	\begin{proof}

		Let 
	$
		\{\hat \bQ_{n}, n \in \bZ\}
	$
	 be a sequence of partitions of $\bZ^{d+1}$ defined as follows: 
	$$
		\hat \bQ_n = \{  [ t 4^{-n}  ,  (t + 1) 4^{- n}] \times \Pi_{i = 1}^d [x^i 2^{-n}, (x^i + 1) 2^{-n}], t, x^1, \ldots, x^d \in \bZ\},   \, \, \text{if} \, \, n \leq 0, 
	$$
	 $$
		\hat \bQ_n = \bZ^{d+1}, n > 0.
	 $$ 
	It is easy to check $\{\hat \bQ_n, n \in \bZ\}$ is a filtration of partitions of $\bZ^{d+1}$ in the terminology of Section 3.1 of \cite{Kr_08}.

	Next, for any $(t, x) \in \bZ^{d+1}$ by $\hat Q_n (x)$ we denote the element of $\hat \bQ_n$ containing $(t, x)$.
	Denote
	$$
		f^{ \#}_{dyadic} (t, x) = \sup_{n \in \bN} |\hat Q_{n} (t, x)|^{-1} \sum_{ (m, y) \in \hat  Q_n (t, x)}  |f (m, y) - f_{\hat Q_n (t, x)}|.
	 $$
	Then,  by Theorem  3.2.10 of \cite{Kr_08}
	\begin{equation}
				\label{12.1}
		||f ||_{  l_p (\bZ^d) } \leq N  (p, d) ||f^{ \#}_{dyadic}||_{ l_p (\bZ^d) }. 
	\end{equation}
	Observe that, for any $n, t, x$, there exists $(\tilde t, \tilde x) \in \bR^{d+1}$ such that
	$$
		\hat Q_n (t, x)  \subset \{ (m, y) \in \bZ^{d+1}: |m  - \tilde t| \leq 4^{-n}, ||y  - \tilde x|| \leq 2^{-n}\}.  
	$$
	Clearly, the latter set is a subset of $Q_{ 2^{-n + 1} } (t, x)$.
	By what was just said for any $n < 0$
	$$
		|\hat Q_{n} (t, x)|^{-1}   \sum_{ (m, y) \in \hat  Q_n (t, x)}  |f (m, y) - f_{\hat Q_n (t, x) }|. 
	$$
	 $$
		 \leq |\hat Q_{n} (t, x)|^{-2}   \sum_{ (m_i, y_i) \in   Q_{ 2^{-n + 1} } (t, x), i  = 1, 2}  |f (m_1, y_1) - f (m_2, y_2)|
	$$
	 $$
		\leq 2  |\hat Q_{n} (t, x)|^{-2} |Q_{ 2^{-n + 1} }|   \sum_{ (m, y) \in   Q_{ 2^{-n + 1} }(t, x) }  |f (m, y) - f_{ Q_{2^{-n+1} (t, x)} }|
	 $$
	  $$
		\leq 2  |\hat Q_{n} (t, x)|^{-2} |Q_{ 2^{-n + 1} }|^2 f^{\#} (t, x) \leq N (d)    f^{\#} (t, x).
	  $$
	 Combining this with \eqref{12.1}, we prove Theorem \ref{theorem 12.2}.
	\end{proof}

	 \end{document}